\documentclass[12pt]{article}
\usepackage{amssymb}
\usepackage{amsfonts}
\usepackage{epsfig}
\usepackage{latexsym}
\usepackage{graphicx}
\usepackage{tikz}
\usetikzlibrary{calc,arrows}
\usepackage{scalefnt}
\usepackage{caption}
\usepackage{url}
\usepackage{amsmath}
\usepackage[margin=1in]{geometry}
\usepackage{verbatim}
\usepackage{parskip}
\usepackage{pgfplots}
\usepackage{float}
\usepackage{amsthm}
\usepackage{float}
\restylefloat{table}

\newcommand{\bbm}{\begin{boldmath}}
\newcommand{\ebm}{\end{boldmath}}

\newcommand{\lt}{<}
\newcommand{\gt}{>}
\newtheorem{theo}{Theorem}[section]
\newtheorem{cor}{Corollary}[section]
\newtheorem{lem}{Lemma}[section]
\newtheorem{prop}{Proposition}[section]

\newtheorem{defy}{Definition}[section]

\newtheorem{mex}{Example}[section]

\newtheorem{conjecture}{Conjecture}[section]
\newtheorem{remk}{{\bf Remark:}}[section]
   {\begin{remk}\begin{normalshape}\mbox{}}%
   {\hfill  \end{normalshape}\end{remk}}%

\tikzstyle{level 1}=[sibling distance=40mm]
\tikzstyle{level 2}=[sibling distance=30mm]
\tikzstyle{level 3}=[sibling distance=20mm]
\tikzstyle{level 4}=[sibling distance=10mm]
\newcommand{\smallskipneg}{\vspace*{-6mm}}

\newcommand{\vind}{\hspace*{7mm}}
\begin{document}

\begin{center}
{\Large {\bf Extending du Bois-Reymond's Infinitesimal and Infinitary Calculus Theory} } \\
\vspace{0.5cm}
{\bf Chelton D. Evans and William K. Pattinson}
\end{center}

\begin{abstract}
 The discovery of the infinite integer leads to
 a partition between finite and infinite numbers.
 Construction of an infinitesimal and infinitary number system,
 the Gossamer numbers.
 Du Bois-Reymond's much-greater-than relations and little-o/big-O defined with the Gossamer number system, and the relations algebra is explored.
 A comparison of function algebra is developed.
 A transfer principle more general than Non-Standard-Analysis is developed,
 hence a two-tiered system of calculus is described.
 Non-reversible arithmetic is proved, and found
 to be the key to this calculus and other theory.
 Finally sequences are partitioned between finite and infinite intervals.
\end{abstract}

\section*{Introduction}
 The papers are an interdisciplinary collaboration.
 The coauthor is a retired high school and tertiary Physics/Maths 
 teacher  with over 33 years teaching experience.
 I produced the mathematics, considered by myself for over 20 years,
 and collaborated with William,
 over the last four years.
 With the goal to make calculus more accessible.

 The theory is developed to support applications
 with comparison algebra and infinitesimals.
 However, 
 it soon emerged that much larger issues needed
 to be considered, as the real numbers do not contain
 infinities.  What we found was a need to
 fundamentally consider what numbers really are, 
 and to find and then
 use algebra to describe the `missing' mathematics.

 While constructing a paper trail
 to validate and prove the comparison theory (Part 3),
 which is used in applications \cite[Convergence sums ...]{cebp2}
 which we developed before the theory,
 we came across du Bois-Reymond's work
 via Hardy in `Orders of infinity' \cite{ordersofinfinity}.
 This had extensive symbolism and theory
 of the magnitude relations $\succ$ and $\succeq$
 which correspond respectively to little-o and big-O relations.

 It made sense to use this, and 
 explore another kind of calculus. 
 We had the classic problem, does the chicken
 come before or after the egg? The idea of
 magnitude arguments to solve limits
 is a natural consideration from the limit problem. 
 We believe the symbol relations
 greatly benefit calculus arguments, because
 they describe mathematics that is otherwise difficult
 to consider. 

 Indeed, the issues of calculus, not just its discovery,
 the very existence of infinitesimals and infinities
 has caused the greatest mathematical argument,
 arguably spanning millennia.
 The Greek's rejection of the infinitesimal, and their
 other mistakes, we have repeated.

 It is our opinion that each of the six parts are connected,
 but can also stand alone.
 The only dependencies in the Parts are that Part 1 and Part 2 use
 comparison algebra in their proofs.
 Hence, a forward reference for Parts 1 and 2 is used.
 In defense of this, what better way of justifying a theory
 than to use it, thereby demonstrating its utility.

 We are producing a mathematical
 foundation that will allow
 the `integration' of separate ideas in a consistent notation.

 This paper develops 
 the theory (the egg) where later
 papers develop our applications
 (the chicken).

 We strongly believe that infinitesimals and infinities must be
 included at every stage
 with infinitary calculus,
 because like real numbers,
 they can no longer be ignored.

\ref{S01}. Gossamer numbers \\
\vind \ref{S0101}. Introduction \\
\vind \ref{S0102}. The infinite integer \\
\vind \ref{S0103}. Preliminaries \\
\vind \ref{S0104}. Infinitesimals and infinities \\
\vind \ref{S0105}. Properties \\
\vind \ref{S0106}. Field Properties \\
\vind \ref{S0107}. Conclusion 

\ref{S02}. The much greater than relations \\
\vind \ref{S0201}. Introduction \\
\vind \ref{S0202}. Evaluation at a point \\
\vind \ref{S0203}. Infinitary calculus definitions \\
\vind \ref{S0204}. Scales of infinity \\
\vind \ref{S0206}. Little-o and big-O notation 

\ref{S03}. Comparing functions \\
\vind \ref{S0301}. Introduction \\
\vind \ref{S0302}. Solving for a relation \\
\vind \ref{S0303}. M-functions an extension of L-functions

\ref{S04}. The transfer principle \\
\vind \ref{S0401}. Introduction \\
\vind \ref{S0402}. Transference \\ 
\vind \ref{S0403}. Overview and prelude 

\ref{S05}. Non-reversible arithmetic and limits \\
\vind \ref{S0501}. Introduction \\
\vind \ref{S0502}. Non-reversible arithmetic \\
\vind \ref{S0503}. Logarithmic change \\
\vind \ref{S0504}. Limits at infinity

\ref{S06}. Sequences and calculus in $*G$ \\
\vind \ref{S0601}. Introduction \\
\vind \ref{S0602}. Sequences and functions \\
\vind \ref{S0603}. Convergence \\
\vind \ref{S0604}. Limits and continuity \\
\vind \ref{S0606}. Epsilon-delta proof \\
\vind \ref{S0605}. A two-tiered calculus \\
\vind \ref{S0607}. A variable reaching infinity before another 
\section{Gossamer numbers} \label{S01}
 The discovery of what we call the gossamer number system $*G$,
 as an extension of the real numbers includes an infinitesimal and
 infinitary number system;
 by using \mbox{`infinite integers'}, an isomorphic construction
 to the reals by solving algebraic
 equations is given.
 We believe this is a total ordered field.
 This could be an
 equivalent construction of the hyperreals. 
 The continuum is partitioned:
 $0 \lt \Phi^{+} \lt \mathbb{R}^{+} \! + \Phi \lt +\Phi^{-1} \lt \infty$.
 A one-one correspondence over an infinity of infinite intervals
 is interpreted differently resulting
 in the infinite-integers having
 a higher cardinality than the countable integers,
 and a likely consequence that $*G$ 
 and the hyperreals have a higher cardinality than the real numbers.
\subsection{Introduction} \label{S0101}
 An infinite integer is an integer at infinity, larger than any finite integer.
 The set of integers $\mathbb{J}$ 
 is a countable infinity. However,
 by representing this countability as an
 infinite integer $n|_{n=\infty}$,
 in Non-Standard-Analysis (NSA) $\omega$;
 using evaluation at a point notation Definition \ref{DEF074};
 if we take the reciprocal $\frac{1}{n}|_{n=\infty}$
 we have an infinite
 number that is positive and not rational, or even
 a real number, but a positive infinitesimal.

 This number system is of interest to
 anyone working with calculus.
 Infinities and infinitesimals in mathematics are universal.

 Despite this, 
 infinitesimals and infinities largely are not recognised 
 as actual numbers.
 That is, they are not declared, like integers or reals.
 Limits are used all the time, and these ``are" infinitesimals,
 that what we daily use as shorthand is not acknowledged.
 Infinitesimals are so successful
 that they are accepted
 as known and are applied so often.
 (Since we use them, why is there any need to study them)

 We have gone back in time to a mathematician who
 systematically considered infinitesimals and infinities.
 Paul du Bois-Reymond, like everyone else,
 was at first uncomfortable with infinity.
 By comparing functions (Example \ref{MEX019}),
 he investigated
 the continuum 
 and realized
 the space which we would call Non-Standard-Analysis today.
 Our subsequent sections in this series consider this in some detail.

 The purpose of this paper is to construct a number system that
 both
 contains infinitesimals and infinities, and a more natural
 extension to the real numbers than other alternatives. 

 A fact of significant importance is that 
 the real numbers are composed of integers.
 Rational numbers are built from the ratio of two integers,
 algebraic numbers from solving an equation of rational coefficients,
 and so other numbers with infinite processes are constructed.

 With the discovery of infinite integers
 \cite[p.1]{victors}
 and A. Robinson's enlargement of the integers $\mathbb{J}$: *J 
 \cite[p.97]{abraham}),
 by a similar process to the construction
 of the real numbers from integers we can
 construct
 infinite and infinitesimal numbers.
 We introduce
 a number system which is comparable to, and an
 alternative to hyperreal and surreal numbers.

By having `infinity' itself as a number,
 like $i$ for complex numbers,
 we can follow a standard construction analogous
 to building the real numbers from integers.

 We believe that no such simple construction
 has been given that provides the usability
 without the complexity. 
 (i.e. avoids the complicated mathematics of
 logic and set theory)

 With time, we would like to see the claimed misuse of infinitesimals
 as `not being rigorous', start to be
 reconsidered.
 A number system which directly supports
 their use in the traditional sense is possible. 
 In Part 4 and Part 5
 we argue that rigour
 should not be the only goal of a number system.

 We take a constructionist's approach
 to the development of the
 infinitesimal and infinitary number system.
 A series
  of papers
 for a larger theory, 
 based on du Bois-Reymond's ideals and
 our extensions to them,
 is developed.

 Working with infinitesimals 
 and infinities,
 we find a need for better representation.
 Indeed, the more we consider calculus,
 the greater is the need for this.

 By adding infinities and infinitesimals
 to the real numbers,
 and declaring these additions 
 to be of these types,
 this primarily is an extension of
 the real number line. Rather than saying
 a number becomes infinitely small, or
 infinitely large, 
 the numbers are able to be declared
 as infinitesimals and infinities
 respectively.
 Theorems with this characterisation naturally
 follow. 

 Why should we be interested
 in infinitesimals or infinities? 
 We are indeed,
 using them
 all the time
 through limits, applied mathematics
 and theory.
 Why should we want to further 
 characterise them?

 In Flatland \cite{flatland}, set in a world of two dimensions,
 where a 3D creature, a sphere, attempting to 
 communicate, describes its world to
 a flatlander inhabitant, the square.
 Proving its existence by
 removing an item from a closed 2D cupboard, 
 and materializing it elsewhere.

If we imagine the infinitesimals and infinities
 as working in a higher dimension, 
 then we project back to the reals
 afterwards. See Part 4 The transfer principle.

 The name we give  to
 the number system
 which we construct 
 is the ``gossamer numbers",
 for, analogous to a spider's web, there
 can be ever finer strands
 between the real numbers. 
 About any real number,
 there are infinitely
 many numbers asymptotic to
 a given real number.
 Similarly about any curve, there
 are infinitely many curves asymptotic
 to it.

 What this view gives is a 
 deeper  
 fabric of space to work with. 
 Propositions in $\mathbb{R}$ 
 may be better explained as propositions
 in the higher dimensional space.
  
 We look at the problem from
 an equation's perspective.
 Just as complex numbers solve 
 the equation
 $x^{2} = -y^{2}$ in terms of $x$, 
 $y = x i$,
 we can then put forward an equation
 at infinity which likewise 
 needs new ideas and new numbers.

 Historically, the development of a number system preceded with an equation or problem 
 which needed to be solved. By looking at ancient mathematics such as Diophantine equations,
 words were used in place of a mathematical language such as our modern algebra.

 Without the language of algebra, it took a genius to solve by today's
 measure routine problems.
 Similarly with Newton's calculus, the English mathematicians lagged behind
 the continent until the adaption of Leibniz's calculus.
 In short, the mathematical language can greatly impact on both the development and use
 of mathematics affecting the culture.
 
 While it can be argued that infinitesimals and infinities are known and
 managed with other mathematics such as asymptotic notation
 and the hyperreal number system,
 by producing two fields of mathematics,
 we argue that there is a need for change to include new mathematical language.  
 Similarly,
 the application of hyperreal numbers justified their existence.
\subsection{The infinite integer} \label{S0102}
 We believe that
 what an atom is to a physicist,
 an infinite integer
 is to a mathematician, well
 that is how it should be!
 With this belief,
 we go on to derive infinitesimals
 and infinities, building
 on the existence of the
 infinite integer $\mathbb{J}_{\infty}$
 (Definition \ref{DEF052}).

 An infinite integer's existence does not contradict Euclid's proof that there is no largest integer
 because the existence of infinite integers
 is not a single largest integer, but a class of
 integers.

 Certain properties become apparent.
 We define a partition between finite
 and infinite numbers,
 but there is no lowest or highest infinite number,
 as there is no lowest or highest infinite integer.
\[ (\ldots, n-2, n-1, n, n+1, n+2, \ldots)|_{n \in \mathbb{J}_{\infty}} \]
 In the construction of the extended numbers, what distinguishes these new numbers
 is the inclusion of an infinity within the number. While the real numbers are constructed
 created by infinite processes,
  \textit{the real numbers do not contain an infinity}. 
 We can imagine such a number with
 the inclusion of an infinity as ever changing, and not static. 
 We need such a construction to explain limits. 

 That we are continually being asked to explain the existence of real
 numbers, let alone numbers with infinities is a challenge. However,
 these numbers provide immeasurable insight into the mechanics of calculus.
 Would we deny the existence of quantum physics because
 we felt only the need for classical physics?
 This is similar to the denial of 
 the infinitesimal and infinitary numbers because we already have the real numbers.

 We are accustomed to the use of complex 
 numbers which have a separate component,
 the imaginary part.
 Similarly with the extension of the
 real number, an infinitesimal or infinitary  
 component is added.
 However, this will generally be seen as part of the
 number, and not represented
 as separate components. Having
 said that, in this paper we do represent
 the components separately,
 to help prove the basic properties
 of the number system.

 Once the infinite integer is accepted,
 an infinity not in $\mathbb{R}$,
 the infinitesimal follows as its reciprocal
 $\frac{1}{n} \not \in \mathbb{R}$.
 
 While the set of integers is a countable infinity,
 we need to distinguish between infinite and finite
 numbers.
\bigskip
\begin{defy}
 We say $\mathbb{J}_{\lt}$ to mean a finite integer,
 $k \in \mathbb{J}$ and 
 $k$ is finite.  $\mathbb{J}_{\lt} \subset \mathbb{J}$
 We similarly define
 $\mathbb{N}_{\lt} \subset \mathbb{N}$ the finite natural numbers,
 $\mathbb{Q}_{\lt} \subset \mathbb{Q}$ the finite rational numbers,
 $\mathbb{A}_{\lt} \subset \mathbb{A}$ the finite algebraic numbers.
\end{defy}
\bigskip
\begin{defy}\label{DEF004}
 We say $\mathbb{J}_{\lt} \lt \mathbb{J}_{\infty}$ to mean
 than for all finite integers $i \in \mathbb{J}$,
 and all `infinite integers' $j \in \mathbb{J}_{\infty}$
 then 
 $i \lt j$.
\end{defy}

 We note that the set of finite numbers is infinite, but any given
 finite number is not infinite.
 We can consider any
 finite integer as large as we please, but once we do it is in existence
 it is without infinity.  
\subsection{Preliminaries} \label{S0103}
 This paper is one of six in
 this series, which we believe is establishing a new field of mathematics.
 The proofs and notation contain mathematics in 
 Part 2 and Part 3;
 however if $f \succ g$ then $\frac{f}{g} \in \Phi$
  is equivalent to $f = o(g)$, where $\Phi$ is
 an infinitesimal (Definition \ref{DEF036}).
 The proofs can solve for a relation,
 which is studied in Part 3.

 We introduce
 a notation at infinity
 which can express both a limit
 and the realization of being
 at infinity.
 We later develop and justify
 this choice more
 fully 
 Part 2, Part 4, Part 5.
\bigskip
\begin{defy}\label{DEF074}
 $f(x)|_{x=\infty}$ 
 represents the notion of the ``order" of
 $f(x)$ as a function ``at infinity", or a limit at infinity.
 See Evaluation at-a-point Definition \ref{DEF001}.
\end{defy}
\bigskip
\begin{remk}
 We will later see that 
 formal objects can be formed with a number system
 that reflects the requirements of
 our notion of infinitesimals and infinities.
\end{remk}

 Now we give an example of an infinitary
 equation, which can
 only be solved by
 an infinitary number.

 Solve $(x^{2}+1)y|_{x=\infty} = x^{3}|_{x=\infty}$.
 $y = \frac{x^{3}}{x^{2}+1}|_{x=\infty}$ is a solution,
 another is $y=x|_{x=\infty}$.

 We already use infinitesimals
 and infinities with little-o and 
 big-O notation. However, unlike real
 numbers, we often do not declare them.
 And this may be a problem, because,
 unlike real numbers, they can
 have radically different properties.

So here is the contradiction,
 you use these numbers ubiquitously,
 for example, calculating limits,
 but you do not declare them as such.

Now we are not entirely adverse to
 the utility of this approach,
 things need to be done and calculations
 performed.

 However, not even relatively modern
 mathematicians
 such as Hardy are immune.
 While having extensive powers of
 computation, he never defined
 an infinitesimal as a number.

 Examples of infinitesimal and infinitary numbers are everywhere.

 Zeno's paradoxes \cite{historyzeno} assume a
 truth and argue to a paradox. Hence disproving 
 the assumed truth. 
 In assuming continuity as an infinity of divisibility
 for physical questions of motion, logical
 contradictions follow.
 Achilles, when chasing the tortoise, never catches up,
 as, when Achilles travels towards the tortoise,
 there is always
 some remaining distance to reach it \cite{historyzenoachilles}.

 The distance left over is of course a positive infinitesimal,
 a number smaller than any
 positive real numbers \cite[Section 2: What does a sum at infinity mean?]{cebp2}

 From \cite{history},
 to counter Zeno's arguments, Atomism, while ascribed to
 physical theory, the indivisibility of matter and presumably
 time. However, Aristotle believed in the continuity of magnitude.
\begin{quote}
Aristotle identifies continuity and discreteness as attributes
 applying to the category of Quantity.
 As examples of continuous quantities, or continua, he offers lines,
 planes, solids (i.e. solid bodies), extensions, movement,
 time and space; among discrete quantities he includes number
 and speech
 \cite{history}
\end{quote}
 
 Infinities and infinitesimals are crucial 
 in the development of calculus,
 before their ostracism (19th century) from mainstream calculus.
 We believe this is a case of
 `throwing the baby out with the bath water'.  

Hence, the banishment of infinitesimals
 was a rejection 
 which sent them underground.
 The Greeks, with their geometric
 methods, had followed a similar story.
 Where they could not describe numbers,
 as this had to wait for
 the discovery of zero and the
 positional number system,
 reasoning was replaced by other arguments
 (Egyptian fractions for calculations
 and Roman numerals for date calculations remained).

In modern times,
 the first 
 person to systematically study the infinitesimals and 
 infinities 
 was
 du Bois-Reymond, whose writing
 around 1870 onwards, on
 the scales of infinities \cite[pp.9--21]{ordersofinfinity},
 for example 
 $(\ldots, x^{-2}, x^{-1}, 1, x, x^{2}, x^{3}, \ldots)|_{x=\infty}$,
 and the ratios of infinities,
 which are instances of comparing functions.

\begin{quote}
 ... certain problems have "familiarized mathematicians with
the use of scales of comparison 
 other than those of powers of a variable
.. This
extension goes back above all to the works of P. du Bois-Reymond who was the
first to approach systematically the problems of the comparison of functions in
the neighborhood of a point, ...
 \cite[Bourbaki p.157]{fisher}
\end{quote}

 Cantor had a competing
 theory 
 which represented the continuum with sets and he believed
 infinitesimals did not fit in.
 So our development may have been
 skewed  by the rise of
 set theory, which became dominant
 and is heavily present in Abraham's NSA. 

 From the work of Abraham Robinson,
 infinitesimals have in recent years been made more rigorous;
 however they have not been made accessible, that is, easy to use.
 To address the inaccessibility, 
 reformations of NSA have been constructed.
 However, NSA by its nature is 
 very technical and used for high-end mathematics.

 Note that the gossamer number system is also technical.
 However, the gossamer number system we present is 
 more accessible. It has been built to be used
 with functions.  We would have no idea how to
 implement many of the applications that later follow
 such as the rearrangement theorems with NSA, 
 or  
 the theory which is later developed.
 This is not a trivial distinction: the tools that
 a worker uses to do a job matter. We may not wish to
 deal with low level set theory logic while working
 with functions, particularly if it is unnecessary.

 Our concerns are NSA's debasement of meaning.
 That is, the meaning and use of infinitesimals
 and infinities is lost or ``muddied" during applications
 and proofs. 

This is not to invalidate NSA. Indeed,
 later we do use NSA which as a reference is
 invaluable. Just as we have axiomatic geometry, 
 does not mean that we have to reason with it.
 NSA is specialized, and requires more knowledge
 to use it, if you want to use it in the first place.

 Instead, we will look where possible
 for an alternative.
 From the premise of this and later papers,  
 it follows that Robinson's NSA is not the only
 way. 

  The paper has two primary tasks;
 describe the construction of the number system, 
 and then work towards proving $*G$ is an ordered
 field.
 We also believe others can forward this
 work with
 the benefits being that the number system
 can be used in many ways.

 These are separate goals, the justification of introducing
 additional complexities as we have separated the
 components.  There is likely other ways to go about this.

 This is not the same as using this mathematics. 
 The real numbers themselves are used in so many different
 ways,
 that we expect a generalisation of them to be even more diverse.
\subsection{Infinitesimals and infinities} \label{S0104}
A more general way of considering infinity is
 the realization of reaching infinity.

Infinitesimals and infinities
 naturally occur in any description
 at infinity. E.g. an infinitesimal, $\frac{1}{n}|_{n=\infty}$. 
 The inverse
 is an infinity. $1/(1/n) = n|_{n=\infty} = \infty$.
 It would be fair to say that
 calculus (for the continuous variable) without infinitesimals and
 infinities would not exist
 (for example, a limit or a derivative could not exist).

 If $x$ is infinite 
 and we realize $x$ to infinity,
  then at infinity $x$ becomes
 an extended real. 
 If $x$ is an \textit{infinite integer}, then it
 is an integer at infinity.
 Considering Robinson's NSA,
 it becomes clear that
 with infinite integer $\omega$,
 we can have an infinity of
 infinite integers $( \omega, \omega+1, \omega+2, \ldots)$. 
 In this context, infinity is its own 
 number system,
 where we have arrived at infinity, 
 and it is a very large space.
 A lower case $n$ at infinity
 will be understood as an infinite integer,
 the same as $\omega$. (Though with the 
 notation any variable can be used.)
 With infinitesimals the situation is similar, since given
 an infinity we can always construct an infinitesimal by dividing
 $1$ by the infinity.

 We can similarly construct \textit{infinite rational numbers}.
 The numbers themselves need not all be
 infinite, but a composition of infinities and reals.
 E.g. $\frac{n^{2}+1}{n^{3}-2}|_{n=\infty}$.
 Similarly, there are \textit{infinite surds}. E.g. $\sqrt{2}n|_{n=\infty}$. 
 What about \textit{infinite reals}? That is, a number which is 
 as dense as a real, but at infinity.
 The infinite numbers will in many respects behave similarly to 
 their finite counterparts.

Meaning can be attributed
 to expressions at infinity.
 We may consider $n|_{n=\infty}$
 as the process of repeatedly adding $1$.
 Similarly $\mathrm{ln}\,n|_{n=\infty}$
 corresponds with summing the harmonic series.
  If divergent series, which are ubiquitous, are
 asymptotic to divergent functions, then the functions can
 have a geometric meaning.
 Another example, $\mathrm{sin}\,x|_{x=\infty}$ continually generates
 a sine curve at infinity.

The next jump is that from the reals
 to du Bois-Reymond's infinitesimals and 
 infinities at infinity,
 realizing the space at infinity.
 To do this we separate 
 finite and infinite numbers, therefore
 separating finite and infinite space. 
 For example, integers $\mathbb{J}_{\lt}$ and
 infinite integers $\mathbb{J}_{\infty}$ at infinity.
 $\mathbb{J}_{\lt} \lt \mathbb{J}_{\infty}$ Definition \ref{DEF004}.

Infinities and infinitesimals form
 their own number system.
 As a number, extending the reals with infinities, the
 infinities are a supremum for the reals, with an extended Dedekind cut, as an infinity is larger than any finite number.
\bigskip
\begin{defy}\label{DEF051}
Define a positive infinite number to be ``larger" than any number in $\mathbb{R}$
\end{defy}
\bigskip
\begin{defy}\label{DEF060}
Define a negative infinite number to be ``smaller" than any number in $\mathbb{R}$
\end{defy}
\bigskip
\begin{defy}\label{DEF059}
Define an infinite number,
 `an infinity' as
 a positive infinity or negative infinity,
 the set of all these being denoted $\Phi^{-1}$.
 However exclude 
 $\pm\infty \notin \Phi^{-1}$ for reversible multiplication.
\end{defy}
\bigskip
\begin{defy}\label{DEF036} 
We say 
 a number is an ``infinitesimal" $\Phi$ 
 if the reciprocal of the number is an infinity.
 If $\frac{1}{x} \in \Phi^{-1}$ then $x \in \Phi$.
\end{defy}
\bigskip
\begin{defy}\label{DEF078}
 If $x$ is a positive infinitesimal
 then
 $x \in \!+\Phi$ or $x \in \Phi^{+}$. 
\end{defy}
\bigskip
\begin{defy}
 If $x$ is a negative infinitesimal
 then
 $x \in \!-\Phi$ or $x \in \Phi^{-}$. 
\end{defy}
\bigskip
\begin{defy}
 If $x$ is a positive infinity or negative infinity then
 $x \in +\Phi^{-1}$ or $-\Phi^{-1}$ respectively.
\end{defy}
\bigskip
\begin{defy}\label{DEF072}
Let $\frac{1}{0} = \infty$, $0$ and $\infty$ are mutual inverses. 
\end{defy}
\bigskip
\begin{cor}\label{P043}
$0$ is not an infinitesimal.
\end{cor}
\begin{proof}
$\frac{1}{0} = \infty \notin \Phi^{-1}$.
 $\infty$ is `infinity', not `an' infinity (see Definition \ref{DEF059}).
\end{proof}
 $0$ is finite, but not an infinitesimal.
 $0$ is a special number, which could be called a super infinitesimal.
 The reason for not including $0$ as an infinitesimal
 is to make reasoning clearer by avoiding division by zero,
 having $R\backslash\{0\} \cup R_{\infty}$ (see Definition \ref{DEF058}) with respect to multiplication
 form an abelian group.
\bigskip
\begin{prop}\label{P018}
 A infinitesimal is less than
 any finite positive number.
\end{prop}
\begin{proof}
 Solving for a comparison relation Part 3.
 A negative number is less than a positive number, then
 only the positive infinitesimal case remains.
 $x \in \mathbb{R}^{+}$; $\delta \in +\Phi$;
 then $\frac{1}{\delta} \in +\Phi^{-1}$,
 compare the infinitesimal and real number,
 $\delta \; z \; x$, 
 $1 \; z \; \frac{x}{\delta}$,
 $\frac{1}{x} \; z \; \frac{1}{\delta}$,
 $\mathbb{R}^{+} \; z \; +\!\Phi^{-1}$, 
 by Definition \ref{DEF051} $z = \; \lt$ then $\delta \lt x$,
 $\Phi \lt \mathbb{R}^{+}$.
\end{proof}
\bigskip
\begin{defy}
 We say 
$\mathbb{R}_{\infty}$ is an `infinireal' number
 if the number is an infinitesimal or an infinity.
If $x \in \mathbb{R}_{\infty}$ then either $x \lt \mathbb{R}^{-}$
 or $x \gt \mathbb{R}^{+}$.
\[ x \in \mathbb{R}_{\infty} \text{ then } x \in \Phi \text{ or } x \in \Phi^{-1} \text{,  } \mathbb{R}_{\infty} = \Phi \cup \Phi^{-1} \]
\end{defy}
\begin{defy}
 $\overline{\mathrm{R}}_{\infty} = \{-\infty, 0, \infty \}$ $\;\;$ (A realization of $\mathbb{R}_{\infty}$)
\end{defy}
\bigskip
\begin{defy}
If $a \not\in b$ then no element in $a$ is in set $b$. Equivalent to $\{ b \}\backslash \{ a \}$ or
 $\{ b \} - \{ a \}$.
\end{defy}

 For example, 
 $\Phi \notin f$ means the variable or function contains
 no infinitesimals. Similarly $\Phi^{-1} \notin f$ means
 $f$ contains no infinities. 
\[ \not \in \mathbb{R}_{\infty} \equiv \mathbb{R} + \Phi [ \mathbb{R} \neq 0] \]
\begin{defy}\label{DEF057}
 Define the numbers $0$ and $\infty$: $0 \lt |\Phi|$ and $|\Phi^{-1}| \lt \infty$.
\end{defy}
\bigskip
\begin{theo}
 Partitioning the number line, $0 \lt \Phi^{+} \lt \mathbb{R}^{+} \cup \Phi \lt +\Phi^{-1} \lt \infty$.
\end{theo}
\begin{proof}
$\delta_{1}, \delta_{2}, \delta_{3} \in \Phi$;
 $x \in \mathbb{R}^{+}$;
 $z \in \mathbb{B}$ a binary relation.
 Consider
  $\delta_{1} \; z \; x + \delta_{2}$, 
 $\delta_{1} - \delta_{2} \; z \; x$.
 Since $(\Phi,+)$ is closed, let $\delta_{3} = \delta_{1}-\delta_{2}$.
 $\delta_{3} \; z \; x$.
 By $\Phi \lt \mathbb{R}^{+}$ (Proposition \ref{P018})
 $z = \; \lt$, since adding and subtracting on both sides
 does not change the inequality.
\end{proof}

 Given the existence of the infinite
 integers,
 a constructive definition $*G$ Definition \ref{DEF058} of
 an extended real number system follows.
 The number system at infinity is
 isomorphic with the real number system, defining integers,
 rational numbers, algebraic numbers,
 irrational numbers and transcendental numbers all at infinity.
\bigskip
\begin{defy}\label{DEF061}
Let $\mathbb{A}_{\lt}$ be the symbol for
 the finite algebraic numbers.
\end{defy}
\bigskip
\begin{defy}
Let $\mathbb{A}_{\lt}'$ be the symbol for
 the finite transcendental numbers.
\end{defy}
\bigskip
\begin{defy}\label{DEF052}
 Define an ``infinite integer" $\mathbb{J}_{\infty}$ at infinity, larger in magnitude than any finite integer. 
\end{defy}
\bigskip
\begin{defy}\label{DEF071}
 Define an ``infinite natural number" $\mathbb{N}_{\infty}$ as a positive infinite integer.
\end{defy}
\bigskip
\begin{defy}\label{DEF053}
 Define an ``infinite rational number" $\mathbb{Q}_{\infty}$ 
 as a ratio of finite integers and  
 ``infinite integers". 
\end{defy}
\smallskipneg
\smallskipneg
\[ \{ \frac{ \mathbb{J}_{\infty}}{ \mathbb{J}_{\infty}}, \frac{\mathbb{J}_{\lt}}{\mathbb{J}_{\infty} }, \frac{ \mathbb{J}_{\infty} }{ \mathbb{J}_{\lt} } \} \in \mathbb{Q}_{\infty} \] 
\begin{defy}\label{DEF054}
 Define an ``infinite algebraic number" $\mathbb{A}_{\infty}$ 
 which is the root of a non-zero finite polynomial  
 in one variable with at least one infinite rational number coefficient.
\end{defy}
\bigskip
\begin{defy}\label{DEF055}
 Define an ``infinite irrational number" $\mathbb{Q'}_{\infty}$
 as an infinity that is not an infinite rational number.
\end{defy}
\bigskip
\begin{defy}\label{DEF056}
 Define an ``infinite transcendental number" $\mathbb{A'}_{\infty}$ as an infinity
 which is not an infinite algebraic number.
\end{defy}
\bigskip
\begin{defy}\label{DEF058}
 Define the gossamer numbers, $*G$
 as numbers that comprise
 of Definitions \ref{DEF052}--\ref{DEF056}. 
\end{defy}
\bigskip
\begin{mex}
 Given a fraction of the form $\frac{\mathbb{J}_{\infty}}{\mathbb{J}_{\infty}}$,
 $\frac{2(n+1)}{n+1}|_{n=\infty}=\frac{2}{1}$
 cancelling like $\mathbb{J}_{\infty}$ terms
 leaves a fraction of the form $\frac{\mathbb{J}_{\lt}}{\mathbb{J}_{\lt}} \in \mathbb{J}_{\lt}$. 
\end{mex}

 As a byproduct of the $*G$ construction,
 $\mathbb{R}$ is embedded 
 within Definition \ref{DEF058}, as
 there exist 
 $\frac{\mathbb{J}_{\infty}}{\mathbb{J}_{\infty}} \in \mathbb{J}_{\lt}$.
 We could choose to define $\mathbb{Q}_{\infty}$
 to exclude $\mathbb{Q}_{\lt}$. Then all 
 $\{ \mathbb{J}_{\infty},
 \mathbb{Q}_{\infty},
 \mathbb{Q}_{\infty}',
 \mathbb{A}_{\infty},
 \mathbb{A}_{\infty}'
 \}$ would contain an explicit infinity within
 the number.
 By definition
 all infinireals $\mathbb{R}_{\infty}$ contain
 an infinity as an element within
 the number. 
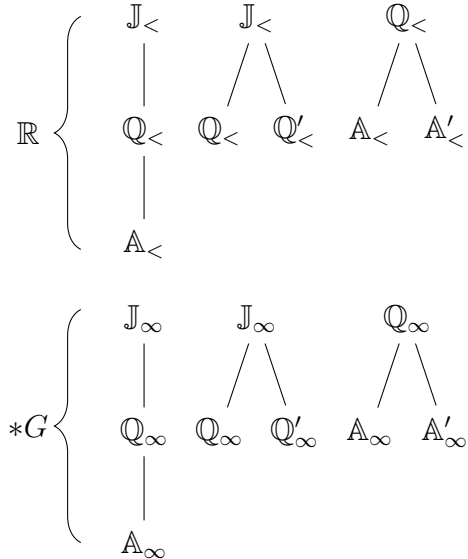
\begin{figure}[H]
\centering
\begin{tikzpicture}[
level 1/.style={sibling distance=10mm}, 
level 2/.style={sibling distance=10mm} 
]
\begin{scope}[ xshift=15mm, yshift=40mm]
\node [] (a11){$\mathbb{J}_{\lt}$}
  child 
  {
    node [] (a00) {$\mathbb{Q}_{\lt}$ } 
  }
  child 
  {
    node [] (a01) {$\mathbb{Q}_{\lt}'$ } 
  }
;
\end{scope}
\begin{scope}[xshift=35mm, yshift=40mm]
\node [] (a15){$\mathbb{Q}_{\lt}$}
  child 
  {
    node [] (a06) {$\mathbb{A}_{\lt}$ } 
  }
  child 
  {
    node [] (a07) {$\mathbb{A}_{\lt}'$ } 
  }
;
\end{scope}
\begin{scope}[xshift=0mm, yshift=40mm]
\node [] (a08){$\mathbb{J}_{\lt}$}
  child 
  {
    node [] (a09) {$\mathbb{Q}_{\lt}$ } 
    child 
    {
      node [] (a10) {$\mathbb{A}_{\lt}$ } 
    }
  }
;
\end{scope}
\begin{scope}[xshift=15mm]
\node [] (a02){$\mathbb{J}_{\infty}$}
  child 
  {
    node [] (a03) {$\mathbb{Q}_{\infty}$ } 
  }
  child 
  {
    node [] (a04) {$\mathbb{Q}_{\infty}'$ } 
  }
;
\end{scope}
\begin{scope}[xshift=35mm, yshift=0mm]
\node [] (a15){$\mathbb{Q}_{\infty}$}
  child 
  {
    node [] (a06) {$\mathbb{A}_{\infty}$ } 
  }
  child 
  {
    node [] (a07) {$\mathbb{A}_{\infty}'$ } 
  }
;
\end{scope}
\begin{scope}[xshift=0mm, yshift=0mm]
\node [] (a12){$\mathbb{J_{\infty}}$}
  child 
  {
    node [] (a13) {$\mathbb{Q}_{\infty}$ } 
    child 
    {
      node [] (a14) {$\mathbb{A}_{\infty}$ } 
    }
  }
;
\end{scope}
\draw [decorate,decoration={brace,amplitude=10pt},xshift=-4pt,yshift=15pt] (-0.7,0.4) -- (-0.7,3.5) node [black,midway,xshift=-0.7cm] {$\mathbb{R}$};
\draw [decorate,decoration={brace,amplitude=10pt},xshift=-4pt,yshift=15pt] (-0.7,-3.5) -- (-0.7,-0.4) node [black,midway,xshift=-0.7cm] {$*G$};
\end{tikzpicture}
\caption{real and gossamer number composition} \label{fig:F07}
\end{figure}
\bigskip
A hierarchy diagram for the number
 systems
 shows
 the parent number system which is used to build
 the child number system (see Figure \ref{fig:F07}). In this way, all reals
 are composed of integers
 and gossamer numbers are composed of integers and infinite integers. 
 The reals are embedded within the gossamer numbers.
\bigskip
\begin{remk}
 We have constructed the real numbers not with $\mathbb{J}$
 but $\mathbb{J}_{\lt}$ as the real numbers cannot contain infinities. 
 In this way, we have also provided a construction
 of the real numbers. Restating with an explicit construction:
 Define a rational number $\mathbb{Q}_{\lt}$
 as a ratio of $\mathbb{J}_{\lt}$ 
 without division by zero.
 Define an irrational number $\mathbb{Q}_{\lt}'$
 which is not a rational number and finite.
 Define the finite algebraic number $\mathbb{A}_{\lt}$ 
 as a root of a non-zero polynomial in one variable with
 at least one rational coefficient $\mathbb{Q}_{\lt}$ and finite.
 Define a transcendental number
 $\mathbb{A}_{\lt}'$ which is not algebraic $\mathbb{A}_{\lt}$
 and finite.
\end{remk}
\bigskip
\begin{remk}
 While a finite number is not an infinity,
 the collection of finite numbers is an infinity,
 as there is no greatest finite number.
 Hence, the countability of the finite numbers and
 the finite numbers is separated.
 This apparent paradox is explained by any instance of a finite
 number begin less than the whole.
 (infinity is non-unique and has other possibilities)
\end{remk}
\bigskip
\begin{defy}\label{DEF048} 
We define
 extended gossamer numbers $*\overline{G}$,
 where $*\overline{G} = *G \cup \pm\infty$ 
\end{defy}

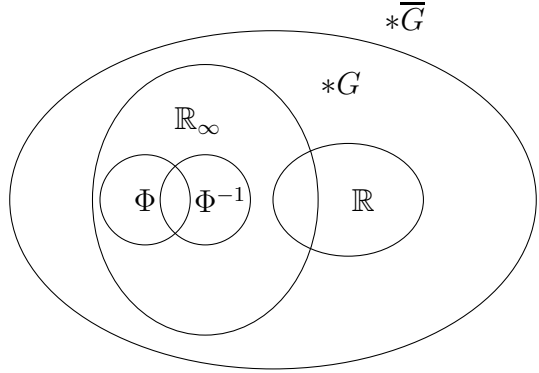
\begin{figure}[H]
\centering
\begin{tikzpicture}
  \draw (1,0) ellipse (1.0cm and 0.75cm);
  \draw (-0.9,0) ellipse (1.5cm and 1.80cm);
  \draw (0,0) ellipse (3.5cm and 2.25cm);
  \node at (-1.0, 1.05) (n001() {$\mathbb{R}_{\infty}$};
  \node at (0.9, 1.55) (n002() {$*G$};
  \node at (1.75, 2.45) (n003() {$*\overline{G}$};

\begin{scope}[xshift=-1cm];
  \draw (-0.7,0) circle (6mm);
  \node at (-0.7, 0.0) (n006() {$\Phi$};
\end{scope}

\begin{scope}[xshift=-0.2cm];
  \draw (-0.7,0) circle (6mm);
  \node at (-0.5, 0.0) (n007() {$\Phi^{-1}$};
\end{scope}
  \node at (1.2, 0.0) (n004() {$\mathbb{R}$};
\end{tikzpicture}
\caption{gossamer numbers and infinireals composition (Not a Venn diagram)} \label{fig:F11}
\end{figure}

The infinitary component generally is not unique.
 Where the intersection in Figure \ref{fig:F11}
 represents adding different components,
 between the numbers there exist overlaps.
 Infinity dominates, with any combination of non-infinity numbers with
 infinity resulting in an infinity.
\bigskip
\begin{defy}\label{DEF073}
Let $\Phi_{*}^{-1}$  be a number type, the infinireal infinity without reals or infinitesimals, under addition in the base.
 If $x \in \Phi^{-1} \backslash\{ \Phi, \mathbb{R} \}$
 then 
 $x \in \Phi_{*}^{-1}$.
\[ \Phi_{*}^{-1} = \Phi^{-1}-\mathbb{R} - \Phi \]
\end{defy}

 The following sets are disjoint and partition $\Phi^{-1}$.
\[ \{\Phi_{*}^{-1}\! + \Phi \} \cup \{ \Phi_{*}^{-1} \! + \mathbb{R}\backslash\{ 0 \} \} \cup \{ \Phi_{*}^{-1} \! + \mathbb{R}\backslash\{0\} + \Phi \} \cup \{ \Phi_{*}^{-1} \} = \Phi^{-1} \]
\bigskip
\begin{mex}
Example of numbers and their relation to
 Figure \ref{fig:F11}.
 $n + \frac{1}{n}|_{n=\infty} \in \Phi^{-1}\!+\Phi$ has an 
 infinitary and infinitesimal component;
 $\sqrt{2}+\frac{1}{n} \in \mathbb{R} + \Phi$ has both a real
 and an infinitesimal component.
\end{mex}

If we compare $\mathbb{R}_{\infty}$
 and $\mathbb{R}$,
 $\mathbb{R}_{\infty}$ have an infinity
 within the number itself, an infinite variable,
 and the reals do not, hence the reals are constant. 
 We could then conceive of the numbers
 $\mathbb{R}_{\infty}$ and $\mathbb{R}+\Phi$
 as ever changing and containing an infinity, 
 something of which the reals are not allowed.

 Having gossamer numbers,
 it would logically follow to define 
 complex gossamer numbers ($a, b \in*G;\, a+bi$).
 Given how useful complex numbers are, this would enable
 complex numbers with infinitesimals and infinities.

 Converting
 a given element of $*G$
 into an algebraic form can help
 identify what type the number is.
 If this is not possible,
 and the number is not of the other simpler types,
 then the number is generally an infinite transcendental number.
\bigskip
\begin{mex}
 Let $k \in \mathbb{J}_{\lt}$, $k = \frac{k}{1} \in \mathbb{Q}_{\lt}$.
 Let $n \in \mathbb{J}_{\infty}$, 
 $n = \frac{n}{1} \in \mathbb{Q}_{\infty}$.
\end{mex}
\bigskip
\begin{mex}\label{MEX050} 
 $n|_{n=\infty}$ is an integer at infinity in $\Phi^{-1}$.
 $n \in \mathbb{J}_{\infty}$. 
 We can compare infinite integers with each other
 in the same way we can compare integers with other integers.
 $n \lt n+1 \lt n+2 \lt \ldots|_{n=\infty}$
\end{mex}
\bigskip
\begin{mex}
$n^{2} + \frac{1}{n}|_{n=\infty} \in \mathbb{Q}_{\infty}$
 because $n^{2} + \frac{1}{n} = \frac{n^{3}+1}{n}|_{n=\infty}$ of form $\frac{\mathbb{J}_{\infty}}{\mathbb{J}_{\infty}} \in \mathbb{Q}_{\infty}$.
\end{mex}
\bigskip
\begin{mex}\label{MEX051} 
 $y^{2} + 2y + 1 = 2n^{2}|_{n=\infty}$
 is of form $\mathbb{J}_{\lt}\,y^{2} + \mathbb{J}_{\lt}\,y + \mathbb{J}_{\lt} = \mathbb{J}_{\infty}$,
 $\mathbb{Q}_{\lt}\,y^{2} + \mathbb{Q}_{\lt}\,y + \mathbb{Q}_{\lt} = \mathbb{Q}_{\infty}$  is in algebraic form, 
 $y \in \mathbb{A}_{\infty}$.
\end{mex}
\bigskip
\begin{mex}\label{MEX055}
Show $\sqrt{2}n|_{n=\infty} \in \mathbb{A}_{\infty}$.
 Let $y = 2^{\frac{1}{2}}n$,
 $y^{2} = 2n^{2}$, $2n^{2} = \frac{2n^{2}}{1} \in \mathbb{Q}_{\infty}$,
 $y^{2} + 0 y^{1} -  \frac{2n^{2}}{1} y^{0}=0$,
 $\mathbb{Q}_{\lt}\, y^{2} + \mathbb{Q}_{\lt}\,y^{1} + \mathbb{Q}_{\infty}y^{0}=0$,
 $y \in \mathbb{A}_{\infty}$.
\end{mex}

 If we consider a number in $*G$
 as composed of three components,
 infinitesimals, reals, and infinities
 while the number is unique, this composition is not. 
 $n+2|_{n=\infty} \in \Phi^{-1}$
 is an infinity, but is composed
 of both an infinity and an integer.
 This is because $\mathbb{R}^{+} \lt n+2|_{n=\infty}$,
 and partitions the reals, 
 hence the number is an infinity, but
 $2 \in \mathbb{J}_{\lt}$, $2+n \in \mathbb{J}_{\lt}+\mathbb{J}_{\infty} \in +\Phi^{-1}$.

As a consequence of the infinitesimals $\Phi$
  and infinities $\Phi^{-1}$ definitions,
 they are not symmetrically defined.
 $\frac{1}{n}+3|_{n=\infty} \not\in \Phi$, but $n^{2}+1|_{n=\infty} \in \Phi^{-1}$.
 An infinity can have an arbitrary real or
 infinitesimal part, but an infinitesimal
 cannot have either a real or an infinity part. 
\bigskip 
\begin{prop}
The gossamer number system comprises of reals and infinirreal numbers.
\[ *G = \Phi + \mathbb{R} + \Phi^{-1} \]
\end{prop}
However, this does not mean that $*G$ as a number cannot be 
 uniquely represented.
 A unique representation 
 may aid with theorems and proofs.
 We define the following
 infinity to exclude infinitesimals and real
 numbers in an addition definition.
 Using $\Phi_{*}^{-1}$, a gossamer number can be
 uniquely represented as
 three components, $\Phi$, $\mathbb{R}$, $\Phi_{*}^{-1}$.
 As a vector of three independent components.
\bigskip
\begin{defy}\label{DEF075}
 Uniquely represent $*G$ by three independent components.
\[ a \in *G \text{ then } a = (\Phi, \mathbb{R}, \Phi_{*}^{-1} ) \]
\end{defy}

 Any number in $*G$ is in one of the following seven forms:
 $\Phi$, 
 $\Phi +\mathbb{R}$,
 $\Phi + \Phi_{*}^{-1}$,
 $\Phi + \mathbb{R} + \Phi_{*}^{-1}$,
 $\mathbb{R}$,
 $\mathbb{R}+\Phi_{*}^{-1}$,
 $\Phi_{*}^{-1}$

 For proofs,
 with the component representation,
 we can
 define addition and multiplication between
 two unique numbers in $*G$,
 for the components which we know and do not know.
\bigskip
\begin{defy}\label{DEF076}
Addition  $a + b$; $ a, b \in *G$;
\[ a+b = (a_{1},a_{2},a_{3})+(b_{1},b_{2},b_{3}) = (a_{1}+a_{2}, b_{1}+b_{2}, a_{3}+b_{3}) \]
\end{defy}

 We know
 $\mathbb{R}\backslash\{0\} \cdot \Phi \in \Phi$ (Proposition \ref{P025}),
 $\mathbb{R}\backslash\{0\} \cdot \Phi^{-1} \in \Phi^{-1}$ (Proposition \ref{P064}),
 $\mathbb{R} \cdot \mathbb{R} \in \mathbb{R}$,
 $\Phi \cdot \Phi \in \Phi$ (Proposition \ref{P090}),
 and $\Phi^{-1} \cdot \Phi^{-1} \in \Phi^{-1}$ (Proposition \ref{P111}),
 hence for these multiplications we 
 can determine which components they belong to.
\bigskip
\begin{defy}\label{DEF077}
Multiplication $a \cdot b$; $a, b \in *G$;
\[ a\cdot b = (a_{1},a_{2},a_{3})\cdot (b_{1}, b_{2}, b_{3}) \]
\[ a \cdot b = (a_{1}b_{2} + a_{2}b_{1} + a_{1}b_{1}, a_{2}b_{2}, a_{3}b_{3} + a_{3}b_{2} + a_{2}b_{3}) + a_{1}b_{3} + a_{3}b_{1} \]
\end{defy}

\bigskip
\begin{remk}
 We do not know the product type of $a_{1}b_{3}$
 and $a_{3}b_{1}$
 as the numbers could be in any or multiple categories.
 This is the familiar indeterminate case $0 \cdot \infty$.
 $\frac{5}{n+1} \cdot n|_{n=\infty}$
 $= 5 - \frac{5}{n+1}$
 $\in \mathbb{R}+\Phi$
\end{remk}

 We have constructed a number system, and 
 from this construction
 believe the following properties,
 while not proved, are true. This would require a proof 
 that $*G$ is a total order field Proposition \ref{P073},
 in perhaps a similar way to a proof of $\mathbb{R}$
 as a total order field.
\subsection{Properties} \label{S0105}
 We believe the $*G$ construction is valid is actually a field,
 but we are less certain about the cardinality.
 This could be a consequence of the representations and
 definitions, or possible errors, or that by doing the mathematics
 in another way, some insight may have been gained.

 The following may be of interest to the set theoretical mathematicians.
 We do not consider one-one correspondence over an infinity of infinite intervals
 to be correct, but over a finite number of infinite intervals to be true.
 A finite number of infinite sets is not an infinite number
 of infinite sets.

 To explain why the one-one correspondence is different,
 we believe that by including infinities within
 the set we are effectively increasing the cardinality.
\bigskip
\begin{defy}
 For determining cardinality,
 a one-one correspondence over an infinite
 interval is a one-to-one correspondence between
 a finite number of infinite intervals. 
\end{defy}

\bigskip
\begin{theo}\label{P077}
 The cardinality of $\mathbb{J}_{\infty}$
 is larger than the cardinality of
 $\mathbb{J}$
\end{theo}
\begin{proof}
$1 \ldots n$ cannot
 be put in one-one correspondence
 with $\mathbb{J}_{\infty}$,
 as one element of $\mathbb{J}_{\infty}$
 alone can be put into one-one
 correspondence with $\mathbb{J}$.
 Let $k \in \mathbb{J}_{\infty}$,
 placing $k$ in one-one correspondence with $\mathbb{J}$,
 $(\ldots, (-2,k-2), (-1,k-1), (0,k), (1,k+1), \ldots)$.
 Since there is an infinity
 of such $k$ elements, unlike the
 diagonalization argument with $2$
 infinite axes
 (which proves rational numbers
 have the same cardinality as
 the integers),
 this case has no one-one correspondence.
 E.g. Consider $(k)=(n,n^{2}, n^{3}, n^{4}, \ldots)|_{n=\infty}$.
 If this sequence were finite it could be put into a one-one correspondence, but it is not.

 A simpler way, countable infinities
 describe, $1 \ldots n$, $1 \ldots n^{2}$,
 \ldots, $1 \ldots n^{w}|_{n=\infty}$
 for finite $w$, but not
 infinite $w$.  
\end{proof}
\bigskip

 This may have consequences for the continuum hypothesis:
 that there is no cardinality between the
 integers and the real numbers.
 However,
 if the continuum hypothesis is true,
 the \mbox{infinite-integers} would have greater than or
 equal to cardinality than the real numbers.

 Theorem \ref{P077} is a different one-one correspondence from
 Cantor. With non-uniqueness at infinity, there may be other 
 mathematics here, so the truth or falsehood is relative
 to the mathematical system chosen, if one is chosen at all.

 This also has a consequence for the cardinality of the hyperreals
 which by the 
 current theory has the same cardinality as the real numbers \cite{hypercard}.
 We believe that this
 is incorrect, and hence the current theory is not modelling the situation.
 In contrast, our theory would result in the reals having 
 a lower cardinality than that of the hyperreals.
 
 We really believe that the hyperreals and $*G$ have the 
 same cardinality.

 Intuitively,
 we do not agree that $*G$ has the same 
 cardinality as $\mathbb{R}$, particularly as
 we have defined $\mathbb{R}_{\lt}$ as the
 real numbers, which are devoid of any infinity $n$ terms.  

The nature of the gossamer numbers gives
 reason to believe infinitesimals and infinities
 are much more dense than reals, for about
 any real number there is an infinity of infinitesimal numbers,
 which map back to a unique real number.
 (Part 4) 

\begin{quote}
The hyperreal point of view is that the geometric line is capable of sustaining a much richer and more intricate set than the real line. \cite[p.14]{gold}
\end{quote}
 The preceding quote does seem to contradict the current position of
 the hyperreals having the same cardinality as the reals.
 Why would this property not be reflected in the cardinality?
\bigskip
\begin{conjecture}\label{P078}
The cardinality of the gossamer numbers
 is larger than the cardinality of
 the reals.
\end{conjecture}

 Since the construction of the number system
 of reals and gossamer numbers is identical except
 with different number types,
 and the cardinality
 of the infinite integers $\mathbb{J}_{\infty}$
 is larger than the cardinality
 of the finite integers $\mathbb{J}_{\lt}$,
 then 
 the input having different cardinality
 results in an output with 
 different cardinality.

 Another approach is to consider the solution space,
 by comparing the transcendental solution space
 with the infinite transcendental solution space.
 Let $|\mathbb{N}_{\infty}|$ describe the infinite integers
 cardinality.

 $a_{k} \in \mathbb{Q}$;
 $\sum_{\mathbb{N}_{\infty}} a_{k}x^{k} = 0$ has 
 $\aleph_0^{|\mathbb{N}_{\infty}|}$ solution space. \\
 $b_{k} \in \mathbb{Q}_{\infty}$;
 $\sum_{\mathbb{N}_{\infty}} b_{k}x^{k} = 0$ has 
 ${|\mathbb{N}_{\infty}|}^{|\mathbb{N}_{\infty}|}$ solution space. \\
 Solving with comparison algebra shows ${|\mathbb{N}_{\infty}|}^{|\mathbb{N}_{\infty}|} \succ \aleph_0^{|\mathbb{N}_{\infty}|}$ and the
 gossamer numbers have a larger cardinality.

After scales of infinities
 (also see Part 2),
 another geometric example at infinity 
 involves the visualization of infinitesimally close curves.
 We can have infinitely many curves, infinitesimally 
 close to a single curve, separated by infinitely
 small distances, $f(x)-g(x)\in \Phi$ (see Example \ref{MEX019}). 

 Infinitesimals are required to describe 
 such a space, in the same way that
 real numbers are needed to extend rational numbers,
 and rational numbers to extend integers.

 Infinitesimal and infinity exclusion 
 can be compared with removing
 other classes of numbers,
 and is staggering.
 Not even Apostol \cite{apostol} discusses infinitesimals
 or infinities 
 as their own number system, 
 but implicitly uses infinitesimals
 and infinities when convenient.
 Given the banishment of infinitesimals from calculus in general,
 for example, their exclusion as a number from Apostol and generally
 every modern calculus text (we cited Apostle as this is highly regarded,
 but does not describe infinitesimals as numbers).
 Such a choice may be understandable, but it is not complete.

What follows is a construction of an extended
 real number system, we call $*G$, which
 includes infinitesimals and infinities,
 which is similar to the hyperreals,
 but different.
 
 We differentiate between infinitesimals
 and zero. Similarly we differentiate between
 an infinity such as $n^{2}|_{n=\infty}$ and the 
 ``number" $\infty$.
\subsection{Field properties} \label{S0106}

\# Unproven propositions. Other propositions 
 assume their truth.

\textbf{Proposition \ref{P042}}  $(+\Phi^{-1},+)$ is closed \\
\textbf{Proposition \ref{P113}}  $(+\Phi_{*}^{-1},+)$ is closed \\
\textbf{Proposition \ref{P060}} $(+\Phi,+)$ is closed \\
\textbf{Proposition \ref{P081}} $(\Phi \cup \{0\},+)$ is closed \\
\textbf{Proposition \ref{P082}} $(\Phi \cup \{ 0 \},+)$ is commutative \# \\ 
\textbf{Proposition \ref{P083}} $(\Phi \cup \{ 0 \},+)$ is associative \# \\
\textbf{Proposition \ref{P084}} $(\Phi_{*}^{-1} \cup \{0\},+)$ is closed \\
\textbf{Proposition \ref{P085}} $(\Phi_{*}^{-1} \cup \{ 0 \},+)$ is commutative \# \\
\textbf{Proposition \ref{P086}} $(\Phi_{*}^{-1} \cup \{ 0 \},+)$ is associative \# \\
\textbf{Proposition \ref{P087}} $(*G,+)$ is closed \\
\textbf{Proposition \ref{P088}} $(*G,+)$ is commutative \\
\textbf{Proposition \ref{P089}} $(*G,+)$ is associative \\
\textbf{Proposition \ref{P114}} $(+\Phi,\cdot)$ is closed \\
\textbf{Proposition \ref{P090}} $(\Phi,\cdot)$ is closed \\
\textbf{Proposition \ref{P111}} $(\Phi^{-1},\cdot)$ is closed \\
\textbf{Proposition \ref{P115}} $(+\Phi_{*}^{-1},\cdot)$ is closed \\
\textbf{Proposition \ref{P091}} $(\Phi_{*}^{-1},\cdot)$ is closed \\
\textbf{Proposition \ref{P092}} $\Phi \cdot \Phi_{*}^{-1} \in *G \backslash \{ 0 \}$ \# \\
\textbf{Proposition \ref{P093}} $\Phi_{*}^{-1} \cdot \Phi \in *G \backslash \{ 0 \}$ \# \\
\textbf{Proposition \ref{P112}} $(\mathbb{R}_{\infty},\cdot) \in *G \backslash \{ 0 \}$ \\
\textbf{Proposition \ref{P094}} $(*G,\cdot)$ is closed  \\
\textbf{Proposition \ref{P121}} $(\Phi^{-1},\cdot)$ is commutative \# \\
\textbf{Proposition \ref{P116}} $(\Phi_{*}^{-1},\cdot)$ is commutative \\
\textbf{Proposition \ref{P117}} $(\Phi,\cdot)$ is commutative \\
\textbf{Proposition \ref{P118}} $(\Phi, \Phi_{*}^{-1}, \cdot)$ is commutative \# \\
\textbf{Proposition \ref{P095}} $(*G \backslash \{ 0 \},\cdot)$ is commutative \\
\textbf{Proposition \ref{P096}} $(*G \backslash \{ 0 \},\cdot)$ is associative \\
\textbf{Proposition \ref{P097}} $(*G,+,\cdot)$ is distributive,
 $a \cdot (b+c) = (a \cdot b) + (a \cdot c)$ \\
\textbf{Proposition \ref{P098}} $(\Phi \cup \{ 0 \},+)$ has identity $0$ \\
\textbf{Proposition \ref{P099}} $(\Phi_{*}^{-1} \cup \{ 0 \},+)$ has identity $0$ \\
\textbf{Proposition \ref{P100}} $(*G,+)$ has identity $(0,0,0)$ \\
\textbf{Proposition \ref{P119}} $(\Phi^{-1} \cup \{ 0 \},+)$ has an inverse \\
\textbf{Proposition \ref{P102}} $(\Phi_{*}^{-1} \cup \{ 0 \},+)$ has an inverse \\
\textbf{Proposition \ref{P101}} $(\Phi \cup \{ 0 \},+)$ has an inverse \\
\textbf{Proposition \ref{P103}} $(*G,+)$ has an inverse \\
\textbf{Proposition \ref{P104}} $(*G,+)$ is an abelian group \\
\textbf{Proposition \ref{P105}} $(*G \backslash \{ 0 \}, \cdot)$ has an identity \\
\textbf{Proposition \ref{P106}} $(*G \backslash \{ 0 \}, \cdot)$ has inverse \\
\textbf{Proposition \ref{P107}} $(*G \backslash \{ 0 \}, \cdot)$ is a group \\
\textbf{Proposition \ref{P108}} $*G$ is a field \\
\textbf{Proposition \ref{P073}} $*G$ is a total order field.
 $a, b, c \in *G$;
 If $a \leq b$ then $a+c \leq b+c$.
 If $0 \leq a$ and $0 \leq b$ then $0 \leq a \times b$. \# \\
\textbf{Proposition \ref{P025}} $\mathbb{R}\backslash\{0\} \cdot \Phi \in \Phi$ \# \\
\textbf{Proposition \ref{P064}} $\mathbb{R}\backslash\{0\} \cdot \Phi^{-1} \in \Phi^{-1}$ \# \\
\textbf{Proposition \ref{P021}} $\mathbb{R}\backslash\{0\} \cdot \mathbb{R}_{\infty} \in \mathbb{R}_{\infty}$ \\
\textbf{Proposition \ref{P053}}  
Let $\theta \in *G \backslash\{0\}$,
 $\theta \gt 0$, 
 $z \in \{ \lt, \leq, ==, \gt, \geq \}$. 
 $f \; (z) \; g \Leftrightarrow \theta \cdot f \; (z) \; \theta \cdot g$ \# \\
\textbf{Proposition \ref{P054}} Let $z \in \{ \lt, \leq, ==, \gt, \geq \}$. 
 The relations $z$ invert when multiplied by a negative number in $*G$.
 For the equality case the left and right sides of the relation are interchanged. \\
\bigskip
\begin{prop}\label{P042}
 $(+\Phi^{-1},+)$ is closed 
\end{prop}
\begin{proof}
$a, b \in +\Phi^{-1}$;
 since $b \gt 0$,
 $a+b \gt a$,
 $a+b \in +\Phi^{-1}$
\end{proof}
\begin{prop}\label{P113}
 $(+\Phi_{*}^{-1},+)$ is closed 
\end{prop}
\begin{proof}
 Since the components are
 separate, by definition;
 $a, b \in +\Phi_{*}^{-1}$;
 $a+b$
 $=(0,0,a_{3}) + (0,0,b_{3})$
 $=(0,0,a_{3}+b_{3}) \in +\Phi_{*}^{-1}$
 since adding two positive numbers are also
 positive. 
\end{proof}

\begin{prop}\label{P060}
 $(+\Phi,+)$ is closed
\end{prop}
\begin{proof}
 $a, b \in +\Phi^{-1}$;
 $\frac{1}{a}, \frac{1}{b} \in \Phi$;
 $\frac{1}{a} + \frac{1}{b} = \frac{b+a}{ab}$,
 since $ab \succ a +b$ then $\frac{a+b}{ab}\in \Phi$.
\end{proof}

\begin{prop}\label{P081}
$(\Phi \cup \{0\},+)$ is closed
\end{prop}
\begin{proof}
 $a, b \in +\Phi^{-1}$;
 For both positive number and negative
 numbers are closed by Proposition \ref{P060}.
 Consider case adding infinitesimals
 of opposite sign,
 $\frac{1}{a}-\frac{1}{b}$
 $= \frac{b-a}{ab}$.
If $a=b$ then  $\frac{1}{a}-\frac{1}{b}$
 $= \frac{0}{ab}=0$ is closed.
 If $a \neq b$ then
 $b-a \prec ab$ and $\frac{b-a}{ab} \in \Phi$ is closed.
\end{proof}

\begin{prop}\label{P082}
$(\Phi \cup \{ 0 \},+)$ is commutative
\end{prop}

\begin{prop}\label{P083}
$(\Phi \cup \{ 0 \},+)$ is associative
\end{prop}

\begin{prop}\label{P084}
$(\Phi_{*}^{-1} \cup \{0\},+)$ is closed
\end{prop}
\begin{proof}
 Since the components are
 separate, by definition;
 $a, b \in \Phi_{*}^{-1}$;
 $a+b$
 $=(0,0,a_{3}) + (0,0,b_{3})$
 $=(0,0,a_{3}+b_{3}) \in \Phi_{*}^{-1}$
 or $(0,0,0)$.
\end{proof}

\begin{prop}\label{P085}
$(\Phi_{*}^{-1} \cup \{ 0 \},+)$ is commutative
\end{prop}

\begin{prop}\label{P086}
$(\Phi_{*}^{-1} \cup \{ 0 \},+)$ is associative 
\end{prop}

\begin{prop}\label{P087} 
$(*G,+)$ is closed
\end{prop}
\begin{proof}
Since $(\Phi,+)$ is closed,
 $(\mathbb{R},+)$ is closed,
 and $(\Phi_{*}^{-1},+)$ is closed,
 then, as an addition in $*G$ is the
 sum of three independent components
 are all closed. 
\end{proof}

\begin{prop}\label{P088} 
$(*G,+)$ is commutative
\end{prop}
\begin{proof}
 Since the components (Propositions \ref{P082}, \ref{P085} and $(\mathbb{R},+)$) are independent,
 and commutative.
\end{proof}

\begin{prop}\label{P089} 
$(*G,+)$ is associative 
\end{prop}
\begin{proof}
 Since the components (Propositions \ref{P083}, \ref{P086} and $(\mathbb{R},+)$) are independent,
 and associative.
\end{proof}

\begin{prop}\label{P114}
$(+\Phi,\cdot)$ is closed
\end{prop}
\begin{proof}
$a, b \in +\Phi^{-1}$;
 $\frac{1}{a}, \frac{1}{b} \in +\Phi$;
 $\frac{1}{a} \cdot \frac{1}{b}
 = \frac{1}{ab} \in +\Phi$ as
 $ab \in +\Phi^{-1}$.
\end{proof}

\begin{prop}\label{P090} 
$(\Phi,\cdot)$ is closed
\end{prop}
\begin{proof}
 Since the sign can be factored
 out, consider
 the positive case only,
 by Proposition \ref{P114} this is true.
\end{proof}

\begin{prop}\label{P111}
$(\Phi^{-1},\cdot)$ is closed
\end{prop}
\begin{proof}
 Since a negative sign can be factored out,
 only need to consider positive infinities.
 $x, y \in +\Phi^{-1}$;
 choose $x \leq y$,
 then $x \leq y \leq xy$, as $x \gt 1$.
 $xy \in \Phi^{-1}$ by Definition \ref{DEF051}. 
\end{proof}

\begin{prop}\label{P115} 
$(+\Phi_{*}^{-1},\cdot)$ is closed
\end{prop}
\begin{proof}
 $a, b \in +\Phi^{-1}$;
 $\frac{1}{a}, \frac{1}{b} \in +\Phi$;
 $\frac{1}{a} \cdot \frac{1}{b}$
 $=\frac{1}{ab}$
 $\in +\Phi$
 as $ab \in +\Phi^{-1}$.
\end{proof}

\begin{prop}\label{P091} 
$(\Phi_{*}^{-1},\cdot)$ is closed
\end{prop}
\begin{proof}
 Factor out the negative sign,
 since 
 $(+\Phi_{*}^{-1},\cdot)$ Proposition \ref{P115} 
 then closed.
\end{proof}

\begin{prop}\label{P092} 
$\Phi \cdot \Phi_{*}^{-1} \in *G \backslash \{ 0 \}$ 
\end{prop}

\begin{prop}\label{P093} 
$\Phi_{*}^{-1} \cdot \Phi \in *G \backslash \{ 0 \}$ 
\end{prop}
\begin{prop}\label{P112}
$(\mathbb{R}_{\infty},\cdot) \in *G \backslash \{ 0 \}$
\end{prop}
\begin{proof}
From Propositions \ref{P092} and \ref{P093}.
\end{proof}
\begin{prop}\label{P094} 
$(*G,\cdot)$ is closed
\end{prop}
\begin{proof}
 Any product with $0$ results in $0$.
 Consider non-zero products,
 since multiplying by the
 components 
 is closed,
 and multiplying by an infinitesimal 
 and infinity is in $*G\backslash\{0\}$ 
 by Propositions \ref{P092} and \ref{P093},
 then by the multiplication of components
 Definition \ref{DEF077},
 the product is closed.
\end{proof}

\begin{prop}\label{P121}
$(\Phi^{-1},\cdot)$ is commutative
\end{prop}

\begin{prop}\label{P116}
$(\Phi_{*}^{-1},\cdot)$ is commutative
\end{prop}
\begin{proof}
By Proposition \ref{P121}, a subset is also commutative.
 $a, b \in \Phi$; $a', b' \in \Phi_{*}$;
 If $a \cdot b = b \cdot a$, and let the
 unique expressions  be $a' = a$, $b' = b$, then it follows
 by substitution $a' \cdot b' = b' \cdot a'$.
\end{proof}

\begin{prop}\label{P117}
$(\Phi,\cdot)$ is commutative
\end{prop}
\begin{proof}
 $a, b \in \Phi^{-1}$;
 By Proposition \ref{P121} the infinities are
 commutative, 
 $\frac{1}{a} \cdot \frac{1}{b}$
 $= \frac{1}{ab}$
 $= \frac{1}{ba}$
 $= \frac{1}{b} \cdot \frac{1}{a}$
\end{proof}

\begin{prop}\label{P118}
$(\Phi, \Phi_{*}^{-1}, \cdot)$ is commutative
\end{prop}

\begin{prop}\label{P095} 
$(*G \backslash \{ 0 \},\cdot)$ is commutative 
\end{prop}
\begin{proof}
 Since the components are commutative,
 and multiplication between infinitesimals
 and infinities is commutative,
 then general multiplication is commutative.
\end{proof}

\begin{prop}\label{P096} 
$(*G \backslash \{ 0 \},\cdot)$ is associative 
\end{prop}
\begin{proof}
 Using Maxima symbolic mathematics package.
 Implement multiplication, with the fourth
 element in an unknown status. \\
f4(a1,a2,a3,a4,b1,b2,b3,b4) := [expand(a1*b2+a2*b1+a1*b1), expand( a2*b2), expand(a3*b3+a3*b2+a2*b3), expand( a1*b3+a3*b1 + a4*(b1+b2+b3) + b4*(a1+a2+a3)+a4*b4)]; \\
f5(a,b) := f4( a[1], a[2], a[3], a[4], b[1], b[2], b[3], b[4] ); 

 The following gave the same output, proving
 the associativity. \\
 f5( f5( [a1,a2,a3,0], [b1,b2,b3,0]), [c1,c2,c3,0] ); \\
 f5( [a1,a2,a3,0], f5( [b1,b2,b3,0], [c1,c2,c3,0] ) ); \\
$[ a1\,b2\,c2+a2\,b1\,c2+a1\,b1\,c2+a2\,b2\,c1+a1\,b2\,c1+a2\,b1\,c1+a1\,b1\,c1,$
 $a2\,b2\,c2,a3\,b3\,c3+a2\,b3\,c3+a3\,b2\,c3+a2\,b2\,c3+a3\,b3\,c2+a2\,b3\,c2+a3\,b2\,c2,$
 $a1\,b3\,c3+a1\,b2\,c3+a3\,b1\,c3+a2\,b1\,c3+a1\,b1\,c3+a1\,b3\,c2+a3\,b1\,c2+a3\,b3\,c1+a2\,b3\,c1+a1\,b3\,c1+a3\,b2\,c1+a3\,b1\,c1]$
\end{proof}

\begin{prop}\label{P109}
$(\Phi_{*}^{-1}\cup\{0\},+,\cdot)$ is distributive
\end{prop}

\begin{prop}\label{P110}
$(\Phi\cup\{0\},+,\cdot)$ is distributive
\end{prop}

\begin{prop}\label{P097} 
$(*G,+,\cdot)$ is distributive,
 $a \cdot (b+c) = (a \cdot b) + (a \cdot c)$
\end{prop}
\begin{proof}
$a(b+c)$
$= (a_{1},a_{2},a_{3})\cdot(b_{1}+c_{1}, b_{2}+c_{2},b_{3}+c_{3})$
$= (a_{1}(b_{2}+c_{2})+a_{2}(b_{1}+c_{1})+a_{1}(b_{1}+c_{1}),
 a_{2}(b_{2}+c_{2}), 
 a_{3}(b_{3}+c_{3})+a_{3}(b_{2}+c_{2})+a_{2}(b_{3}+c_{3}))$
 $+a_{1}(b_{3}+c_{3}) + a_{3}(b_{1}+c_{1})$
$= (a_{1}b_{2}+a_{1}c_{2}+a_{2}b_{1}+a_{2}c_{1}+a_{1}b_{1}+a_{1}c_{1},
 a_{2}b_{2}+a_{2}c_{2}, 
 a_{3}b_{3}+a_{3}c_{3}+a_{3}b_{2}+a_{3}c_{2}+a_{2}b_{3}+a_{2}c_{3})$
 $+a_{1}b_{3}+a_{1}c_{3} + a_{3}b_{1}+a_{3}c_{1}$

$ab+ac$
$=(a_{1},a_{2},a_{3})(b_{1},b_{2},b_{3}) + (a_{1},a_{2},a_{3})(c_{1},c_{2},c_{3})$ \\
$=[ (a_{1}b_{2} + a_{2}b_{1} + a_{1}b_{1}, a_{2}b_{2}, a_{3}b_{3} + a_{3}b_{2} + a_{2}b_{3}) + a_{1}b_{3} + a_{3}b_{1} ] + [ (a_{1}c_{2} + a_{2}c_{1} + a_{1}c_{1}, a_{2}c_{2}, a_{3}c_{3} + a_{3}c_{2} + a_{2}c_{3}) + a_{1}c_{3} + a_{3}c_{1} ]$ \\
$= (a_{1}b_{2} + a_{2}b_{1} + a_{1}b_{1}+a_{1}c_{2} + a_{2}c_{1} + a_{1}c_{1},
 a_{2}b_{2}+a_{2}c_{2},
 a_{3}b_{3} + a_{3}b_{2} + a_{2}b_{3}+a_{3}c_{3} + a_{3}c_{2} + a_{2}c_{3}) + a_{1}b_{3} + a_{3}b_{1} + a_{1}c_{3} + a_{3}c_{1}$

Consider that the three components are
 distributive, Propositions \ref{P109}
 and \ref{P110}
 and $\mathbb{R}$.
 They are also
 commutative, Propositions \ref{P085}
 and \ref{P082}
 and $\mathbb{R}$.
 Then the expressions are equal.
\end{proof}

\begin{prop}\label{P098} 
$(\Phi \cup \{ 0 \},+)$ has identity $0$
\end{prop}
\begin{proof}
 $a \in +\Phi^{-1}$;
 $\frac{1}{a} + 0$
 $= \frac{1}{a} + \frac{0}{a}$
 $= \frac{1 + 0}{a}$
 $=\frac{1}{a}$.
\end{proof}

\begin{prop}\label{P099}  
$(\Phi_{*}^{-1} \cup \{ 0 \},+)$ has identity $0$
\end{prop}
\begin{proof}
 By Proposition \ref{P098} since true for
 any $\Phi^{-1}$.
\end{proof}

\begin{prop}\label{P100}  
$(*G,+)$ has identity $(0,0,0)$
\end{prop}
\begin{proof}
 Solve $a+b=a$ for identity $b$.
 $(a_{1}, a_{2}, a_{3}) + (b_{1}, b_{2}, b_{3}) = (a_{1}, a_{2}, a_{3})$ then
 by the components independence,
 $a_{1}+b_{1}=a_{1}$, $b_{1}=0$ by Proposition \ref{P089}, 
 $a_{2}+b_{2} = a_{2}$, $b_{2}=0$ as a real number,
 $a_{3}+b_{3} = a_{3}$,
 $b_{3} = 0$ by Proposition \ref{P099}.
\end{proof}

\begin{prop}\label{P119} 
$(\Phi^{-1} \cup \{ 0 \},+)$ has an inverse
\end{prop}
\begin{proof}
 By assumption that any number in $*G$
 can have an integer coefficient
 multiplier
 of $\pm 1$; $x \in \Phi^{-1}$;
 consider $x + (-x)$
 $= x(1 + -1)$
 $= x \cdot 0$
 $=0$
\end{proof}

\begin{prop}\label{P102} 
$(\Phi_{*}^{-1} \cup \{ 0 \},+)$ has an inverse
\end{prop}
\begin{proof}
 By Proposition \ref{P119} as true for any $\Phi^{-1}$.
\end{proof}

\begin{prop}\label{P101} 
$(\Phi \cup \{ 0 \},+)$ has an inverse
\end{prop}
\begin{proof}
 $a, b \in +\Phi^{-1}$;
$\frac{1}{a} + \frac{1}{b} = 0$,
 $\frac{b}{ba} + \frac{a}{ba}$
 $= \frac{b+a}{ba}$
 $=0$ by Proposition \ref{P119} when $b=-a$,
 and inverse of
 $\frac{1}{a}$ is $-\frac{1}{a}$.
\end{proof}

\begin{prop}\label{P103} 
$(*G,+)$ has an inverse
\end{prop}
\begin{proof}
Since each component has an inverse, Propositions \ref{P102} and \ref{P101},
 $a+b$
 $=(a_{1},a_{2},a_{3})+(b_{1},b_{2},b_{3}) = (0,0,0)$, 
 $b = (-a_{1}, -a_{2}, -a_{3})$. 
\end{proof}

\begin{prop}\label{P104} 
$(*G,+)$ is an abelian group
\end{prop}
\begin{proof}
 Group properties: 
 closure Proposition \ref{P087}, 
 associativity Proposition \ref{P089},
 identity Proposition \ref{P100},
 inverse Proposition \ref{P103}.
 Abelian group by Proposition \ref{P088}.
\end{proof}

\begin{prop}\label{P105} 
$(*G \backslash \{ 0 \}, \cdot)$ has an identity $1$
\end{prop}
\begin{proof}
$(a_{1}, a_{2}, a_{3}) \cdot (0,1,0)$
 $= ( a_{1} \cdot 1 + a_{2} \cdot 0 + a_{1} \cdot 1, a_{2} \cdot 1,
 a_{3} \cdot 0 + a_{3} \cdot 1 + a_{2} \cdot 0) + a_{1} \cdot 0  + a_{3} \cdot 0$
 $= (a_{1} + 0 + 0, a_{2}, 0 + a_{3} + 0) + 0 + 0$
 $= (a_{1}, a_{2}, a_{3})$
\end{proof}

\begin{prop}\label{P106} 
$(*G \backslash \{ 0 \}, \cdot)$ has inverse 
\end{prop}
\begin{proof}
 If we consider the $7$ forms of $a_{1} b_{3}$ and the $7$ forms of $a_{3}b_{1}$ there
 are $49$ combinations, leading to $49$ sets of equations
 of the form $(a_{1}, a_{2}, a_{3}) \cdot (b_{1}, b_{2}, b_{3}) = (0,1,0)$.  
 Since this is $3$ linear equations with $3$ unknowns, providing
 the equations do not contradict, there is always a unique solution.
 
Contradictory solutions are a consequence of the different cases. For example, $a_{2} \neq 0$
 for the following set of equations.
 $a_{1}a_{3} \in \Phi$; $a_{3}b_{1} \in \Phi$;
 $a_{1}b_{2} + a_{2}b_{1} + a_{1}b_{1} + a_{1}b_{3} + a_{3}b_{1} = 0$;
 $a_{2}b_{2}=1$; $a_{3}b_{3} + a_{3}b_{2} + a_{2}b_{3}=0$;
\end{proof}
 
\begin{prop}\label{P107} 
$(*G \backslash \{ 0 \}, \cdot)$ is a group
\end{prop}
\begin{proof}
 Properties:
 closure Proposition \ref{P094},
 associativity Proposition \ref{P096},
 identity Proposition \ref{P105},
 inverse Proposition \ref{P106}.
\end{proof}

\begin{prop}\label{P108} 
 $*G$ is a field
\end{prop}
\begin{proof}
 Properties: 
 Proposition \ref{P104}, $(*G,+)$ is a group.
 Proposition \ref{P107}, $(*G \backslash \{ 0 \}, \cdot)$ is a group.
 Proposition \ref{P097}, $(*G,+,\dot)$ is distributive.
\end{proof}

\begin{prop}\label{P073}
$*G$ is a total order field \cite{orderedfield}.
 $a, b, c \in *G$;
 If $a \leq b$ then $a+c \leq b+c$.
 If $0 \leq a$ and $0 \leq b$ then $0 \leq a \times b$. 
\end{prop}

Scalar multiplication by a real number except $0$ is closed.
\bigskip
\begin{prop}\label{P025} 
  $\mathbb{R}\backslash\{0\} \cdot \Phi \in \Phi$
\end{prop}
\begin{prop}\label{P064} 
  $\mathbb{R}\backslash\{0\} \cdot \Phi^{-1} \in \Phi^{-1}$
\end{prop}
\begin{prop}\label{P021}
  $\mathbb{R}\backslash\{0\} \cdot \mathbb{R}_{\infty} \in \mathbb{R}_{\infty}$
\end{prop}
\begin{proof} $\mathbb{R}_{\infty}$ is the above infinitesimal and infinity cases,
 Propositions \ref{P025} and \ref{P064}.
\end{proof}

\begin{prop}\label{P053}
 $\theta \in *G \backslash\{0\}$; if $\theta$ is positive, $z \in \{ \lt, \leq, ==, \gt, \geq \}$.
 \[ f \; (z) \; g \Leftrightarrow \theta \cdot f \; (z) \; \theta \cdot g \]
\end{prop}
\begin{prop}\label{P054}
 $z \in \{ \lt, \leq, ==, \gt, \geq \}$; 
 The relations $z$ invert when multiplied by a negative number in $*G$.
 For the equality case the left and right sides of the relation are interchanged. 
\end{prop}
\begin{proof}
 $a, b \in *G$; 
 $a \; z \; b$, 
 $-a \; (-z) \; -b$,
 $-a+b \; (-z) \; 0$,
 $b \; (-z) \; a$.
\end{proof}
\subsection{Conclusion} \label{S0107}
 It is with disbelief that so many practising mathematicians
 have little explicit application of infinitesimals;
 that the online forums have so little discussion and
 that only the specialized few use NSA.
 This must change.
 Without explicitly partitioning the finite and infinite,
 there is a world of analysis that will not see the light of day.

 The
 gossamer
 number system structure
 is the simplest explanation (Occam's razor) for infinitesimals
 and infinities. 
 The construction is identical
 with the implicit equation
 construction of real numbers
 except 
 where the 
 integers, the building blocks of the real number system,
 are replaced with `infinite integers'.

 Since $*G$ (we believe) is a field,
 the theory is general.
 For example,
 you could perhaps plug $*G$
 into Fourier analysis
 theory, and extend the theory.

 Du Bois-Reymond expressed the numbers like we do today
 as functions. Largely to avoid accounts of the numbers
 being fiction and devoid of real world meaning.

  However, had he realised the number system (which we
 believe needed the discovery of the infinite integers),
 he could have expressed his theory with two number systems
 $\mathbb{R}$ and $*G$ or $*R$.
 Many others, including
 Newton, Leibniz, Cauchy, Euler and Robinson
 have considered these questions. 
 A two-tiered number system $\mathbb{R}$ and $*G$ has evolved, not as an option,
 but as a necessity in explaining
 mathematics.

 We have attempted to find
 a rigorous formulation
 to ``traditional non-rigorous extensions",
 which we believe has not been considered. 
 By construction, the zero divisors problem relevant to
 the hyperreals is avoided, as in $*G$ there is only
 one zero.

 Acknowledgment, help from a reviewer: thank you for reviewing
 the original paper in earlier stages,
 which was subsequently split into
 six parts (due to the scope of the investigation),
 and then this paper. Thank you.

 In response, better communication with numerous suggestions
 and extensions were subsequently applied,
 and the construction
 of the gossamer number system was found.

\section{The much greater than relations} \label{S02}
 An infinitesimal and infinitary number system the 
 Gossamer numbers is fitted to
 du Bois-Reymond's 
 infinitary calculus,
 redefining the magnitude relations.
 We connect 
 the past symbol relations much-less-than $\prec$ and
 much-less-than or equal to $\preceq$
 with the present little-o and big-O notation, which have identical definitions.
 As these definitions are extended,
 hence we also extend little-o and big-O,
 which are 
 defined in Gossamer numbers.
 Notation for a reformed infinitary calculus,
 calculation at a point is developed.
 We proceed with
 the introduction of an extended infinitary calculus.
\subsection{Introduction} \label{S0201}
\begin{sloppy}
\begin{quote}
 While the majority of mathematicians readily accepted
 the emancipation of \mbox{analysis} from geometry there were,
 nonetheless, powerful voices raised against the
 arithmetization programs.
 One of the sharpest critics was Paul du Bois-Reymond (1831-1889)
 who saw the arithmetization as a contentless attempt
 to destroy the necessary union between number and magnitude.
 \cite[p.92]{pointsett}
\end{quote}
\end{sloppy}

The separation between geometry and number by
 the arithmetization of analysis has led
 to the dominance of set theory.
 However, just
 as we have different languages, problems
 can be described with functions or set theory and
 other ideas with infinity.
 Language for theories is both
 an evolution and also can be more of a choice,
 but `does' have an effect on how we see the mathematics.

 We believe that arguments of magnitude are essential to understand
 real and gossamer numbers. Without a theory from this viewpoint 
 many things are left without explanations.
 Without the relations described, the symbolism and language
 of algebra that they describe is harder to encapsulate.

 Contradictory to du Bois-Reymond, we find arithmetization 
 in his relations that lead to a transfer principle (Part 4)
 and non-reversible arithmetic Part 5. So we claim that
 they are important.

 Before becoming
 aware of du Bois-Reymond's work, we 
 defined $\gg$ equivalently to $\succ$,
 as during a mathematical modelling
 subject a lecturer had symbolically used the symbol
 to describe (without definition) large
 differences in magnitude.

 The notion of the `order' or the `rate of increase'
 of a function is essentially a relative one \cite[p.2]{ordersofinfinity}.
 Consider functions $f(x)$ and $\phi(x)$, we could 
 have functions satisfying relation $f \gt \phi$.	
 However, what about their ratio?
 Knowing only $\gt$ or $\geq$ does not give
 a size difference of the numbers involved.

 Consider monotonic functions which
 over time settle down, and have
 properties such as 
 their ratio is monotonic too.
 In examining these well behaved functions,
 families of ratios, 
 scales of infinities (Section \ref{S0204}) are considered. 
 From these investigations,
 the characterisation
 of an infinity in size difference
 was discovered, and defined as a relation $\succ$ (Definition \ref{DEF006})
 and $\succeq$ (Definition \ref{DEF005}). Here, it is not the sign of the number,
 but the size of the number which determines the relation.
\begin{quote}
 With the particular system of notation that he invented, it is, no doubt, 
 quite possible to dispense; 
 but it can hardly be denied that the notation is exceedingly useful,
 being clear, concise, and expressive in a very high degree. \cite[p. (v)]{ordersofinfinity}
\end{quote}
 However, the notation was quickly superseded by little-o and big-O, primarily
 because the magnitude relation,
 instead of being expressed separately, could
 be packaged as a variable. E.g. $\mathrm{sin}\,x = x + O(x^{3})$
 instead of $x^{3} \succeq -\frac{x^{3}}{3!}+\frac{x^{5}}{5!}-\ldots|_{x=0}$

 We believe du Bois-Reymond's relations do have
 a critical place, where we develop an algebra
 for comparing functions Part 3. 
 For many reasons, we introduce the at-a-point notation,
 which is used throughout our series of papers
 towards the development of
 an alternative to non-standard analysis
 which we refer to as infinitary calculus.

 By fitting an infinitesimal number system to du Bois-Reymond's
 infinitary
 calculus definitions (Section \ref{S0203}), 
 the theory is better explained. Instead of defining a limit in $\mathbb{R}$,
 the limit is defined in $*G$ the extended number system.

 The benefits continued with
 the later development of a transfer principle Part 4 between $*G$
 and $\mathbb{R}$,
 which
 explains mathematics that would not make sense without infinitesimals
 and infinities. General limit calculations do not make
 sense in $\mathbb{R}$ because the number system has no
 infinity elements.
 
 This theory is then used to derive a new field
 of mathematics `convergence sums' \cite{cebp2},
 with applications to convergence or divergence
 of positive series. 
 Where du Bois-Reymond's theory
 of comparing functions has 
 been forgotten,
 we now believe 
 we have found useful applications. 

 We believe this 
 shows value and general applicability
 of the mathematics. The Gossamer number system's 
 utility is demonstrated. 
 So, we see this as a building paper.

 The relation $\succ$,
 is defined as equivalent to little-o,
 and 
 is 
 general because it exists
 in a number system which includes infinities and infinitesimals,
 and not a modified or implicitly defined $\mathbb{R}$. 
 The current practise implicitly uses infinitesimals and infinities,
 without declaring them as their own number type
 Part 1.
 
 In this sense we have extended du-Bois Reymond and
 Hardy's work.  By explicitly having a number system,
 we can better compare functions.
 In a later paper on the transfer principle,
 we will argue that this is not just an option,
 but a fundamental part of calculus.

 Having said the above,
 the objective of this paper is
 to introduce definitions and notations,
 re-state du Bois-Reymond's infinitary calculus,
 and connect the past with the present
 little-o and big-O notation.
\subsection{Evaluation at a point} \label{S0202}
Motivation: Approximating functions by truncation in calculation
 is common practice. We use infinitesimals all the time.
\bigskip
\begin{mex}\label{MEX047} 
 A quick numerical check will provide evidence by approximation,
 where successive powers significantly reduce in magnitude,
 $x=0.1$, $x^{2} = 0.01$, $x^3=0.001$. 

 $f(x) = x + x^2 +x^3 + \ldots|_{x=0}$ at $x=0$
 we mean $x \in \Phi$
 (see Definition \ref{DEF041}).
 This may be represented by
 $f(x)=x$, $f(x) = x + x^{2}$,  $f(x)=x+x^{2}+x^{3}$ or any other
 number of first terms at $x=0$. 
 As $x^2$ is much smaller than $x$,
  $x^{3}$ is much smaller than $x^{2}$, ...
\end{mex}

When we send $x \to 0$, which functions that go to zero faster
 matters, as these may be truncated, and we
 can start using infinitesimals.
 It is this sort of reasoning and calculation that 
 leads to the definition of the magnitude relation
 (see Definitions \ref{DEF005} and \ref{DEF006}),
 and then to little-o and big-O notation which we use today. 

Some people, who find negative numbers difficult
 to accept, will happily add and subtract
 positive numbers, but are 
 unable to do so using negative numbers.
 In a similar way, there may occasionally
 be a problem where people can reason with infinity
 but not zero, or the other way round.  Logically
 $0$ and $\infty$ as numbers are very similar.
\bigskip 
\begin{mex}\label{MEX045} 
 Similar reasoning can be done with infinities. 
 Let $x=1/n$ and 
 assume a solution. Consider the series first three terms.
 $y= \frac{1}{n} + \frac{1}{n^{2}} + \frac{1}{n^3}$, 
 $y n^{3} = n^{2} + n + 1$ when $n=\infty$. 
 As $n$ is much greater than $1$, 
 assume $n+1=n$,
 $y n^{3} = n^{2} + n$, reversing,
 $y=\frac{1}{n}+\frac{1}{n^{2}}$ and
 $\frac{1}{n^{3}}$ was truncated. 
\end{mex}

Truncation non-uniqueness in calculation:
 Let $f(x)$ be a function of infinite terms used in function $h(x)=g(f(x))$.
 Since a truncated $f(x)$ can solve $h(x)$ for infinitely many truncations,
 we say $f(x)$ is not unique, as an infinite number of solutions may
 give a satisfactory result.
 When evaluating $h(x)$, $f(x)$ is not unique,
 as an infinite number of truncated evaluations can occur.
 It is often desirable to use the minimum number of first terms of $f(x)$ to evaluate $h(x)$. 
 In this way, asymptotic expansions as given by $f(x)$ are said to be non-unique. 

Calculation
 is a major part
 of analysis,
 and one of the most common
 evaluations is the limit of a function.
 This evaluation can be thought of more
 generally by
 considering the behaviour at a point, with the inclusion
 of infinity as a number and as a point. 

When the ideas of a point
 are extended to include such
 properties as continuity,
 infinity, existence and divergence at a point, then it becomes clear
 that a point,
 whatever it may be, is
 both what we interpret 
 and how we calculate. 

 With a view to realising something more
 general than a limit,
 the following definition at
 a point is given. 

\textit{ By virtue of reaching a point,
 we have to pass through or approach the point.
 The
 definition of evaluation at
 a point will also encompass 
 approaches to the point.} 

This interpretation of a point accepts non-uniqueness
 - two parallel lines could meet at infinity
 or they may never meet at infinity.
 In particular, asymptotic expansions
 are not unique, but subject to
 orders of magnitude. 
\bigskip
\begin{defy} \label{DEF001}
Let $f(x)$ be an expression.
 Then evaluation
 at-a-point $f(x)|_{x=a}$ is the evaluation of $f(x)$
 at $x=a$. (Optionally omit variable assignment, $f(x)|_{a}$) \\
Case 1. All possibilities or \\ 
Case 2. Context dependent evaluation
\end{defy}

Infinity can be considered as a point. 

\textit{Case 1} concerns itself with all the
 different ways a point could be 
 interpreted and calculated, and
 is a conceptual tool. 

Given a problem either theoretical or practical,
 there are often different views or interpretations
 which may help. (For example from a programming perspective, an object
 orientated approach to problem modelling.)

 Let $C^{j}$ describe a curve continuous in the first $j$
 derivatives, then let $C^{0}$ describe a continuous curve.
 When building a curve that is
 a function, except at a point,
 the following possibilities may occur (see Figure \ref{FIG07}). 
 The curve may be 
 discontinuous at the point, or its vector equations are continuous but
 the function has infinitely many values at the point,
 or $C^{0}$ but not $C^{1}$ continuous,
 or the curve is a function and also an s-curve between an interval.
 With a point at infinity
 the possibilities are endless.

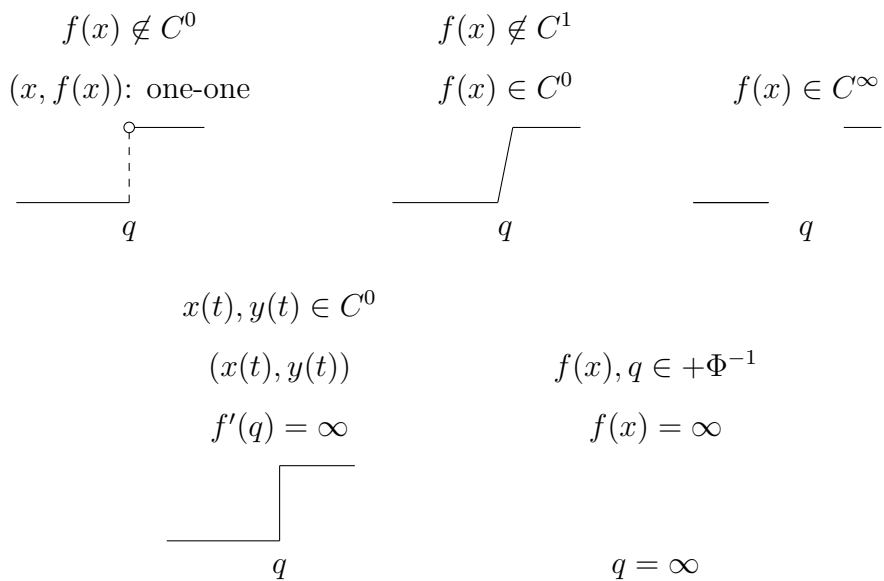
\begin{figure}[H]
\centering
\begin{tikzpicture}

\begin{scope}[xshift=0mm, yshift=45mm]
  \draw (0,1) to (1.5,1);
  \draw (1.5,1) [dashed] to (1.5,2);
  \draw (1.5,2) to (2.5,2);
  \draw[fill=white] (1.5,2.0) circle(0.07cm);
  \node [label={[shift={(15mm,2mm)}]$q$}] {};
  \node [label={[shift={(15mm,20mm)}]$(x,f(x))\text{: one-one}$}] {};
  \node [label={[shift={(15mm,28mm)}]$f(x) \not\in C^{0}$}] {};
\end{scope}

\begin{scope}[xshift=20mm, yshift=0mm];
  \draw (0,1) to (1.5,1);
  \draw (1.5,1) to (1.5,2);
  \draw (1.5,2) to (2.5,2);
  \node [label={[shift={(15mm,36mm)}]$x(t),y(t) \in C^{0}$}] {};
  \node [label={[shift={(15mm,28mm)}]$(x(t),y(t))$}] {};
  \node [label={[shift={(15mm,2mm)}]$q$}] {};
  \node [label={[shift={(15mm,20mm)}]$f'(q)=\infty$}] {};
\end{scope}
\begin{scope}[xshift=50mm, yshift=45mm];
  \draw (0,1) to (1.4,1);
  \draw (1.4,1) to (1.6,2);
  \draw (1.6,2) to (2.5,2);
  \node [label={[shift={(15mm,2mm)}]$q$}] {};
  \node [label={[shift={(15mm,20mm)}]$f(x) \in C^{0}$}] {};
  \node [label={[shift={(15mm,28mm)}]$f(x) \not\in C^{1}$}] {};
\end{scope}
\begin{scope}[xshift=90mm, yshift=45mm];
  \draw (0.0,1.0) -- (1.0,1.0);
  \node [label={[shift={(15mm,2mm)}]$q$}] {};
  \draw (0.0,1.0) -- (1.0,1.0);
  \draw plot[samples=600,domain=1.0:2.0] function {1.0*1/(1+exp(16.0*(-x+1.5)))+1};
  \draw (2.0,2.0) -- (2.5,2.0);
  \node [label={[shift={(1.5cm,2.0cm)}]$f(x) \in C^{\infty}$}] {};
\end{scope}
\begin{scope}[xshift=70mm];
  \draw plot[samples=600,domain=1.0:2.0] function {0.2*x*x*x+0.8};
  \node [label={[shift={(15mm,2mm)}]$q=\infty$}] {};
  \node [label={[shift={(15mm,20mm)}]$f(x)=\infty$}] {};
  \node [label={[shift={(15mm,28mm)}]$f(x), q \in +\Phi^{-1}$}] {};
\end{scope}
\end{tikzpicture}
\caption{Examples of interpretations at a point} \label{FIG07}
\end{figure}

\textit{Case 2} is the practical aspect of
 calculating, where a choice
 of interpretation has
 been made,
 on proceeding with
 the ``actual calculation". 
 The context calculation separates responsibility for
 the justification from the theory
 to the point of use.
 This decoupling is
 important. If there is another
 way of calculating or using
 another branch of mathematics, the evaluation
 at-a-point is simply interpreted then.
 The trade is that less
 can be said, in that the definitions
 and theory are less exacting,
 but this is 
 mitigated by the calculation being context specific, 
 and more adaptable to our problem solving.

 The consequences of decoupling can be non-trivial.
 For example, we do
 not believe 
 in 
 necessarily
 using a field when 
 extending the reals.
 A trade-off for a different kind of generality
 may be a different number system,
 or chosen differently, not depending on what you want to do.

Evaluating a function
 at-a-point can often result
 in the 
 evaluation of the limit.
 Indeed, this limit at a point,
 is a subset of the possibilities.
In the context
 of a calculation,
 $\lim\limits_{x \to a} f(x)$ can be represented by 
 $f(x)|_{x=a}$. 

As mathematics is a language,
 a further purpose of Definition \ref{DEF001}
 is to communicate
 to the reader 
 that other ways of calculation
 might be employed.  
 For instance, where
 these could be incompatible with rigorous argument,
 one way of distinguishing
 differences could be through  
 the above notation. 

The notation can also be used to make
 existing arguments more explicit.
 For example $f(x) = O(g(x))|_{x=\infty}$
 says that the function is being considered at infinity, not $0$ or 
 any other finite value. 

Motivation for a separate notation is used so that different mathematics can work side by side with standard mathematics, and
 in a sense be contained.
 The limit concept is so `ingrained' that doing operations 
 that use a different paradigm, without clear communication to the reader, would be unsatisfactory. 

A consequence of such flexibility is that mathematical inconsistencies can and are invariably introduced into the calculations.
 Where this is viable, the benefits brought to the calculation can outweigh
 any adverse presumptions.
 Designing methods to protect against inconsistencies
 other than narrowed definition and practice
 can actually make the calculus more accessible and interesting. 
\bigskip
\begin{defy}\label{DEF002} 
Let the ``at-a-point" definition, appearing on
 the right-hand side, 
 apply to all functions within the expression,
 unless overridden by another at-a-point definition. 
\end{defy}
\bigskip
\begin{mex}\label{MEX002} 
 $f(x) \; z \; g(x)|_{x=a}$
 means
 $f(x)|_{x=a} \; z \; g(x)|_{x=a}$
 where 
$z$ is any relation.
 Optionally we can include round brackets around the
 expression,
 $(f(x) \; z \; g(x))|_{x=a}$.
\end{mex}
\bigskip
\begin{defy} \label{DEF003} 
We say  
 $f(x) \; z \; g(x)|_{x=a}$ can have a context
 meaning, where $f(x)$ and $g(x)$ are dependent, and in some way
 governed by operator or relation $z$. 
\end{defy}

When forming conditions for infinitely
 small or infinitely large, we employ a bound,
 which itself is going to zero or infinity;
 for instance, when forming definitions.
\bigskip
\begin{defy}\label{DEF031}
In context, a variable $x$ can be described at infinity
 $|_{x=\infty}$,   
 then $\exists x_{0}, \forall x: x \gt x_{0}$
\end{defy}
\bigskip
\begin{defy}\label{DEF041}
In context, a variable $x$ can be described at zero 
 $|_{x=0}$  
 then $\exists x_{0}, \forall x: |x| \lt x_{0}$
\end{defy}
 With the transfer principle Part 4, Definitions \ref{DEF031} and \ref{DEF041}
 which describe the neighborhood can be better expressed
 with infinitesimals Definition \ref{DEF016}
 and infinities Definition \ref{DEF014};
 are defined with
 sequences more generally in Part 6.
\bigskip
\begin{defy}\label{DEF014}
In context, a variable $x$ can be described at infinity
 $|_{x=\infty}$,   
 then $x \in +\Phi^{-1}$ an infinity.
\end{defy}
\bigskip
\begin{defy}\label{DEF016}
In context, a variable $x$ can be described at zero 
 $|_{x=0}$  
 then $x \in \Phi$ an infinitesimal.
\end{defy}
\bigskip
\begin{defy}\label{DEF047}
In context, we say $f(x)|_{x=\infty}$ then $\lim\limits_{x \to \infty} f(x)$.
\end{defy}
Generally the equality ``$=$" with respect to assignment is defined with
 a left-to-right ordering see Definition \ref{DEF028}.
 Essentially as reasoning, such that the right follows from the left.
 This can be exact, as in one form is converted to another,
 or as a generalization, or rather implication. Therefore the context
 needs to be understood.
\bigskip
\begin{defy}\label{DEF028}
 In context,
 assignment has a left-to-right ordering.
\[ \text{In context, instance } = \text{ generalization} \]
\end{defy}
\bigskip
\begin{mex}\label{MEX057}
 How to use the notation is open, 
 some examples follow.
\\ 1. Limit calculations $(1+\frac{1}{n})^{n}|_{n=\infty}=e$ \\ 
 2. Divergent sums $\sum_{k=1}^{n} \frac{1}{k} = \mathrm{ln}\,n|_{n=\infty}$ \\ 
 3. A conversion between series and integrals,
 read from left-to-right. \\
 $\hspace*{15mm} \sum_{1}^{n} a_{n} = \int_{1}^{n} a(n) \,dn+c|_{n=\infty}$  \\
 4. A comparison relation. $n! \gt n^{2}|_{n=\infty}$ \\
 5. An infinitary calculus relation $f_{n} \prec g_{n}|_{n=\infty}$
 (as described in Definition \ref{DEF006}). \\
 6. Asymptotic results
\end{mex}

Provided there are no contradictions,
 the expressions
 at infinity can be handled algebraically in the usual ways.
\bigskip
\begin{mex}\label{MEX003} 
$n! = (2 \pi)^{\frac{1}{2}} n^{n+\frac{1}{2}} e^{-n}|_{n=\infty}$ 
 then times by $e^{n}|_{n=\infty}$ gives
 $n!e^{n} = c n^{n+\frac{1}{2}} |_{n=\infty}$.
\end{mex}
\bigskip
\begin{mex}\label{MEX001} 
$\mathrm{ln}(n!)|_{n=\infty}$
$= \sum_{k=1}^{n} \mathrm{ln}\,k|_{n=\infty}$
$= \int_{1}^{n} \mathrm{ln}\,k\,dk + \gamma|_{n=\infty}$
$=[ k \, \mathrm{ln}\,k\ - k ]_{1}^{n} + \gamma|_{n=\infty}$
$= n \, \mathrm{ln}\,n\ - n |_{n=\infty}$
\end{mex}
\bigskip
\begin{defy}\label{DEF200} 
Generalize the at-a-point Definition \ref{DEF002} to
 include condition $c(x)$ in relation
 $f(x)|_{c(x)}$. Where $c(x)$ can describe an interval.
\end{defy}
\bigskip
\begin{mex}
 $f(x) \;\; z \;\; g(x) |_{x =(0,1]}$ describes
 the relation $z$ (see Part 3) over the interval $(0,1]$. 
\end{mex}
\bigskip
\begin{mex}\label{MEX038}
The vertical bar notation is more general 
 when
 working across different situations. Such as when
  little-o and big-O notation may be cumbersome
 $O(x^{2}) + O(x^{3}) = O(x^{2})$ becomes 
 $x^{2} + x^{3} = x^{2}|_{x=0}$,
 as an alternative to the approximation
 symbol
 so $a \approx b = c$ becomes
 $a=b=c|_{n=\infty}$,
 $x^{n} e^{-n} \approx n^{n} e^{-n} e^{-\xi^{2}/2}$
 becomes $x^{n} e^{-n} = n^{n} e^{-n} e^{-\xi^{2}/2}|_{n=\infty}$ 
, 
$f \sim g$ becomes $f=g|_{n=\infty}$,
 so assignment becomes consistent.
\end{mex}

 Since non-uniqueness is accepted with
 the notation; $\mathrm{sin}\,x = x- x^{3}/3!|_{x=0}$ can be understood
 to mean truncation - an exact happening
 at $x=0$:
 an approximation in $*G$ where $\Phi \mapsto 0$ see
 Part 4.
 The notation gives you the choice. 
 If you want to be more `exact' or explicit,
 then use other or further relations.
 $\mathrm{sin}\,x = x - x^{3}/3! + O(x^{5})|_{x=0}$.

 If say, after $k$ or more terms, the calculation
 is invariant with non-reversible arithmetic,
 we define this as an `exact happening'.
 Increasing the number of terms does not change
 the calculation. After a transfer, the calculations
 produce the same result. Such situations
 are common, where to few terms in the approximations
 give incorrect results.

 The notation is also built with comparing
 functions in mind, where
 non-reversible arithmetic (see Part 5) is applied.
 An alternative to the limit notation, as it applies
 across relations
 and as an aid to Landau notation,
 can be replacing $x \to \infty$ with $|_{x=\infty}$.
 The concept of approach is logically equivalent 
 to being at the value, and the notation can say this.

 A notation for chaining arguments using commas as an implication and context is used.
 The last proposition uses the first expression.
 As a free algebra, this does place responsibility on both reading expressions
 and writing expressions.
 The notation is concise. When there are errors in evaluation or proofs,
 the chaining arguments can be rewritten one expression per line
 and edited. (Later, for similar reasons but a narrower purpose,
 in context we have defined $=$, with a left-to-right ordering.)
\bigskip
\begin{mex}
 $x, y \in \mathbb{R}; \, x \gt 0$, $y \gt 0$, $x+y \gt 0$
\end{mex}
\bigskip
\begin{defy}
 Mathematical arguments can be chained with context by
 commas (`,')
 and semicolons (`;'),
 from left-to-right order.
 The next statement optionally has a left-to-right implication.
 The semicolons have a lower precedence.
\end{defy}

In evaluating a function
 at-a-point we can shift any point $x=a$ to the 
 origin or infinity,
 stating
 the definition at $\infty$, is as general
 as
 stating the definition at any other point.  
\bigskip
\begin{prop}\label{P057}
If; $f, x \in *G$; then 
 $f(x)|_{x=a}  = f(x+a)|_{x=0} = f(\frac{1}{x}+a)|_{x=\infty}$ 
\end{prop}
\bigskip
Hence the investigation at infinity
 is similar to the corresponding theory at 
 zero.
 It is here that infinitesimals (near zero)
 and infinities (near infinity)
 of infinitary calculus operate.
\subsection{Infinitary calculus definitions} \label{S0203}
The following is a summary 
 and extension, by means of the
 gossamer numbers Part 1, 
 of infinitary calculus definitions
 and a comparison with the derived
 work of
 little-o and big-O relations. 
 This is a calculus of magnitudes.

 Occasionally 
 $\ll$ and $\gg$
 are used 
 to indicate much smaller or larger 
 numbers.
 Correspondingly $\prec$ and $\succ$
 implement the idea of much smaller than and much greater than
 numbers by defining 
 infinitely smaller and infinitely larger relations.

While a finite number is not infinity, a very
 large number treated as infinity would 
 model the situation and allow reasoning. 
 These relations, at zero or infinity provide
 an implementation of this.  

The idea of much 
 larger numbers is extended to infinity,
 where it becomes obvious that there
 are much larger numbers than others;
 hence du Bois-Reymond's development of
 the Definitions \ref{DEF006} and \ref{DEF005},
 which are equivalent to little-o and big-O 
  respectively.

 We demonstrate the connection
 between modern relations and du Bois-Reymond's relations
 and restate some definitions
 of du Bois-Reymond, 
 referred to 
 by G. Hardy in \textit{Orders of Infinity}
 \cite[pp 2--4]{ordersofinfinity}
 with 
 their representation
 in Landau notation.
\bigskip
\begin{defy}\label{DEF005}  
We say $f(x) \preceq g(x)|_{x=\infty}$
 if there exists $M \in \mathbb{R}^{+}$: $|f(x)| \leq M|g(x)||_{x=\infty}$. 
\[ f(x) \preceq g(x) \text{ is the same as } f(x) = O(g(x)) \] 
\end{defy}
\bigskip
\begin{defy} \label{DEF006} 
We say $f(x) \prec g(x)|_{x=\infty}$ 
 then $\frac{f(x)}{g(x)}|_{x=\infty} \in \Phi$
\[f(x) \prec g(x) \text{ is the same as } f(x) = o(g(x)) \]
\end{defy}

 Definition \ref{DEF006} is equivalently
 defined $|f(x) \leq M_{2} | g(x)|$,
 $M_{2} \in \Phi^{+}$.
 We see that this is almost the same
 as Definition \ref{DEF005} except $M \in \mathbb{R}^{+}$.
 One is bounded in $\mathbb{R}$, the other
 in $*G$.
\bigskip
\begin{prop}\label{P020}
 If $f(x) \prec g(x)|_{x=\infty}$ then
 $\frac{f(x)}{g(x)}|_{x=\infty} = 0$ in $\mathbb{R}$. 
 As a definition see \cite[p.2]{ordersofinfinity}. 
\end{prop}
\begin{proof}
 Apply a transfer $\Phi \mapsto 0$ (see Part 4) to 
 Definition \ref{DEF006}.
\end{proof}
\bigskip
\begin{prop}\label{P203}
If $f(x) \prec g(x)|_{x=\infty}$ then $\frac{g(x)}{f(x)}|_{x=\infty}=\infty$.
\end{prop}
\begin{proof}
 By Definition \ref{DEF006}
 let $\delta \in \Phi$,
 $\frac{f}{g}|_{x=\infty}=\delta$,
 $\frac{g}{f}|_{x=\infty} \in \Phi^{-1} = \infty$.
\end{proof}
\bigskip
\begin{prop}\label{P026}
$\delta \in \Phi^{+}$; if $\frac{f(x)}{g(x)} \leq \delta$ then $f(x) \prec g(x)$.
\end{prop}
\begin{proof}
Since $f$ and $g$ are positive, an infinitesimal is their upper bound.
 Since choosing
 any infinitesimal in $(0,\delta]$ satisfies the much-less-than relation,
 $f \prec g$.
\end{proof}

[ When applying Proposition \ref{P011}, 
 we will need to avoid division by zero via a transfer 
 $0 \mapsto \Phi$ and $\infty \mapsto \Phi^{-1}$,
 thereby treating $0$ and $\infty$ separately.
 See Part 4 ]
\bigskip
\begin{prop}\label{P011}
 $0 \prec \Phi$ and $\Phi^{-1} \prec \infty$ 
\end{prop}
\begin{proof}
 Since a magnitude relation, we need only consider the positive case.

 $0 \prec \Phi^{+}$:
 $\delta_{1}, \delta_{2} \in \Phi^{+}$
 Consider $0 \; z \; \delta_{1}$.
 Choose $\delta_{2} \prec \delta_{1}$ as there is no smallest number.
 $0 \lt \delta_{2} \lt \delta_{1}$.
 Since $0$ is smaller than $\delta_{2}$ then $0 \prec \delta_{1}$.

 $+\Phi^{-1} \prec \infty$: 
 By inverting the relation we obtain the infinite case.
 Choose $\delta_{2} \prec \delta_{1}$,
 $\frac{1}{\delta_{2}} \succ \frac{1}{\delta_{1}}$,
 $\frac{1}{\delta_{1}} \prec \frac{1}{\delta_{2}}$,
 however $\frac{1}{\delta_{2}} \lt \infty$ then
 $\frac{1}{\delta_{1}} \prec \infty$.
\end{proof}
\bigskip
\begin{defy} \label{DEF045}
\[ \text{When } f(x) \prec g(x) \text{ we say }
 f(x) \text{ is much-less-than } g(x)\]
\[ \text{When } f(x) \succ g(x) \text{ we say }
 f(x) \text{ is much-greater-than } g(x) \]
\end{defy}
\bigskip
\begin{defy}\label{DEF007}  
\[ g(x) \prec f(x) \text{ is the same as } f(x) \succ g(x) \] 
\[ g(x) \preceq f(x) \text{ is the same as } f(x) \succeq g(x) \]
\end{defy}

\begin{defy} \label{DEF008} 
\[ f(x) \succ g(x) \text{ is the same as } f(x) = \omega(g(x)) \]
\[ f(x) \succeq g(x) \text{ is the same as } f(x) = \Omega(g(x)) \]
\end{defy}

Because little-o and big-O are
 defined on a right-hand side order
 for $\prec$ and $\preceq$, 
 additional symbols are needed for $\succ$ and $\succeq$.  Here, infinitary calculus has
 a notational advantage. 
\bigskip
\begin{defy} \label{DEF010}
 We say $f(x) \asymp g(x)|_{x=\infty}$ if $f(x) \succeq g(x)|_{x=\infty}$ and $f(x) \preceq g(x)|_{x=\infty}$.
 (See Definition \ref{DEF202} and Proposition \ref{P074}.)
\[ f(n) \asymp g(n) \text{ is the same as } f(n) = \Theta(g(n)) \]
\end{defy}
\bigskip
\begin{defy}\label{DEF201}
We say $a \simeq b$ then $a$ and $b$ are infinitesimally close,
 $a-b \in \Phi \cup \{0\}$ \cite[p.57]{abraham}
\end{defy}
\bigskip
\begin{defy} \label{DEF009} 
We say $f(x) \sim g(x)|_{x=\infty}$
 then
 $\frac{f(x)}{g(x)}|_{x=\infty} \simeq 1$
\end{defy}

 We may consider the asymptotic relation $\sim$
 as an equality with respect to the product,
 and the infinitesimally close relation $\simeq$
 as an equality with respect to addition.
\bigskip
\begin{defy}\label{DEF011}  
We say $f(x) \propto g(x)|_{x=\infty}$
 if $f(x) / g(x)|_{x=\infty} \simeq c$.
  This uses a different relation symbol from
 Hardy's in  
\cite[pp 2--4]{ordersofinfinity}.
\end{defy}

The functions $f(x) / g(x)$ may not necessarily be compared,
 particularly
 if oscillating between categories at infinity occurs.
 \cite[p.4]{ordersofinfinity}
 $f \succ g$ and
 $f \preceq g$ are not each other's logical
 negations in general. 
\bigskip
\begin{mex}\label{MEX023} 
 A counter example demonstrating logical
 negation does not imply a much less than or equal to
 relation. 
\begin{align*}
 (f_{n}/g_{n})|_{n=\infty} = (0, \infty, 0, \infty, \ldots) \tag{Sequence at infinity} \\
\text{Assume } f_{n} \nsucc g_{n} \text{ implies } f_{n} \preceq g_{n} \tag{Dividing by $g_{n}$}|_{n=\infty} \\
 ( f_{n}/g_{n} ) \preceq ( g_{n}/g_{n} ) \tag{Component-wise comparison} \\
 ( f_{n}/g_{n} ) \preceq (1, 1, 1, \ldots ) \\
(0, \infty, 0, \infty, \ldots) \preceq (1, 1, 1, \ldots ) \tag{A component-wise contradiction} 
\end{align*}
\end{mex}
\bigskip
\begin{prop}\label{P033}
 If the relation between 
 $f$ and $g$ are
 $\{ f \prec g, f \asymp g, f \succ g  \}$, 
 then
 the negation of
 one of these
 relations 
 would imply 
 one of the other two relations.
\end{prop}
\begin{proof}
$f \succ g$ and $f \prec g$ are disjoint.
Since $\asymp$: $f \succeq g$ or $f \preceq g$ covers
 the remaining cases.
 Since this is given as disjoint,
 only one of the three cases can occur. 
\end{proof}
 
 Further theorems follow:
 if $f \succ g$, $g \succeq h$, then $f \succ h$.
 This is interesting from
 an application perspective as the ratio
 $f/g$ has settled down into one
 of the three relations.

In comparing relations $\{ o(), O(), \omega(), \Omega() \}$
 with $\{ \prec, \preceq, \succ, \succeq \}$,
 while the variable relation and symbols are equivalent,
 the symbols can be easier to manage and understand in comparison.

However the relation variables of Landau's notation
 have a major
 advantage over infinitary calculus
 relation symbols
 in that the relation 
 is packaged as a variable in
 the equation.
 $\frac{1}{1-x} = 1 + x + x^{2} + O(x^{3})$.

Consequently the definitions of 
 infinitary calculus symbols and  Landau notation
 can be viewed
 as complementary.

We add further definitions to infinitary calculus that 
 extend its use as an
 infinitesimal calculus analysis.
\bigskip
\begin{defy}\label{DEF025}
We say $c(x) \prec \infty$ when $c(x) \neq \pm\infty$ and that
 $c(x)$ is bounded.
\end{defy}

$a \prec \infty$ is not
 the same as $a \in \mathbb{R}$
 as the bound includes infinitesimals.
 While the function $c(x)$ has finite bounds,
 $c_{0} \lt c(x) \lt c_{1}$,
 these functions do not need to converge.
 E.g. $f(x)=\mathrm{sin}\,x|_{x=\infty} \prec \infty$, $f(x)|_{x=\infty} \prec \infty$.  
 At times such functions behave similarly to constants.
 However an infinitesimal when
 realized is $0$ and not positive,
 hence the need to exclude infinitesimals from a finite positive bound definition.
\bigskip
\begin{defy}\label{DEF202}
A variable has a ``finite positive bound''
 if $a \lt x \lt b$, $\{a, b \} \in \mathbb{R}^{+}$.
\end{defy}
\bigskip
\begin{mex}\label{MEX061}
 $(\frac{1}{n}, 1)|_{n=\infty}$ is not
 a finite positive bound
 as the infinitesimal $\frac{1}{n}|_{n=\infty} \notin \mathbb{R}^{+}$, $\frac{1}{n}|_{n=\infty}$ is not finite.
 Further when realizing the infinitesimal,
 the interval is not all positive, as it includes $0$.
 $(\frac{1}{n}, 1)|_{n=\infty} = [0,1)$ 
 Similarly $\mathrm{sin}\frac{1}{n} = \frac{1}{n}|_{n=\infty}=0$
 does not have finite positive bound.
\end{mex}
\bigskip
\begin{prop}\label{P074}
If $f \asymp g$ then $\frac{f}{g}$ has a finite positive bound. 
\end{prop}
\begin{proof}
 From Definition \ref{DEF010} if $f \asymp g$ then $f \preceq g$ and $f \succeq g$.
 $M, M_{2} \in \mathbb{R}^{+}$;
 from Definition \ref{DEF202}
 if $f \succeq g$ then $M f \geq g$.
 $f \preceq g$ then $f \leq M_{2}g$,
 $\frac{f}{M_{2}} \leq g \leq M f$,
 $\frac{1}{M_{2}} \leq \frac{g}{f} \leq M$.
 Inverting,
 $M_{2} \geq \frac{f}{g} \geq \frac{1}{M}$.
\end{proof}
Hardy in \textit{Orders of Infinity} \cite[p.4]{ordersofinfinity}
 states several theorems with the much greater than relations
 and their transitivity. The infinitesimal numbers
 developed earlier can be used as a tool to prove these theorems. 
 When proving,
 without loss of generality, consider positive functions.

While Hardy states the reader will be able to
 prove the theorems without difficulty,
 here a new number system is used for that purpose. 
\bigskip
\begin{prop}\label{P032}  
$f \succ \phi$, $\phi \succeq \psi$,
 then $f \succ \psi$
\end{prop} 
\begin{proof} 
  $\delta \in \Phi^{+}$; 
  \begin{align*}
    f \succ \phi \text{ then } \phi = \delta f \tag{\text{Definition } \ref{DEF006}} \\
    \phi \succeq \psi \text{ then } M \phi \geq \psi  \tag{\text{Definition } \ref{DEF005}} \\
    M \delta f \geq \psi \tag{\text{Redefine $\delta$ to absorb $M$}} \\
    \delta f \geq \psi  \\
    \delta \geq \frac{\psi}{f} \tag{\text{Proposition } \ref{P026}} \\
    f \succ \psi 
  \end{align*} 
\end{proof}
\bigskip
\begin{defy}
Let $\neg$ be the negation operation, and $z$ the binary 
 relation. $\neg \neg (f \; z \; g) = (f \; z \; g)$.
\[ \neg (f \; z \; g) \; = \; ( f \; (\neg z) \; g) \]
\end{defy}
\begin{mex} Examples of negation in $\mathbb{R}$ or $*G$. 
 $(\neg \lt) \; = \;\; \geq$.
 $(\neg ==) \; = \;\; \neq$.
\end{mex}
\bigskip
\begin{theo}\label{P030}
$a, b, c \in *G\backslash\{0\}$; 
 If $a b \neq c$ then $a \neq c b^{-1}$
\end{theo}
\begin{proof}
$(ab \neq c)$
 $= \neg(ab = c)$
 $= \neg(a = c b^{-1})$
 $= (b \neq c b^{-1})$
\end{proof}
\bigskip
\begin{prop}\label{P027}
$f \succeq \phi$ implies the negation of $f \prec \phi$, $\neg(f \prec \phi)$.
 However $\neg(f \prec \phi)$ does not imply $f \succeq \phi$.
\end{prop}
\begin{proof}
 Without loss of generality, consider positive $f$ and $\phi$.
 Since $f \succeq \phi$, 
 $\exists M: M \in \mathbb{R}^{+}$ then $M f \geq \phi$, $M \geq \frac{\phi}{f}$.
 Since $\frac{\phi}{f}$ is positive and bounded above, 
 and $\frac{\phi}{f} \in \mathbb{R}^{+}+\Phi$.

 If we consider the negation relation,
 $\delta \in \Phi$,
 $\neg (f \prec \phi)$,
 $\neg ( \frac{f}{\phi} = \delta)$
 $\frac{f}{\phi} \neq \delta$, which excludes
 infinitesimals, but not infinities.
 Since $\frac{f}{\phi}$ is positive, then
 expressed as an interval $\frac{\phi}{f} \in [ \mathbb{R}^{+}+\Phi, +\Phi^{-1}]$.
 
 We can see the first interval is a subinterval in the second, hence implication
 is confirmed, but that the second interval is not contained in the first,
 then the `not implied' confirmed.
\end{proof}
\subsection{Scales of infinity} \label{S0204}
In music a scale ordered by increasing pitch is an ascending scale,
 while descending scales are ordered by decreasing pitch. 
 Indeed everyone has heard musicians going through
 the scales in rehearsal before a performance.

In an analogous way mathematics has its scales where
 families of functions ascend and descend. 
 Because of a property of the numbers zero and infinity,
 the scales are defined at these points, giving
 a number system at zero and at infinity.

 Since these scales are intimately involved with
 the evaluation of a function at a point(extended sense), the
 scales apply to any function evaluation.
 A simple example is that when $a \neq 0$, $x^{2}|_{x=a}$
 $= (x+a)^{2}|_{x=0}$
 $= x^{2} + 2ax + a^{2}|_{x=0}$, we 
 also see $x^{2} \prec 2ax \prec a^{2}|_{x=0}$, 
 correlates to the scale $x^{2} \prec x \prec 1|_{x=0}$

Hardy discusses in detail
 the rates of growth
 of functions, and compares
 two functions
 where
 different functions 
 could be ordered. 
 Hence I believe the title
 of his book is fittingly ``\textit{Orders of Infinity}".
 A new function can always be inserted
 between 
 two ordered functions.
Different
 families of functions possess
 different orderings.
 This 
 is similar to
 the real number system
 where we can always find
 a number between two other numbers.

The notion of numbers
 being  much greater than ($\succ$)
 or much smaller than ($\prec$)
 other numbers
 makes sense for numbers
 that are infinitely large
 or infinitely small.
 
Consider the family of functions $x^{k}$ as $x \to \infty$.
Moving away from the origin,
 each function with increasing 
 exponent $k$,
 becomes steeper.

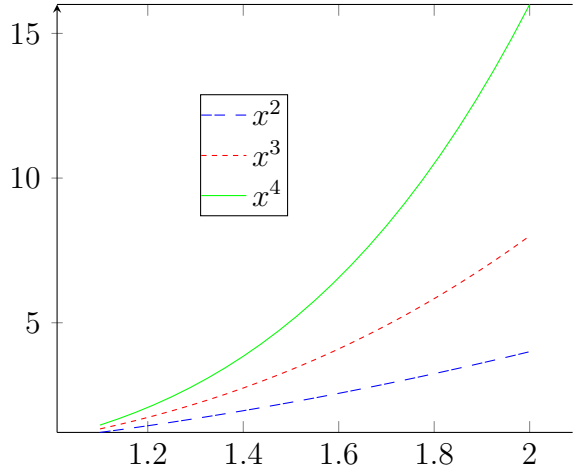
\begin{figure}[h]
\centering
\begin{tikzpicture}[domain=1.1:2.0]
\begin{axis}
  [
    axis y line=center,
    legend style={inner xsep=0pt, inner ysep=0pt, at={(2.5cm,4.5cm)}},
  ]
  \addplot[blue, dash pattern=on 4pt off 1pt on 4pt off 4pt]{ x^2 };
  \addlegendentry{$x^{2}$}
  \addplot[red, dash pattern=on 2pt off 2pt on 2pt off 2pt]{ x^3 };
  \addlegendentry{$x^{3}$}
  \addplot[green]{ x^4 };
  \addlegendentry{$x^{4}$}
\end{axis}
\end{tikzpicture}
\caption{powers at infinity} \label{FIG01}
\end{figure}

 From the considerations of infinitely large functions, relational scales can  be developed.
 The progression of these functions form a scale
 of higher infinities.

Let $x \to \infty$
 then $x^{2}/x \to \infty, x^{3}/x^{2} \to \infty, \ldots$

By defining a measure of the magnitude
 as the absolute value of
 the ratio between two functions
 at infinity, the scales of infinity are more easily expressed. 
 See the $\succ$ relation (Definition \ref{DEF006}).

$(\ldots \succ x^{3} \succ x^{2} \succ x^{1} \succ x^{0} \succ x^{-1} \succ \ldots)|_{x=\infty}$

This relation is conveniently
 symmetrical such that by swapping
 the function's sides, the arrow
 reverses in direction in the same
 way $3 \lt 5$ becomes
 $5 \gt 3$. \\ 
 $(x \prec x^{2} \prec x^{3} \prec \ldots)|_{x=\infty}$

Other examples:  
$(e^{x} \prec e^{e^{x}} \prec e^{e^{e^{x}}} \prec \ldots)|_{x=\infty}$  
 and importantly the logarithmic 
 scale  
$(n \succ \mathrm{ln}\,n \succ \mathrm{ln}\,\mathrm{ln}\,n \succ \mathrm{ln} \,\mathrm{ln}\,\mathrm{ln}\,n \succ \ldots)|_{n=\infty}$
\bigskip
\begin{defy}\label{DEF038}
Let $k$-powers of $e$ 
 be represented by $e_{k}(x)$:
 $e_{0}(x)=x$, 
 $e_{k}(x) = e^{e_{k-1}(x)}$
\end{defy}
\bigskip
\begin{defy}\label{DEF012}
Let $k$-nested natural logarithms
 be represented by $\mathrm{ln}_{k}(x)$:  
 $\mathrm{ln}_{-1}\,x = 1$,  
 $\mathrm{ln}_{0}\,x = x$,  
 $\mathrm{ln}_{k}\,x = \mathrm{ln}( \mathrm{ln}_{k-1}\,x)$
\end{defy}
\bigskip
\begin{defy}\label{DEF039}
As a convention, when $\mathrm{ln}_{k}$ has no argument,
 we define $\mathrm{ln}_{k} = \mathrm{ln}_{k}\,n$.
\end{defy}

Consider $\mathrm{ln}_{k}|_{n=\infty}$. 
If $n$ reaches infinity before the $k$-nested log functions,  
 $\mathrm{ln}_{k} = \infty$ is guaranteed. Looking
 at this another way, let $k$ be finite. 
 This avoids the possibility of the logarithm becoming negative
 or complex.
 Each of these infinities belongs to a family scale.
 $( \mathrm{ln}\,x \succ \mathrm{ln}_{2}\,x \succ \mathrm{ln}_{3}\,x \succ \ldots )|_{x=\infty}$
\bigskip
\begin{conjecture}\label{C001}
Given $f(x)|_{x=\infty}=\infty$, $k=\infty$,  $\mathrm{ln}_{k} f(x) = \infty$ when $x$ reaches
 infinity before $k$.
\end{conjecture}

While Conjecture \ref{C001} is usually
 expressed as a definition,
 the possibility of an ordering
 of variables at infinity should 
 be expected, and 
 this may provide much further investigation.
 A variable reaching infinity before another variable
 could better explain
 partial differential equations, 
 where other variables are held constant,
 and the target variable differentiated.
\bigskip
\begin{defy}\label{DEF013}
Given initial relation $\phi_{1} \succ \phi_{2}$, and function $\phi: \phi_{n+1} = \phi( \phi_{n})$ with the property $\phi_{n} \succ \phi_{n+1}$,
 the relations
$( \phi_{1} \succ \phi_{2} \succ \ldots \succ \phi_{n} \succ \ldots) |_{n=\infty}$,
 are referred to
 as `scales of infinity' \cite[p.9]{ordersofinfinity}.
 Similarly with the much-less-than relation $\prec$.
\end{defy}
\bigskip
With the definition of much less than and much greater than,
 multiplying the scale by constants has no effect.
\bigskip
\begin{prop}\label{P017}
$f, g \in *G$;
 $\,\alpha_{1}, \alpha_{2} \in \mathbb{R} \backslash \{0\}$;
 then  $f \succ g \Leftrightarrow \alpha_{1} f \succ \alpha_{2} g$
\end{prop}
\begin{proof}
$f \succ g$ then $\frac{g}{f} = \Phi$,
 $\frac{\alpha_{2} g}{f} = \alpha_{2} \Phi = \Phi$,
 $\frac{\alpha_{2} g}{\alpha_{1}f} = \frac{1}{\alpha_{1}} \Phi = \Phi$, 
 then $\alpha_{1} f \succ \alpha_{2} g$
\end{proof}
\bigskip
\begin{cor}
\[ \text{Given } (\phi_{1} \succ \phi_{2} \succ \ldots)|_{n=\infty} \]
\[ \text{If } a_{n} \in \mathbb{R}\backslash \{0\} \text{, then } (a_{1} \phi_{1} \succ a_{2} \phi_{2} \succ a_{3} \phi_{3} \ldots)|_{n=\infty} \]
\end{cor}
\begin{proof}
Apply Proposition \ref{P017} to each relation.
\end{proof}

Infinitely small magnitude scales
 can likewise be considered.
 For the powers of $x^{k}$ this has the effect of reversing
 the relation when evaluating $x$ at $0$. \\
$(\ldots \prec x^{4} \prec x^{3} \prec x^{2} \prec \ldots )|_{x=0}$

\begin{figure}[H]
\centering
\begin{tikzpicture}[domain=0.001:0.08]
\begin{axis}
  [
    axis y line=center,
    legend style={inner xsep=0pt, inner ysep=0pt, at={(2.5cm,4.5cm)}},
  ]
  \addplot[blue, dash pattern=on 4pt off 1pt on 4pt off 4pt]{ x^2 };
  \addlegendentry{$x^{2}$}
  \addplot[red, dash pattern=on 2pt off 2pt on 2pt off 2pt]{ x^3 };
  \addlegendentry{$x^{3}$}
  \addplot[green]{ x^4 };
  \addlegendentry{$x^{4}$}
\end{axis}
\end{tikzpicture}
\caption{powers at zero} \label{FIG02}
\end{figure}
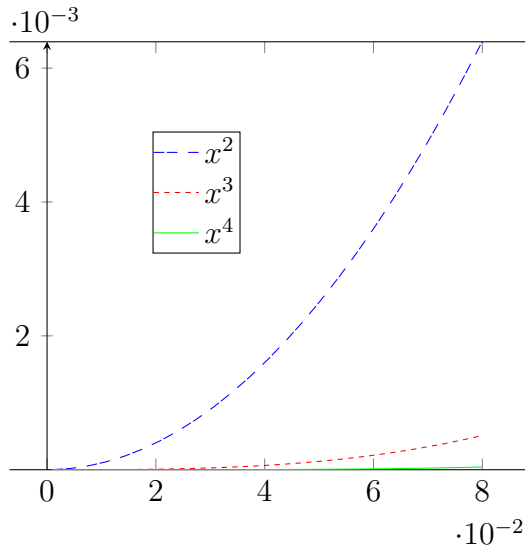

The scales of infinity are often used implicitly in
 calculations.  For example, truncating the Taylor series,
 or with limit calculations by ignoring the infinitesimals
 which
 effectively sets the infinitesimals to zero.

 Since scales of infinity describe infinitesimals, calculus
 can be constructed with these ideas. 
 A most useful scale in algebraic simplification orders
 different families of curves, whereby different types
 of infinitesimals and infinities are compared.

\begin{table}[H]
  \centering
  \begin{tabular}{c}
 $(c \prec \mathrm{ln}(x) \prec x^{p}|_{p \gt 0}  \prec a^{x}|_{a \gt 1} \prec x! \prec x^{x} )|_{x=\infty}$ \\
 $( \ldots \prec x^{-2} \prec x^{-1} \prec 1 \prec x \prec x^{2} \prec \ldots )|_{x=\infty}$ \\
 $( \ldots \succ x^{-2} \succ x^{-1} \succ 1 \succ x \succ x^{2} \succ \ldots )|_{x=0}$ \\
 $( \ldots \prec e^{e^{-x}} \prec e^{-x} \prec 1 \prec e^{x} \prec e^{e^{x}} \prec \ldots )|_{x=\infty}$ \\ 
 $(x \succ \mathrm{ln}\,x \succ \mathrm{ln}_{2}\,x \succ \ldots)|_{x=\infty}$ \\
 $(v \prec \mathrm{ln}\,v \succ \mathrm{ln}_{2} \,v \succ \mathrm{ln}_{3}\, v \succ \ldots)|_{v=0^{+}}$ \\
 $(h f' \succ \frac{h^{2}}{2!} f^{(2)} \succ \frac{h^{3}}{3!} f^{(3)} \succ \ldots )|_{h=0^{+}} \text{ when } f^{(k)} \prec \infty$
  \end{tabular}
  \caption{Summary of scales} \label{FIG03}
\end{table} 

\subsection{Little-o and big-O notation} \label{S0206}
Since infinitary calculus
 has equivalent definitions
 for little-o and big-O notation,
 it can be used
 to do the same things.
 It can describe function growth,
 compare functions, and derive
 theorems.  

Where little-o and big-O notation
 surpasses the
 infinitary calculus notation,
 we see both notations as complementary. 
 In particular, 
 the Landau notation's strength is that it contains the
 relation as an 
 end term to a formula. That is, a relation
 is packaged and managed as a variable.
\[ e^{x} = 1 + x + \frac{1}{2}x^{2} + O(x^{3}) \]
The infinitary calculus symbols are not ``side dependent", 
 $f(x) \succ g(x)$ is the same as $g(x) \prec f(x)$, which
 can give the algebra 
 a sense of freedom. The Landau notation introduced
 $\omega(x)$ and $\Omega(x)$ to express the relations on the
 ``other side", see Definition \ref{DEF008}.

Before proceeding, properties of 
 the magnitude relations $\{ \prec, \preceq, \succ, \succeq \}$
 are derived using $*\overline{G}$, thus demonstrating its usefulness
 in proofs. These properties are then used
 to prove theorems with little-o and big-O, demonstrating an
 equivalence with the magnitude relations.

 To simplify the proofs, 
 from Proposition \ref{P007}, 
 we can make the arguments positive.
 Therefore with
 assumptions regarding $a$ and $b$,
 we can always
 transform the problem to one with sequences
 positive or greater than zero, since these relations are
 not affected by the sign of the elements of the sequence.  

 Since $a$ and $b$ are positive numbers,
 either infinitesimals, infinities, or real numbers except $0$,
 then we can
 multiply and divide $a$ and $b$, before
 realizing the infinitesimal or infinity.
\bigskip
\begin{prop} \label{P005}
 $b \succ a \Leftrightarrow \frac{1}{b} \prec \frac{1}{a}$
\end{prop}
\smallskipneg
\begin{proof}
 $a \neq 0$, $b \neq 0$,
 let $\delta \in \Phi$, $b \succ a$ then
 $\frac{a}{b} = \delta$,
 $\frac{1}{ \frac{a}{b} } = \frac{1}{\delta}$,
 $\frac{ \frac{1}{a}}{ \frac{1}{b} } = \frac{1}{\delta}$,
 $\frac{ \frac{1}{b}}{ \frac{1}{a} } = \delta$,
 $\frac{1}{b} \prec \frac{1}{a}$.
 Similarly if 
 $\frac{1}{b} \prec \frac{1}{a}$ then
 $\frac{ \frac{1}{b}}{ \frac{1}{a} } = \delta$,
 $\frac{a}{b} = \delta$,
 $a \prec b$.
\end{proof}
\bigskip
\begin{prop} \label{P209}
 $b \succ a \Leftrightarrow cb \succ ca$, $c \in *G\backslash\{0\}$
\end{prop}
\smallskipneg
\begin{proof}
 $a \neq 0$, $b \neq 0$,
 let $\delta \in \Phi$, $b \succ a$ then $\frac{a}{b} = \delta$,
 $\frac{ca}{cb} = \delta$, 
 $cb \succ ca$.
 Similarly if
 $cb \succ ca$ then 
 $\frac{ca}{cb} = \delta$, 
 $\frac{a}{b} = \delta$, 
 $b \succ a$.
\end{proof}
\bigskip
\begin{prop} \label{P007}
 $a \prec b \Leftrightarrow -a \prec b$
\end{prop}
\smallskipneg
\begin{proof}
 $a \neq 0$, $b \neq 0$,
 let $\delta \in \Phi$,
  $a \prec b$ then $\frac{a}{b}=\delta$, $\frac{-a}{b}=-\delta$, $-a \prec b$.
  Similarly if $-a \prec b$ then $\frac{-a}{b}=\delta$, $\frac{a}{b}=-\delta$, $a \prec b$.
\end{proof}
\bigskip
\begin{prop} \label{P008}
 $a \succ b \Leftrightarrow a +\lambda \succ b + \lambda$
 when $\lambda \prec a$ and $\lambda \prec b$.
\end{prop}
\smallskipneg
\begin{proof}
 $\frac{a+\lambda}{b+\lambda}$
 $=\frac{a}{b+\lambda}$
 $=\frac{a}{b}$
 $\in \Phi^{-1}$ 
 then
 $a + \lambda \succ b + \lambda$.
 Reversing the argument,
 $\frac{a}{b}$
 $=\frac{a+\lambda}{b}$
 $=\frac{a+\lambda}{b+\lambda}$
 $\in \Phi^{-1}$ then $a + \lambda \succ b + \lambda$.
\end{proof}
\bigskip
\begin{prop} \label{P205}
 $a \succeq b \Leftrightarrow \frac{1}{a} \preceq \frac{1}{b}$
\end{prop}
\smallskipneg
\begin{proof}
 $a \neq 0$, $b \neq 0$,
 $a \succeq b$,
 $\exists \alpha: \alpha |a| \geq |b|$,
 $\alpha \geq \frac{|b|}{|a|}$,
 $\alpha \frac{1}{|b|} \geq \frac{1}{|a|}$,
 $\frac{1}{b} \succeq \frac{1}{a}$,
 $\frac{1}{a} \preceq \frac{1}{b}$.
\end{proof}
\bigskip
\begin{prop} \label{P207}
 $a \succeq b \Leftrightarrow  c a \succeq c b$
\end{prop}
\smallskipneg
\begin{proof}
Let $c \in *G\backslash\{0\}$.
 Consider
 $a \succeq b$,
 $\exists \alpha:$
 $\alpha |a| \geq |b|$,
 $\alpha |c| |a| \geq |c| |b|$,
 $\alpha |c a| \geq |c b|$,
 $c a \succeq c b$
 then 
 $a \succeq b \Rightarrow  c a \succeq c b$,
 reversing the argument gives the implication in the other direction.
\end{proof}
\bigskip 
\begin{prop} \label{P200}
 $a \preceq b \Leftrightarrow -a \preceq b$
\end{prop}
\smallskipneg
\begin{proof}
 Consider
 $a \succeq b$,
 $\exists \alpha:$
 $\alpha |a| \geq |b|$,
 $\alpha |-a| \geq |b|$,
 $-a \succeq b$,
 then 
 $a \succeq b \Rightarrow  -a \succeq b$,
 reversing the argument gives the implication in the other direction.
\end{proof}
\bigskip
\begin{prop} \label{P004}
 $\lambda \prec \infty, \;$
 $a \succeq b \Rightarrow a+\lambda \succeq b+\lambda$
\end{prop}
\smallskipneg
\begin{proof}
 $a \succeq b$,
 $\exists \alpha:$
 $\alpha |a| \geq |b|$,
 $\alpha |a| + \lambda \geq |b| + \lambda$,
 Assume $\alpha \gt 1$ as we can always increase $\alpha$.
 Case $\lambda \gt 0$, 
 $\alpha |a| + \alpha \lambda \geq |b| + \lambda$,
 $\alpha |a + \lambda| \geq |b| + |\lambda| \geq |b+\lambda|$,
 $a+\lambda \succeq b+\lambda$.
 Case $-\lambda$ then
 $\alpha |a| - \lambda \geq |b| - \lambda$,
 $\alpha |a| + \lambda \geq |b| + \lambda$,
 the above positive case.
 Hence 
 $a \succeq b \Rightarrow a+\lambda \succeq b+\lambda$ 
\end{proof}
\bigskip
\begin{mex}\label{MEX007} 
For proving the following big-O theorem 
 we found infinitary calculus
 to be easier to reason with than the solution given in
 \cite[Theorem 2.(8)]{littleobigO}
 . 
\end{mex}
\smallskipneg
\[ \text{If} \; g(x) = o(1) \text{ then }  
  \frac{1}{1+O(g(x))} = 1 + O(g(x)) \]
\smallskipneg
\begin{proof}
  Let $v(x)=O(g(x))$
  \begin{align*}
    1+v(x) \succeq 1
    \\ \frac{1}{1+v(x)} \preceq 1  \tag{from Proposition \ref{P205} } 
    \\ \frac{v(x)}{1+v(x)} \preceq v(x) \tag{from Proposition \ref{P207} }
    \\ \frac{-v(x)}{1+v(x)} \preceq v(x)   \tag{from Proposition \ref{P200} } 
    \\ \frac{-v(x)-1+1}{1+v(x)} \preceq v(x) 
    \\ \frac{1}{1+v(x)} -1 \preceq v(x) 
    \\ \frac{1}{1+v(x)} \preceq 1+v(x)   \tag{from Proposition \ref{P004} } 
    \\ \frac{1}{1+O(g(x))} = 1+O(g(x)) &&\qedhere
  \end{align*}
\end{proof}
\smallskipneg
Verification: rather than building the inequality,
 the inequality can be verified directly.  
$\frac{1}{1+v(x))} \preceq 1+v(x)$,
$\frac{1}{1} \preceq 1+v(x)$ is true,
 since $1+v(x)|_{x=a}=1$ and $v(x) \preceq g(x) \prec 1|_{x=a}$ 
\bigskip
\begin{mex}\label{MEX006} 
Consider the proof of the
 following theorem from
 \cite[Theorem 2.(8)]{littleobigO}
 .
\[ \text{If} \; g(x) = o(1) \text{ then } 
  \frac{1}{1+o(g(x))} = 1 + o(g(x)) \]
Using an inequality in infinitary
 calculus to prove the theorem. 
Let $h(x)=o(g(x))|_{x=a}$,  
 $h(x) \prec g(x)|_{x=a}$,  
$1 + h(x) \succeq 1|_{x=a}$,  
$\frac{1}{1+h(x)} \preceq 1|_{x=a}$, 
$\frac{h(x)}{1+h(x)} \preceq h(x)|_{x=a}$, 
$\frac{-h(x)}{1+h(x)} \preceq h(x)|_{x=a}$, 
$\frac{-h(x)-1+1}{1+h(x)} \preceq h(x)|_{x=a}$, 
$\frac{1}{1+h(x)}-1 \preceq h(x)|_{x=a}$, 
$\frac{1}{1+h(x)}-1 \prec g(x)|_{x=a}$,  
$\frac{1}{1+h(x)}-1 = o(g(x))|_{x=a}$,  
$\frac{1}{1+o(g(x))} = 1 + o(g(x))$ 

From
 \cite{littleobigO}
 the theorem is derived in the standard way
 by taking the limit.  
Applying the little-o definition directly.  
 $\lim\limits_{x \to a} \frac{ \frac{1}{1+h(x)} - 1}{g(x)}$
$= \lim\limits_{x \to a} \frac{ \frac{1 - (1+h(x))}{1+h(x)}}{g(x)}$
$= -\lim\limits_{x \to a} \frac{h(x)}{g(x)} \frac{1}{1+h(x)}$
$= -\lim\limits_{x \to a} \frac{h(x)}{g(x)} \frac{\frac{1}{g(x)}}{\frac{1}{g(x)}+\frac{h(x)}{g(x)}}$
$= -\lim\limits_{x \to a} \frac{h(x)}{g(x)} \frac{1}{1+\frac{h(x)}{g(x)} g(x)}$
$= -\lim\limits_{x \to a} 0 \cdot \frac{1}{1 +0 \cdot g(x)}$
$=0$

The same calculation with infinitary calculus 
 evaluation at the point and applying the definition.  
 $\frac{ \frac{1}{1+h(x)} - 1}{g(x)}|_{x=a}$
$= \frac{ \frac{1 - (1+h(x))}{1+h(x)}}{g(x)}|_{x=a}$
$= \frac{ \frac{-h(x)}{1+h(x)}}{g(x)}|_{x=a}$
$= -\frac{h(x)}{g(x)} \frac{1}{1+h(x)}|_{x=a}$
$= -\frac{h(x)}{g(x)}|_{x=a}$
$=0$ as $1+h(x)|_{x=a}=1$ and $h(x) \prec g(x)|_{x=a}$
\end{mex}

The infinitary calculus expresses scales of infinities with
 more intuitive meaning.
 Writing the scales with big-O notation, using
 the left side to right side definition,
 big-O notation is defined where $f = O(g)$
 is not the same as $O(g)=f$.
 Let $a \gt 0$, $b \gt 0$, $k \gt 0$.
 We can write 
 $O(e^{-ax}) = O(x^{-b}) = O(\mathrm{ln}(x)^{-k})$
 which has a left-to-right definition of $O()$ and 
 express in infinitary calculus as  
 $e^{-ax} \preceq x^{-b} \preceq \mathrm{ln}(x)^{-k}|_{x=\infty}$.

\section{Comparing functions} \label{S03}
 An algebra for comparing functions at infinity
 with infinireals, comprising of infinitesimals
 and infinities, is developed:
 where the unknown relation is solved for.
 Generally, we consider positive monotonic functions $f$ and $g$,
 arbitrarily small or large,
 with relation $z$: $f \; z \; g$. 
 In general we require 
 $f$, $g$, $f-g$ and $\frac{f}{g}$ to be
 ultimately monotonic.

\subsection{Introduction} \label{S0301}
 In extending du Bois-Reymond's theory,
 we have discovered a new number system Part 1,
 and used this to rephrase du Bois-Reymond's much greater
 than relations Part 2. 
 However, at the heart of Reymond's theory
 is the comparison of functions.

 Today this may seem of little interest because
 there are no applications which directly require this.
 Even Hardy, through writing and
 extending 
 du Bois-Reymond's work \cite{ordersofinfinity} 
 thought this.
 Others better incorporated the theory: 
 little-o and big-O notation have the same definitions
 as relational operators $\{ \prec, \preceq \}$
 and similarly 
 other relations Part 2.

 Instead, du Bois-Reymond's work became a catalyst
 for other higher mathematics
 and itself as an operational calculus was
 largely forgotten.
 In this era 
 of immense change and other 
 issues which they faced, this is not surprising.
 For just one example, the theory on divergent sums and
 functions was being established.

 Our aim is to open the field of infinitary calculus
 through the development of another
 infinitesimal and infinitary calculus that
 is 
 derived from du Bois-Reymond's work.
 The method solves for relations between functions.
 Subsequent papers giving  
 applications for comparing functions are found.
 (E.g. sum convergence \cite[Convergence sums]{cebp2})

 Comparing functions is the key discovery in the theory's
 development. 

 With  the method of comparing functions at infinity
 described in this paper,
 we believe there is a significant improvement over
 the method described by Hardy,
 which either computes
 the relation with a limit, 
 which is fine, 
 or uses logarithmico-exponential scales
 \cite[pp.31--33]{ordersofinfinity}. 

 By comparing functions at infinity,
 du-Bois Reymond showed the existence of
 curves infinitely close to each other.
 For example, the construction
 of infinitely many curves infinitesimally
 close to a straight line (see Example \ref{MEX019}).
 Thereby demonstrating  
 the curves
 exist at infinity.

 However, the number system in which the curves
 reside includes infinitesimals and infinities,
 and hence is a non-standard analysis.

 While du Bois-Reymond did not define a number system
 as Abraham Robinson had done with Non-standard analysis (NSA),
 the constructions prove that such a number system
 exists. 

That Abraham made little reference 
 to du Boise-Reymond's work 
 is unexplained. 
 Though he did use similar notation.
 For example in \cite[p.97]{abraham}: if $a \in \mathbb{J}$ and
 $b \in *J$ then $a \prec b$. Elements of $*J$ he called infinite. 

 As this is a reference paper in the sense
 that it
 contains propositions which
 we have collected,  
 while working through problems. 
 Subsequent papers reference this paper.

 We can consider comparison
 in terms of addition or multiplication. Where $z$ is a relation. 

\begin{table}[H]
  \centering
  \begin{tabular}{|c|c|} \hline
  $f \; z \; g$ & \text{Comparison} \\ \hline
  $f-g \; z \; 0$ & \text{additive sense} \\ \hline
  $\frac{f}{g} \; z \; 1$ & \text{multiplicative sense} \\ \hline
  \end{tabular}
  \caption{Binary relation comparison} \label{FIG05}
\end{table}
\subsection{Solving for a relation} \label{S0302}
While we are very familiar with
 solving for variables as values,
 in general we do not solve variables
 for relations.
 However there is no real reason not to
 do so.

In the course of devising an alternative
 way to compare functions at infinity,
 a new way of comparing functions
 has been developed, where the primary
 focus is to solve for the relation.
\bigskip 
\begin{defy}\label{DEF017}
Let $f(x) \; z \; g(x)$ 
 be a comparison of the functions
 $f(x)$ and $g(x)$
 where $z$ is the variable relation.
\end{defy}

When possible, we could then solve for a relation $z$,
 for example
$z \in \{ \lt, \leq \prec, \preceq, \gt, \geq, \succ, \succeq , =, \simeq, \not =, \prec\!\prec, \succ\!\succ \}$.
\bigskip
\begin{defy}\label{DEF018}
Given $f(x) \; z \; g(x)$, with 
 relation $z$, then applying function
 or operator $\phi$ to one or both sides of the variable relation $z$ 
 results in a new relation $(\phi(z))$. 
\end{defy}

The brackets about the relation are
 an aid  to distinguish 
 the relation as a variable.
\[  f(x) \; z \; g(x) \text{, } \;\;
\phi(f(x)) \; (\phi(z)) \; \phi(g(x)) \]
The function is also applied to the middle
 as changing either $f$ or $g$ can change
 the relation.  
For example,
 applying an exponential function to both sides
 and the middle
 gives $e^{f(x)} \; (e^{z}) \; e^{g(x)}$,
 where 
 $(e^{z})$ is the new relation.
 Applying the logarithm 
 function 
 to all parts,
 $\mathrm{ln}\,f(x) \; (\mathrm{ln}\,z) \; \mathrm{ln}\,g(x)$,
 where 
 $(\mathrm{ln}\,z)$ is the new relation.
 Applying differentiation to all parts,
 $D f(x) \; (D z) \; D g(x)$,
 where $(D z)$ is the new relation.
 $D$ is a shorthand
 operator for differentiation $\frac{d}{dx}$.

 Differentiating or integrating
 positive monotonic functions, with
 the condition that their difference is monotonic,
 preserves the $\lt$ and $\leq$ relations: 
 $D f(x) \; (D z) \; D g(x) \Leftrightarrow f(x) \; z \; g(x)$
 (see Part 6), and by notation, 
 relations  
 $(D z) = (z)$, $(\int z) = (z)$.
\bigskip
\begin{defy}\label{DEF033}
\smallskip
We say $z \in \mathbb{B}$ to mean that $z$ is a binary
 relation.
\end{defy}
\bigskip
\begin{defy}\label{DEF032}
When $f \; ( \phi(z) ) \; g$ 
 and  $(\phi(z))  = z_{2}$,
 where 
 ; $z, z_{2}, (\phi(z)) \in \mathbb{B}$;
 we may propose $\phi(z) = z_{2}$, provided
 that there is no contradiction. 
\end{defy}

In practice, when solving for a relation, we
 presume such a relation exists and then proceed
 to solve for it. The method of brackets about
 the relation is a label of applied operations.
 If this is reversible, a solution exists to unravel
 the said operations.
 Definition \ref{DEF032} allows you to proceed with
 the solution process without having to formerly say
 so. 
\[ \text{If } (D^{n}z) = \; \succ \text{ then solve } D^{n}z = \; \succ \]

\begin{mex}
$2 x^{2} \; z \; 5x|_{x=\infty}$, 
 $4x \; (D\,z) \; 5|_{x=\infty}$,  
 $(D\,z) = \; \succ$, removing the brackets
 and solving for $z$, 
 $D\,z = \; \succ$,
 $\int \! D\,z = \; \int \! \succ$,
 $z = \int \! \succ \; = \; \succ$ (see Table \ref{FIG04}). 
\end{mex}
\bigskip
\begin{defy}\label{DEF019}
 Let a ``finite relation" be a relation without consideration of
 infinitesimal or infinitary
 arithmetic.
\end{defy}
\bigskip
\begin{mex}\label{MEX052}
 The relations
 $\forall n \gt n_{0}: n+1 \gt n$,
 for positive $n$,
 then  
 $\frac{1}{n} \gt 0$
 is not finite
 relations, as they include infinite arithmetic.
 No lower bound exists. The infimum (greatest lower bound)
 $0$ exists,
 but is another type of number.
 Positive $n$ has no upper bound either,
 but has an infimum $\infty$. The ``bounds" do exist,
 but involve infinite arithmetic, and in a sense are numbers
 of another dimension.
\end{mex}
\bigskip
\begin{defy}\label{DEF020}
 Let a
 relation that is not
 finite
 have infinitesimal or infinite arithmetic.
\end{defy}

 Any operation on the relation produces
 a new comparison, or new $z_{k}$,
 however such a system allows us to solve
 for the initial $z$, 
 through the myriad of possibilities.
\bigskip
\begin{mex}\label{MEX049} Compare in a multiplicative-sense $n^{2}$ and $n$.
 See Figure \ref{FIG06}.
 In solving for $\succ$,  
 the relations $z_{1}$, $z_{2}$, $z_{3}$ 
 did not change
 after division,
 $(\ldots, n^{2} \succ n, n \succ 1, 1 \succ n^{-1} \ldots)|_{n=\infty}$,
 by $b \succ a \Leftrightarrow cb \succ ca$ Proposition \ref{P209}.
\begin{figure}[H]
\centering
\begin{tikzpicture}
\node [] (a11){$n^{2} \; z_{1} \; n$}
  child 
  {
    node [] (a00) {$n \; z_{2} \; 1$ } 
    child 
    {
      node [] (a01) {$1 \; z_{3} \; \frac{1}{n}$}
      child 
      {  
        node [] (a12) {$1 \succ \frac{1}{n}$}
      } 
      child 
      {
        node [] (a04){$\mathbb{R}:\,1 > 0$}
      }
    }
    child 
    {
      node [] (a05) {$n \succ 1$}
    }
  }
  child 
  {
    node [] (a08) {$n^{2} \succ n$}
  }
;
\node at (-3.0,-2.1) (n009) {\small $\div n$};
\node at (-1.3,-0.6) (n008) {\small $\div n$};
\node at (-4.6,-3.8) (n007) {\small $\frac{1}{n} \in \Phi$ };
\node at (-1.9,-3.8) (n006) {\small $\frac{1}{n}|_{n=\infty}=0$ };
\node at (2.2,-0.6) (n005) {\small $\frac{n^{2}}{n}|_{n=\infty}=\infty$};
\node at (-0.1,-2.2) (n004) {\small $\frac{n}{1}|_{n=\infty}=\infty$};
;
\end{tikzpicture}
\caption{Example calculation of relations connecting $*G$ and $\mathbb{R}$} \label{FIG06}
\end{figure}
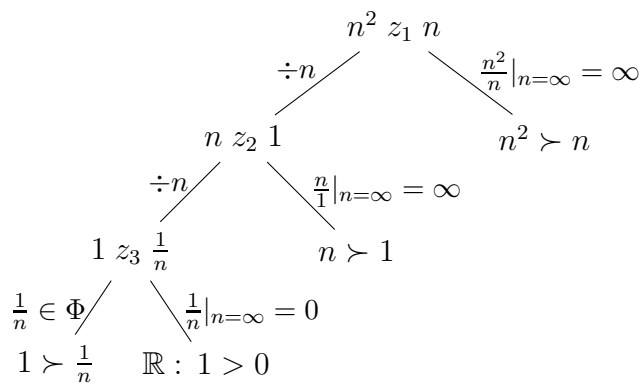
\end{mex}
\bigskip
\begin{defy}\label{DEF021}
 Given relations $z_{1}$ and $z_{2}$, \\
If $f(x) \;\; z_{1} \;\; g(x) \Rightarrow \phi(f(x)) \;\; z_{2} \;\; \phi(g(x))$ then we say $\phi(z_{1}) = z_{2}$. 

If $f(x) \;\; z_{1} \;\; g(x) \Rightarrow e^{f(x)} \;\; z_{2} \;\; e^{g(x)}$ then we say $e^{z_{1}} = z_{2}$. \\
If $f(x) \;\; z_{1} \;\; g(x) \Rightarrow \mathrm{ln}\,f(x) \;\; z_{2} \;\; \mathrm{ln}\,g(x)$ then we say $\mathrm{ln}\,z_{1} = z_{2}$. 
\end{defy}

Use as an aid in calculation
 as a left to right operator when
 solving relations, 
 where all functions in the relations
 are positive,
 $e^{\gt} = \;\, \succ$, 
 $\mathrm{ln} \, \succ \; = \;\, \gt$,  
 $\succ \; = \;\, \gt$.

The following
 compares
 the rate of increase of functions
 in infinitary
 calculus, which we
 find computationally 
 easier than that described
 by Hardy, and at the least
 is an alternate way
 of performing such calculations.

\begin{quote}
The Caterpillar was the first to speak.
 ``What size do you want to be?" it asked.
 ``Oh, I'm not particular as to size," Alice hastily replied; 
 ``only one doesn't like changing so often, you know."
 ``I don't know," said the Caterpillar.
 \cite[pp 72--73]{alice}
\end{quote}

After eating the mushroom in
 her right hand, Alice was shrunk,
 and eating from the left hand she was magnified.

 In an analogous way, by applying powers
 and logarithms we can magnify or shrink
 aspects of the function comparison. Combined
 with non-reversible arithmetic Part 5
 (for example $n^{2}+n = n^{2}|_{n=\infty}$),  we can solve
 for the relation.

 Powers and logarithms are
 mutual inverses.
 While logs of different bases can undo 
 any powers, the natural logarithm $\mathrm{ln}$ and $e$ 
 are the most useful.
 In solving
 relations, it is often convenient to apply these
 functions to both sides of a relation, in a similar
 manner to solving equations. Then
 apply infinitary arithmetic with non-reversible algebra to simplify the relation.

Consider raising both sides of a finite inequality 
 to a power. E.g. 
 $3 \gt 2$, $e^{3} \gt e^{2}$ and the relation symbol did not change.
\bigskip
\begin{mex}\label{MEX008} 
Now consider a relation where both numbers are diverging to
 infinity. 
 For example, $3x \gt 2x|_{x=\infty}$, then $e^{3x} \gt e^{2x}|_{x=\infty}$
 but more importantly $e^{3x} \succ e^{2x}|_{x=\infty}$
 as $e^{3x} / e^{2x} = e^{x}|_{x=\infty}=\infty$.
\end{mex}

If two numbers are positive and one is much larger than
 the other, then the weaker relation that one of the numbers 
  is greater than the other, must be true.
\bigskip
\begin{theo} \label{P009}
 $f=\infty$, $g=\infty$,  
 if $f \succ g$ then $f \gt g$ 
\end{theo}
\begin{proof}
 $f \succ g$ then $\frac{g}{f} = \delta$, $\delta \in \Phi$,
 $g = f \delta$.
 Comparing, 
 $f \; z \; g$, 
 $f \; z \; f \delta$,
 $1 \; z \; \delta$,
 $z = \; \gt$.
\end{proof}

Magnifying a less than or greater
 than relationship magnifies the inequality,
 provided their difference is not finite
 (see Theorems $\ref{P010}$ and $\ref{P201}$).
 Demonstrating this, consider
 a condition with an infinitesimal
 difference, so the inequality exists
 in $*G$, but not $\mathbb{R}$.
\bigskip
\begin{mex}\label{MEX043} 
 Show 
: $f \lt g$ does
 not imply $e^{f} \prec e^{g}$. 
 $\delta \in \Phi$; 
 $g = f + \delta$, $f=\infty$, 
 $f \lt g$, $f \lt f + \delta|_{\delta=0}$,
 $e^{f} \; z \; e^{f + \delta}|_{\delta=0}$, 
 $1 \; z \; e^{\delta}|_{\delta=0}$, 
 $1 \; z \; 1 + \delta + \frac{1}{2} \delta^{2} + \ldots|_{\delta=0}$, 
 $0 \; z \; \delta + \frac{1}{2} \delta^{2} + \ldots|_{\delta=0}$, 
 $0 \; z \; \delta |_{\delta=0}$, 
 $z = \; \lt$, but $z$ is 
 not $\prec$.

If we realize the infinitesimals $*G \mapsto \mathbb{R}$ and
 we have equality. $1 \; z \; 1 + \delta|_{\delta=0}$,
 $\delta \mapsto 0$, $1 \; == \; 1$.
\end{mex}
\bigskip
\begin{theo} \label{P010}
 $f=\infty$, $g=\infty$,  
 If $f \lt g$ and $f-g \prec \infty \; \; \Rightarrow \; e^{f} \lt e^{g}$ 
\end{theo}
\begin{proof}
  $f \lt g$, $0 \lt g-f$, $e^{0} \lt e^{g-f}$,
 $1 \lt e^{g-f}$, $e^{f} \lt e^{g}$
\end{proof}
\bigskip
\begin{theo} \label{P201}
 $f=\infty$, $g=\infty$,  
 if $f \lt g$ and $g-f = \infty \; \; \Rightarrow \; e^{f} \prec e^{g}$ 
\end{theo}
\begin{proof}
  $e^{f} \; z \; e^{g}$,
  $e^{0} \; z \; e^{g-f}$,
  $1 \; z \; e^{\infty}$,
  $z= \;\prec$.
\end{proof}
\bigskip
\begin{theo}\label{P206}
$f=\infty$, $g=\infty$, if $f \prec g$ then $g-f=\infty$
\end{theo}
\begin{proof}
 $\delta \in \Phi$; 
 since $f \prec g$,
 let $f = \delta g$.
 Consider $g-f$
 $= g - g \delta$
 $= g(1-\delta)$
 $\simeq g$
 $=\infty$.
\end{proof}
\bigskip
\begin{theo} \label{P204}
 $f=\infty$, $g=\infty$, 
 if $f \prec g$ then $e^{f} \prec e^{g}$
\end{theo}
\begin{proof}
$f \prec g$,
 $g-f=\infty$.
 $e^{f} \; z \; e^{g}$,
 $1 \; z \; e^{g-f}$,
 $1 \; z \; e^{\infty}$,
 $1 \; z \; \infty$,
 $z = \; \prec$
\end{proof}
\bigskip
\begin{prop} \label{P208}
 $f=\infty$, $g=\infty$,
 if $f \succ g$ then
 $\mathrm{ln}\,f-\mathrm{ln}\,g = \infty$
\end{prop}
\begin{proof}
$f \succ g$ then
 let $f \delta = g$ where 
 $\delta \in +\Phi$;
 $f = \delta^{-1} g$.
 Consider 
 $\mathrm{ln}\,f-\mathrm{ln}\,g$
 $=\mathrm{ln}(\delta^{-1} g)-\mathrm{ln}\,g$ 
 $=\mathrm{ln}\,\delta^{-1} + \mathrm{ln}\,g-\mathrm{ln}\,g$ 
 $=\mathrm{ln}\,\delta^{-1}$ 
 $=\infty$.
\end{proof}
\bigskip
In reducing a large number,
 the log function
 applied to both sides of a relation 
 can decrease the inequality.
\bigskip
\begin{theo} \label{P019}
$f=\infty$, $g=\infty\;\;$ 
  If $f \lt g$ then $\mathrm{ln}\,f \lt \mathrm{ln}\,g$
\end{theo}
\begin{proof}
$f \lt g$, at any point let $g = f^{1+\delta}$, $\delta \gt 0$.
 Compare
 $\mathrm{ln}\,f \; z \; \mathrm{ln}\,g$,
 $\mathrm{ln}\,f \; z \; \mathrm{ln}\,f^{1+\delta}$,
 $\mathrm{ln}\,f \; z \; (1+\delta)\mathrm{ln}\,f$,
 $0 \; z \; \delta \,\mathrm{ln}\,f$,
 $z = \; \lt$
 then 
 $\mathrm{ln}\,f \lt \mathrm{ln}\,g$.
\end{proof}
\bigskip
\begin{theo} \label{P013}
$f=\infty$, $g=\infty\;\;$ 
 If $f \prec g$ then
 $\mathrm{ln}\,f \lt \mathrm{ln}\,g$
\end{theo}
\begin{proof}
$f \prec g$ then $\frac{f}{g} = \delta$; $\delta \in \Phi^{+}$; $f = \delta g$ Since
 $g$ is not an infinitesimal, multiplying by an infinitesimal decreases the
 number, then $f \lt g$. Alternatively $f \; z \; g$, $\delta g \; z \; g$, $\delta \; z \; 1$, $z = \; \lt$
 as an infinitesimal is smaller than any positive number.
 As $\mathrm{ln}$ preserves the relation,
 $f \lt g \Rightarrow \mathrm{ln}\,f \lt \mathrm{ln}\,g$
\end{proof}
In reducing an infinite number with a log from a much
 less than relation does not imply another much 
 less than relation, nor does it exclude it.
\bigskip
\begin{prop}
 $f=\infty$, $g=\infty$, 
 $f \prec g \nRightarrow \mathrm{ln}\,f \prec \mathrm{ln}\,g$
\end{prop}
\begin{proof} Counter example: Example \ref{MEX009}
\end{proof}
\bigskip
\begin{theo}\label{P016}
 $f=\infty$, $g=\infty$, 
 if $f \succ g$ then $Df \succ Dg$.
\end{theo}
\begin{proof}
 $Df \; z \; Dg$,
 $1 \; z \; \frac{Dg}{Df}$,
 but $\frac{Dg}{Df} = \frac{g}{f}$ then
 $1 \; z \; \frac{g}{f}$.
 $g = \delta f$; $\delta \in \Phi^{+}$;
 $1 \; z \; \delta$, $z = \; \succ$
\end{proof}
\bigskip
\begin{mex}\label{MEX009} 
 $e^{3x} \succ e^{2x}|_{x=\infty} \nRightarrow 3x \succ 2x|_{x=\infty}$
 as $\frac{3x}{2x}|_{x=\infty}=\frac{3}{2} \neq \infty$. 
\end{mex}

When asking what happens when we reduce an infinity,
 in a 
 similarly way to magnifying
 the relationship, we can consider the two
 complete cases $g-f \prec \infty$ and $g-f=\infty$,
 thereby showing these conditions
 to be necessary and sufficient
 for determining what happens when reducing the relation.
\bigskip 
\begin{prop}\label{P035}
 $f=\infty$, $g=\infty$,
 $f \propto g \Leftrightarrow \mathrm{ln}\,g-\mathrm{ln}\,f \prec \infty$,
\end{prop}
\begin{proof}
$f \propto g$ then $\frac{f}{g} = \alpha_{n} \prec \infty$,
 $\mathrm{ln}\frac{f}{g} = \mathrm{ln} \,\alpha_{n}$,
 $\mathrm{ln}\,f - \mathrm{ln}\,g \prec \infty$.
\end{proof}
\bigskip
\begin{prop}\label{P202}
 $f=\infty$, $g=\infty$,
 $f \asymp g \Leftrightarrow \mathrm{ln}\,g-\mathrm{ln}\,f \prec \infty$,
\end{prop}
\begin{proof}
 By definition $f \asymp g$ then $f \preceq g$
 and $f \succeq g$ Part 2.
 $a, b \in \mathbb{R}^{+}$;
 If 
 $f \asymp g$ then $a \lt \frac{f}{g} \lt b$.
 Case $a \lt \frac{f}{g}$,
 $a g \lt f$,
 $\mathrm{ln}\,a + \mathrm{ln}\,g \; (\mathrm{ln}\,\lt) \; \mathrm{ln}\,f$
 $\mathrm{ln}\,a \; (\mathrm{ln}\,\lt) \; \mathrm{ln}\,f - \mathrm{ln}\,g$.
 Similarly,
 $\frac{f}{g} \lt b$,
 $f \lt b g$,
 $\mathrm{ln}\,f \; ( \mathrm{ln}\, \lt) \; \mathrm{ln}\,b+ \mathrm{ln}\,g$,
 $\mathrm{ln}\,f-\mathrm{ln}\,g \; (\mathrm{ln}\,\lt ) \; \mathrm{ln}\,b$.
 Putting the two conditions together,
 $\mathrm{ln}\,a \; (\mathrm{ln}\,\lt) \; \mathrm{ln}\,f-\mathrm{ln}\,g \; (\mathrm{ln}\,\lt) \; \mathrm{ln}\,b$. 
 By Theorem \ref{P019}, $\mathrm{ln}\,\lt \; = \; \lt$.
 $\mathrm{ln}\,a \lt \mathrm{ln}\,f-\mathrm{ln}\,g \lt \mathrm{ln}\,b$. 
 Similarly ; $a', b' \in \mathbb{R}^{+}$; then
 $a' \lt \frac{g}{f} \lt b'$ gives finite bounds.
\end{proof}

The development of solving for the unknown relation as a variable
 comes about through
 comparing functions, which includes the calculation of limits,
 see Part 5.
 After developing the theory,
 while investigating infinitesimals 
 in \textit{Orders of Infinity} \cite{ordersofinfinity}, similar problems were found
 and some of du Bois-Reymond's known theorems were rediscovered.
 At this point the alternate calculation
 was already useful, rather than trying to follow Hardy's calculations.

Using the equality symbol as an operator reading
 from left to right, define $z_{1} = z_{2}$,
 from Definition \ref{DEF021} a table of relation implications,
 the equals symbol is interpreted
  from left to right,
 generally leading
 to the right side being a generalization
 from the left. As an operator 
 analogy; $\text{Mercedes} = \text{car}$;
 $\text{BMW} = \text{car}$.
 The right-hand side is the generalisation
 of the left.

When using equality operator $=$ for generalization,
 place the variable relation being solved
 for, on the left side of the
 equals sign. For example, writing $\mathrm{ln}\,z=\;\succ$
 instead of $\succ \; = \mathrm{ln}\,z$.
 Then the generalisation can be combined with solving
 the variable. 
 $\mathrm{ln}\,z=\;\succ$, $e^{\mathrm{ln}\,z} = e^{\succ} = \; \succ$, $z=\;\succ$.
 Examples of exponential and
 log functions:

In $*G$, given $f=\infty$, $g=\infty$, $f \; z \; g$ and $\phi(f) \; (\phi(z)) \; \phi(g)$, then $\phi(z) = z_{2}$.
 The functions are continuous and monotonic in $*G$.

\begin{table}[H]
  \centering
  \begin{tabular}{|c|c|c|c|} \hline
    $\phi(\{ \lt, \prec \})$ & $\phi( \{ \gt, \succ \})$ & \text{Condition} & \text{Reference} \\ \hline
    $e^{\lt} = \; \lt$ & $e^{\gt} = \; \gt$  & $f-g \prec \infty$ & Th. \ref{P010} \\
    $e^{\lt} = \; \prec$ & $e^{\gt} = \; \succ$  & $f-g = \infty$ & Th. \ref{P201} \\
    $\prec \; = \; \lt$ & $\succ \; = \; \gt$ & & Th. \ref{P009} \\
    $e^{\prec} \; = \; \prec$ &  $e^{\succ} \; = \; \succ$ & & Th. \ref{P204} \\ 
    $\mathrm{ln}\!\lt \; = \; \lt$ & $\mathrm{ln}\!\gt \; = \; \gt$ & & Th. \ref{P019} \\
    $\mathrm{ln}\!\prec \; = \; \lt$ & $\mathrm{ln}\!\succ \; = \; \gt$ & & Th. \ref{P013} \\
    $\mathrm{ln}\!\prec \; = \; \prec$ & & $f \prec\!\prec g$ & Part 5 \\
    $\int\!{\prec} \;\! = \; \prec$ & $\int\!{\succ} \;\! = \; \succ$  & \text{Ignore integration constants} & \text{Part 6} \\
    $D{\prec} \; = \; \prec$ & $D{\succ} \; = \; \succ$ & &  Th. \ref{P016} \\ 
    $\int\!\!{<} \;\! = \; <$ & $\int\!{>} \;\! = \; >$ & & Part 6 \\
    $\int\!\!{\leq} \;\! = \; \leq$ & $\int\!{\leq} \;\! = \; \leq$ & & Part 6 \\
    $D{<} \; = \; \lt$ & $D{\gt} \; = \; \gt$ & $Df-Dg$ \text{ is not constant} & \text{Part 6} \\ 
    $D{\leq} \; = \; \leq$ & $D{\geq} \; = \; \geq$ & $Df-Dg$ \text{ is not constant} & \text{Part 6} \\ 
    \hline
  \end{tabular}
  \caption{Relation simplification for positive divergent functions } \label{FIG04}
\end{table}

\bigskip
\begin{mex}\label{MEX010} 
 Decreasing/reducing the infinities  
\begin{align*}
3x \gt 2x|_{x=\infty} \\ 
\mathrm{ln}(3x) \;\; (\mathrm{ln}\,\gt) \;\; \mathrm{ln}(2x)|_{x=\infty} \\ 
\mathrm{ln}\,3 + \mathrm{ln}\,x \;\; (\mathrm{ln}\,\gt) \;\; \mathrm{ln}\,2 + \mathrm{ln}\,x|_{x=\infty} \\ 
\mathrm{ln}\,3 \;\; (\mathrm{ln}\,\gt) \;\; \mathrm{ln}\,2 \\
\mathrm{ln}\,3 \gt \mathrm{ln}\,2 \text{ then } \mathrm{ln} \gt \; = \; \gt
\end{align*}
\end{mex}
\bigskip
\begin{mex}\label{MEX011} 
 Increasing/magnifying the infinities  
\begin{align*}
3x \gt 2x|_{x=\infty} \\ 
e^{3x} \;\; (e^{\gt}) \;\; e^{2x}|_{x=\infty}  \tag{$\frac{e^{3x}}{e^{2x}}|_{x=\infty}=e^{x}|_{x=\infty} = \infty$} \\
e^{3x} \succ e^{2x}|_{x=\infty} \text{ then } 
 e^{\gt} = \; \succ  
\end{align*}
\end{mex}

To show how all this works, 
 take an example problem from
 \cite[p.8]{ordersofinfinity}.
 Solve for $z$ the following, where $\Delta$ is an
 arbitrarily large but fixed value.
\bigskip
\begin{mex}\label{MEX012} 
\begin{align*}
 e^{x} \;\; z \;\; x^{\Delta}|_{x=\infty} \\ 
 \mathrm{ln}\,e^{x} \;\; (\mathrm{ln}\,z) \;\; \mathrm{ln}\,x^{\Delta}|_{x=\infty} \\ 
x \;\; (\mathrm{ln}\,z) \;\; \Delta \, \mathrm{ln}\,x|_{x=\infty} \\ 
 \mathrm{ln}\,z = \; \succ, 
 z = e^{\succ} = \; \succ \\
e^{x} \succ x^{\Delta}|_{x=\infty}
\end{align*}
\end{mex}

At first, raising a relation to a power may seem silly,
 but it is useful when applied as a notational aid in the solution;
 understood as a magnification it makes sense.
 However, the solution is not always unique,  $z = e^{\succ} = \; \gt$
 is true too, as $\succ \; = \; \gt$ with the left-to-right reading. 

The comparison at infinity
 can ignore
 added 
 constants, that is,
 the comparison is with infinite
 elements. 
\bigskip
\begin{prop}\label{P037}
 If $f=\infty$, $g=\infty$, $f \succ \alpha$, $g \succ \beta$, 
 $z \in \{ \gt, \geq, \succ, \succeq \}$ then
\[ f + \alpha \;\; z \;\; g + \beta \;\;\; \Rightarrow \;\;\; f \;\; z \;\; g \]
\end{prop}
\begin{proof}
 $f + \alpha \; z \; g + \beta$,
 $f \; z \; g + \beta$ because ($f+\alpha=f$ as $f \succ \alpha$),
 $f \; z \; g$ as $g \succ \beta$
\end{proof}
\bigskip
\begin{mex}\label{MEX046} 
 Demonstrated by example, another problem from
 \cite[p.8]{ordersofinfinity}.
Given
 $P_{m}(x) = \sum_{k=0}^{m} p_{k} x^{k}$, $p_{k}$ is positive and 
$Q_{n}(x) = \sum_{k=0}^{n} q_{k} x^{k}$, $q_{k}$ is positive. 
 Show $\mathrm{ln}\,\mathrm{ln}\,P_{m}(x) \sim \mathrm{ln}\,\mathrm{ln}\,Q_{n}(x)|_{x=\infty}$
  \begin{align*}
\mathrm{ln}\,\mathrm{ln}\, P_{m}(x) \; z \; \mathrm{ln}\,\mathrm{ln}\, Q_{n}(x)|_{x=\infty} \\
\mathrm{ln}\,\mathrm{ln} \sum_{k=0}^{m} p_{k}x^{k} \; z \; \mathrm{ln} \, \mathrm{ln} \sum_{k=0}^{n} q_{k}x^{k} |_{x=\infty} & \tag{Apply $x^{k} \succ x^{k-1}$, $p_{k}x^{k} + p_{k-1}x^{k-1} = p_{k}x^{k}|_{x=\infty}$} \\ 
\mathrm{ln}\,\mathrm{ln} \, p_{m}x^{m} \; z \; \mathrm{ln} \, \mathrm{ln} \, q_{n}x^{n} |_{x=\infty}  \\ 
\mathrm{ln}( \mathrm{ln} \, p_{m} + m \, \mathrm{ln}\,x) \; z \; \mathrm{ln}(  \mathrm{ln} \, q_{n} + n\, \mathrm{ln}\,x) |_{x=\infty}  \\
\mathrm{ln}( m \, \mathrm{ln}\,x) \; z \; \mathrm{ln}( n\, \mathrm{ln}\,x )|_{x=\infty} \\
\mathrm{ln}\, m  + \mathrm{ln}\, \mathrm{ln}\,x \; z \; \mathrm{ln}\, n + \mathrm{ln}\, \mathrm{ln}\,x |_{x=\infty} \\
\mathrm{ln}\, \mathrm{ln}\,x \; z \; \mathrm{ln}\, \mathrm{ln}\,x |_{x=\infty} & \;\; \tag{ from $\mathrm{ln}_{2}\,x \succ \mathrm{ln}\,m, \; \mathrm{ln}_{2}\,x \succ \mathrm{ln}\,n$} \\
z = \; \sim & \\
  \end{align*}
\end{mex}

 Contrast the above with
 an example where the highest
 order diverging terms are simplified
 (subtracting equally infinite quantities);
 the next highest order diverging 
 terms determine the relation.
\bigskip
\begin{mex}\label{MEX014} 
Solve for $z$, for the comparison 
 $n^{n} n \; z \; e^{n} n! |_{n=\infty}$. 
 $n \, \mathrm{ln}\,n + \mathrm{ln}\,n \; (\mathrm{ln}\,z) \; n + \sum_{k=1}^{n} \mathrm{ln}_{k}|_{n=\infty}$.  
 Given $\sum_{k=1}^{n} \mathrm{ln}_{k} = n \,\mathrm{ln}\,n - n|_{n=\infty}$ 
 [$\sum_{k=1}^{n} \mathrm{ln}_{k}\,n = \int_{1}^{n} \mathrm{ln}_{k}\,n \,dn = n \,\mathrm{ln}\,n - n|_{n=\infty}$]
 then, 
 $n \, \mathrm{ln}\,n + \mathrm{ln}\,n \; (\mathrm{ln}\,z) \; n + n \, \mathrm{ln}\,n - n |_{n=\infty}$, 
 $\mathrm{ln}\,n \; (\mathrm{ln}\,z) \; 0|_{n=\infty}$, 
 $\mathrm{ln}\,z = \; \succ$, $z = e^{\succ} = \; \succ$, 
 $n^{n} n \succ e^{n} n! |_{n=\infty}$ 
\end{mex}

The application of the logarithm has simplified
 the problem from products of functions to sums of functions.
\bigskip 
\begin{mex}\label{MEX015} 
 Consider the following theorems from 
 \cite[pp.300--302]{ordersofinfinity}  
Theorem 7.11. If $a \gt 0, \; b \gt 0$ we have
\[ \lim\limits_{x \to \infty} (\mathrm{ln}\,x)^{b} /x^{a}=0, \; \; \lim\limits_{x \to \infty} x^{b}/e^{ax}=0 \] 
\begin{proof}
 $\mathrm{ln}\,x \prec x|_{x=\infty}$, 
 $\mathrm{ln}\,x \; z \; x^{a}|_{x=\infty}$, 
 $\mathrm{ln}_{2}\,x \; (\mathrm{ln}\,z) \; a \, \mathrm{ln}\,x|_{x=\infty}$, 
 $\mathrm{ln}_{2}\,x \prec a \, \mathrm{ln}\,x|_{x=\infty}$, 
 $z = e^{\prec} = \; \prec$,
 $\mathrm{ln}\,x \prec x^{a}|_{x=\infty}$, 
 $(\mathrm{ln}\,x)^{b} \; z_{2} \;  x^{a}|_{x=\infty}$, 
 $b \, \mathrm{ln}_{2}\,x \; (\mathrm{ln}\,z_{2}) \;  a \, \mathrm{ln}\,x|_{x=\infty}$, 
 $b \, \mathrm{ln}_{2}\,x \prec a \, \mathrm{ln}\,x|_{x=\infty}$, 
 $(\mathrm{ln}\,x)^{b} \prec x^{a}|_{x=\infty}$. 
$(\mathrm{ln}\,x)^{b} /x^{a}|_{x=\infty}=0$
\end{proof}

\begin{proof}
 $e^{x} \succ x|_{x=\infty}$, 
 $e^{x} \;\; z \;\; x^{b}|_{x=\infty}$, 
 $x \;\; (\mathrm{ln}\,z) \;\; b \, \mathrm{ln}\,x|_{x=\infty}$, 
 $x \succ b \, \mathrm{ln}\,x|_{x=\infty}$, 
 $z = e^{\succ} = \; \succ$, 
 $e^{x} \succ x^{b}|_{x=\infty}$, 
 $e^{ax} \;\; z_{2} \;\; x^{b}|_{x=\infty}$, 
 $ax \;\; (\mathrm{ln}\, z_{2}) \;\; b \, \mathrm{ln}\,x|_{x=\infty}$, 
 $ax \succ b \, \mathrm{ln}\,x|_{x=\infty}$, 
 $e^{ax} \succ x^{b}|_{x=\infty}$. 
$x^{b}/e^{ax}|_{x=\infty}=0$
\end{proof}
\end{mex}

Applying infinitary calculus to problems 
 can result 
 in choosing whether to
 use a theorem, or solving 
 by
 calculating directly.
\bigskip
\begin{mex}\label{MEX016} 
Find $\lim\limits_{x \to 0^{+}} x^{\alpha}\, \mathrm{ln}\,x$,
 where $\alpha \gt 0$, 
 $x^{\alpha}\,\mathrm{ln}\,x|_{x=0^{+}}$,
 $x^{-\alpha}\,\mathrm{ln}\,x^{-1}|_{x=\infty}$,
 $\mathrm{ln}\,x^{-1} \;\;  z \;\; x^{\alpha}|_{x=\infty}$, 
 $\mathrm{ln}\,x^{-1} \prec x^{\alpha}|_{x=\infty}$, 
 $x^{-\alpha}\,\mathrm{ln}\,x^{-1}|_{x=\infty}=0$,
 $x^{\alpha}\, \mathrm{ln}\,x|_{x=0^{+}}=0$.

\smallskipneg
 Another way.
 Let 
 $y=x^{\alpha}\,\mathrm{ln}\,x|_{x=0^{+}}$
 $= x^{-\alpha}\,\mathrm{ln}\,x^{-1}|_{x=\infty}$,
 $\mathrm{ln}\,y= \mathrm{ln}(x^{-\alpha}\,\mathrm{ln}\,x^{-1})|_{x=\infty}$
 $= \mathrm{ln}\,x^{-\alpha} + \mathrm{ln}_{2}\,x^{-1}|_{x=\infty}$
 $= \mathrm{ln}\,x^{-\alpha}|_{x=\infty}$
 as $\mathrm{ln}\,x \succ \mathrm{ln}_{2}\,x|_{x=\infty}$,
 $y = x^{-\alpha}|_{x=\infty}=0$
\end{mex}
\bigskip
\begin{mex}\label{MEX017} 
 \cite[p.31]{ordersofinfinity} 
 Compare the rate of increase
 of $f = (\mathrm{ln}\,x)^{ ( \mathrm{ln}\,x)^{\mu} }$
 and $\phi = x^{(\mathrm{ln}\,x)^{-v}}$.
\begin{align*}
f \;\; z \;\; \phi|_{x=\infty}  \\
(\mathrm{ln}\,x)^{(\mathrm{ln}\,x)^{\mu} } \;\; z \;\; x^{(\mathrm{ln}\,x)^{-v}}|_{x=\infty}  \\
(\mathrm{ln}\,x)^{\mu} \, \mathrm{ln}_{2}\,x \;\; (\mathrm{ln}\,z) \;\; (\mathrm{ln}\,x)^{-v} \, \mathrm{ln}\,x|_{x=\infty}  \\
\mu \, \mathrm{ln}_{2}\,x + \mathrm{ln}_{3} \,x \;\; (\mathrm{ln}_{2}\,z) \;\; -v \, \mathrm{ln}_{2}\,x + \mathrm{ln}_{2}\,x|_{x=\infty} \\
(\mu + v) \, \mathrm{ln}_{2}\,x + \mathrm{ln}_{3} \,x \;\; (\mathrm{ln}_{2}\,z) \;\; \mathrm{ln}_{2}\,x|_{x=\infty}
\end{align*}
Case $\mu + v=1$, $\mathrm{ln}_{3}\,x \;\; (\mathrm{ln}_{2}\,z) \;\; 0|_{x=\infty}$, $f \succ \phi$.  \\
Case $\mu + v \lt 1$, $(\mu + v) \, \mathrm{ln}_{2}\,x \;\; (\mathrm{ln}_{2}\,z) \;\; \mathrm{ln}_{2}\,x|_{x=\infty}$, 
 $0 \prec (1-\mu-v) \mathrm{ln}_{2}\,x|_{x=\infty}$, $f \prec \phi$. \\
Case $\mu + v \gt 1$, $(\mu + v) \, \mathrm{ln}_{2}\,x \;\; (\mathrm{ln}_{2}\,z) \;\; \mathrm{ln}_{2}\,x|_{x=\infty}$, 
 $(\mu+v-1) \mathrm{ln}_{2}\,x|_{x=\infty} \succ 0$, $f \succ \phi$.
\end{mex}

In solving relations of infinite magnitude,
 another case
 occasionally arises where
 both sides of the relation are infinite,
 but opposite in sign.
 Raise all parts to a power,
 with the effect of pushing the positive infinity
 further to infinity, and the negative infinity to
 zero, effectively pulling the relation further
 apart.
\bigskip
\begin{mex}$n=\infty$
  \begin{align*}
    -n \;\; (\mathrm{ln}\,z) \;\; \mathrm{ln}\,n \\
    e^{-n} \;\; z \;\; e^{\mathrm{ln}\,n} \\
    e^{-(n+1)} \;\; z \;\; e^{\mathrm{ln}(n+1)} \\
    e^{-(n+2)} \;\; z \;\; e^{\mathrm{ln}(n+2)} \\
    0 \leftarrow \;\; z \;\; \rightarrow \infty 
  \end{align*}
\end{mex}
\begin{prop}\label{P038}
\[ \text{If } f=\infty \text{, } g=\infty \text{, } -f \; z \; g \text{ then } e^{-f} \prec e^{g} \] 
\end{prop}
\begin{proof}
Solving for $z$,  
 $e^{-f} \; z \; e^{g}$, $\frac{1}{e^{f}} \; z \; e^{g}$, $0 \prec \infty$,
 as $e^{f} = +\infty$ then $e^{-f}=0$.
 $z= \; \prec$.
\end{proof}
\bigskip
\begin{mex}\label{MEX018} 
 Solve $n\, \mathrm{ln}\,x \; (\mathrm{ln}\,z) \; \mathrm{ln}\,n|_{n=\infty}$
 when $x=(0,1)$. Within this interval $\mathrm{ln}\,x$ is negative.
  $n\, \mathrm{ln}\,x \succ \mathrm{ln}\,n|_{n=\infty}$, 
  $e^{n\, \mathrm{ln}\,x} \; (e^{\succ}) \; e^{\mathrm{ln}\,n}|_{n=\infty}$, 
  $x^{n} \; (e^{\succ}) \; n|_{n=\infty}$, 
  $x^{n} \prec n|_{n=\infty}$, 
  $z = e^{\succ} = \; \prec$. 
\end{mex}
\bigskip
\begin{defy}\label{DEF023}
Let $==$ mean an equality relation.
\end{defy}
\bigskip
\begin{defy}\label{DEF022}
Let  $z = \; ==$ mean equality is assigned to
 the variable $z$ 
\end{defy}

 We have further introduced
 a use of assignment as a left-to-right generalisation.
 As maths is a language, this decision was made to chain together
 implications.

 Having the context of the problem being solved for is important.

 Multiplying by $-1$ reverses the orders direction,
 the order relations 
 $\{ \lt, \leq, \gt, \geq \}$ are not effected by positive multiplication
 or addition.

 Comparison in an additive sense `is' effected by adding and subtracting terms.
 Not in the sense of the order as described,
 but the magnitudes can be shifted,
 hence the magnitude relations can change direction.
 If $a \succ b$, $a + -a \prec b-a$, $0 \prec b-a$. Part 2

 When adding the same value
 to both sides
 the much-greater-than relation can
 change direction,
 unlike inequalities $\{ \lt, \leq \}$
 which are invariant.
\bigskip
\begin{mex}
\begin{align*}
 x^{2} \succ x|_{x=\infty} \\
 x^{2}-x^{2} \prec -x^{2}+x|_{x=\infty} \tag{direction changes} \\
 0 \prec -x^{2}+x|_{x=\infty}
\end{align*}
\end{mex}

 In solving a relation, say relation $z$, 
 if the aim in solving is to satisfy all the expressions involving $z$,
 then adding to both sides can introduce contradictory 
 solutions.
 If after addition, the highest terms magnitude
 is removed,
 the next highest order determines the magnitude relation. 
\bigskip
\begin{mex}
 Solve $z_{1}$ and $z_{2}$ for the same relation.
\begin{align*}
  n^{2}+3 \; z_{1} \; n^{2}+n|_{n=\infty} \tag{subtract $n^{2}$ } \\
  3 \; z_{2} \; n|_{n=\infty} \\
  z_{1} = \; \lt, \;\; z_{2} = \; \lt 
\end{align*}
\end{mex}

 However, for usability, we really do not wish to be this formal.
\bigskip
\begin{mex} \label{MEX200}
\begin{align*}
 n^{2}+3 \; z \; n^{2}+n|_{n=\infty} \tag{subtract $n^{2}$ } \\
 3 \; z \; n|_{n=\infty} \tag{ solving $z = \; \prec$ contradicts the first expression} \\
 z = \; \lt  \tag{True for both expressions}
\end{align*}
\end{mex}

 This management of variables is part of the mathematics.
 At times, it is only necessary to solve with forward implications.
 Other times, for example in constructing proofs, reversibility, having implications in both
 directions is required. Or more generally, 
 we want to solve not for several binary relation variables, when
 we can do so with one.

Up till now we have explored some of the mechanics
 for solving the relation. Of course,
 the premise was simple, solve for the relation
 as a variable. 

However, infinitary calculus
 concerns itself with functions and curves,
 and continuous families of curves; we
 imagine curves between curves
 as real numbers between other real numbers.
 Just as the relation separated and 
 defined different numbers, relations
 again separate and define different curves.

 In discussing the
 infinity of curves
 near an existing curve,
 G. Fisher
 \cite[pp 109--110]{fisher}
 comments and beautifully quotes du Bois-Reymond
 in developing an infinity of curves close to $y=x$,
 with the following relationship at infinity.
 $x^{\frac{p}{p+1}} \lt x^{\mathrm{ln}_{2}(x) / \mathrm{ln}_{2}(x+1)} \lt x|_{x=\infty}$,
 $p \in \mathbb{N}$.
 That is, there is a function
 between $x^{\frac{p}{p+1}}$ and $x$, and there are
 an infinitely many such functions,
 $x^{\frac{p}{p+1}} \lt \phi(x) \lt x|_{x=\infty}$.  
 Therefore, there are an infinitely many
 functions close to $y=x$ at infinity,
 where the space of real valued functions diverge, $f(x)|_{x=\infty}=\infty$.
\bigskip
\begin{mex}\label{MEX019} 
 There exists an infinity of curves
 infinitesimally close to the straight line
 $y=x$.
 Show 
 $x^{\frac{p}{p+1}} \lt x^{ \frac{\mathrm{ln}_{2}\,x}{\mathrm{ln}_{2}(x+1)} } \lt x^{1}|_{x=\infty}$
.

We will use an indirect inequality approach,
 where we introduce another
 inequality.
 Undoing the base $x$, 
 show 
 $\frac{p}{p+1} \lt \frac{x}{x+1} \lt \frac{\mathrm{ln}_{2}\,x}{\mathrm{ln}_{2}(x+1)}|_{x=\infty} \lt 1$.
\begin{align*}
 \frac{p}{p+1} \; z \; \frac{x}{x+1}|_{x=\infty} \\
 p(x+1) \; z \; (p+1)x|_{x=\infty} \\
 px+p \; z \; px + x|_{x=\infty} \\
 p \; z \; x|_{x=\infty}, \;\; z = \; \lt 
\end{align*}
\smallskipneg
\begin{align*}
  x^{\frac{x}{x+1}}|_{x=\infty} \; z \; x^{\mathrm{ln}_{2}(x) / \mathrm{ln}_{2}(x+1)}|_{x=\infty} \\
  \frac{x}{x+1} \cdot \mathrm{ln}\,x \; (\mathrm{ln}\,z)  \; \mathrm{ln}_{2}(x) / \mathrm{ln}_{2}(x+1) \cdot \mathrm{ln}\,x|_{x=\infty}  \\
  \frac{x}{x+1} \; (\mathrm{ln}\,z) \; \mathrm{ln}_{2}(x) / \mathrm{ln}_{2}(x+1)|_{x=\infty}  \\
  x \, \mathrm{ln}_{2}(x+1) \; (\mathrm{ln}\, z) \; (x+1) \, \mathrm{ln}_{2}\,x|_{x=\infty}  \\
  \mathrm{ln}\,x + \mathrm{ln}_{3}(x+1) \; (\mathrm{ln}_{2}\, z) \; \mathrm{ln}(x+1) + \mathrm{ln}_{3}\,x|_{x=\infty} \\
  \mathrm{ln}\,x \; (\mathrm{ln}_{2}\, z) \; \mathrm{ln}(x+1)|_{x=\infty} \tag{as $\mathrm{ln}\,x \succ \mathrm{ln}_{3}\,x|_{x=\infty}$ }  \\
  x \; (\mathrm{ln}\,z) \; x+1|_{x=\infty}  \\
 \mathrm{ln}\,z = \; \lt, \;\; z = e^{\lt} \; = \; \lt
\end{align*}

 The last relation trivially follows,
 $\frac{\mathrm{ln}_{2}\,x}{\mathrm{ln}_{2}(x+1)} \; z \; 1|_{x=\infty}$,
 $\mathrm{ln}_{2}\,x \; z \; \mathrm{ln}_{2}(x+1)|_{x=\infty}$, $z= \;\lt$.
 Then 
 $x^{\frac{p}{p+1}} \lt x^{ \frac{\mathrm{ln}_{2}\,x}{\mathrm{ln}_{2}(x+1)} } \lt x^{1}|_{x=\infty}$.
\end{mex}
 
 We did not need to introduce the additional
 inequality.  However, like algebra in general,
 we may not get the minimal solution.
 The same calculation, without the added inequality,
 and using an asymptotic approximation could have been
 used, see \cite[Example 2.5]{cebp4}. 

As a rule of thumb, think of infinity
 as being as large as the reals,
 so that you could construct any graph there, and then know that the space is larger still. However what is striking is the existence of the relations themselves, and this property,
 to not just partition, but develop algorithms at infinity.
 It is obvious then that if you can iterate in that space, have relations existing in that space, then you can construct algorithms and mathematical reasoning in that space.
 
 In comparing sequences and functions,
  monotonic 
 sequences and functions,
 that is sequences and functions which
 are either equal to or increasing,
 or equal to and decreasing,
 are of great interest.

 Since a monotonic sequence can be made into a 
 monotonic function,
 and a monotonic function back into a monotonic sequence,
 the theory of functions given is true for
 sequences.

 Roughly, if we can determine
 that two functions are monotonic,
 we can do other things such as compare
 their ratio, and other mathematics.  

 Hence, the interest with
 Hardy's L-functions,
 which are monotonic functions,
 that comprise of
  a 
 finite combination of $\{ +, -, \div, \times, \mathrm{ln}, \mathrm{e} \}$ operations.

 Also, with monotonic 
 functions, the association of a sequence of points
 to a curve would allow sequences and
 functions to be connected, via
 the same
 relations. 

 As sequences are indexed, a connection
 between the discrete and continuous
 can be made.
\bigskip
\begin{prop}\label{P039}
Given a relation $z$,
 and functions $(f,g)$, then
 $(z,(f,g))$ is a relation if
 for all values in the domain,
 there is no contradiction.
 If there is a contradiction,
 $(z,(f,g))$ is not a relation,
 and is disproved as one.
\end{prop}
\begin{proof}
Transforming the functions into sets,
 since the set generated by $(z,(f,g))$ is not
 an exact subset of the relation set generated from $z$, by definition $(z,(f,g))$ is not a relation.
\end{proof}
\bigskip
\begin{mex}\label{MEX048}
Show $n^{2} \gt e^{n}|_{n=\infty}$ contradicts.
 By using L'Hopital Part 5,
 $\infty \gt \infty$, differentiate,
 $2n \gt e^{n}|_{n=\infty}$,
 $2 \gt e^{n}|_{n=\infty}$,
 contradicts as $2 \prec e^{n}|_{n=\infty}$.

 Another approach is by solving,
 and showing that the symbols are contradictory.

$n^{2} \; z \; e^{n}|_{n=\infty}$,
 $\mathrm{ln}(n^{2}) \; (\mathrm{ln}\,z) \; \mathrm{ln}(e^{n})|_{n=\infty}$,
 $2 \, \mathrm{ln}\,n  \; (\mathrm{ln}\,z) \; n|_{n=\infty}$,
 $(\mathrm{ln}\,z) = \; \prec$,
 $z = e^{\prec} = \;\prec$.
 Since both arguments are positive, $\prec$ implies $\lt$ which
 contradicts $\gt$.
\end{mex}

 Propositions and theorems for the infinitely small can be similarly constructed.
\bigskip
\begin{prop}\label{P212}
If ; $f, g \in +\Phi$; and $z \in \{ \lt, \leq, =, \gt, \geq \}$,
 \[f \; z \; g \; \Leftrightarrow e^{f} \; z \; e^{g} \]
\end{prop}
\begin{proof}
 Assuming the partial sums of the
 exponential function with an infinitesimal
 are asymptotic to
 the infinite sum,
  $e^{\delta} = 1 + \delta + \frac{\delta^{2}}{2} + \ldots \sim 1+\delta + \sum_{i=0}^{w} \frac{ f^{(i)}}{i!}$. 

 We have the general partial sum comparisons $z_{i}$,
 $1 + f \; z \; 1 + g$,
 $1 + f + \frac{f^{2}}{2} \; z_{2} \; 1 + g + \frac{g^{2}}{2}$,
 $1 + f + \frac{f^{2}}{2} + \frac{f^{3}}{3!} \; z_{3} \; 1 + g + \frac{g^{2}}{2} + \frac{g^{3}}{3!}$, $\ldots$.

 Case $e^{f} \; z \; e^{g}$ implies $f \; z \; g$:
$e^{f} \; z \; e^{g}$,
 $1 + f + \frac{f^{2}}{2} + \ldots \; z \; 1 + g + \frac{g^{2}}{2} + \ldots$,
 apply non-reversible addition Part 5 Theorem \ref{P058} 
 to $f$ and $g$ partial sums,
 for example
 $\frac{f^{i}}{i!} + \frac{f^{i+1}}{(i+1)!}$
 $= \frac{f^{i}}{i!}$ as $f^{i} \succ f^{i+1}$
 and $z_{i+1} = z_{i}$.
 $1 + f \; z \; 1 + g$, 
 $f \; z \; g$.

 Reversing the process to prove the implication
 in the other direction. Adding the infinitesimal
 does not change the relation as the next term is 
 much less than in magnitude to the previous term.
 $f \; z \; g$,
 $1+f \; z \; 1+g$,
 $1+f + \frac{f^{2}}{2!} \; z \; 1+g + \frac{g^{2}}{2!}$,
 $1 + f + \frac{f^{2}}{2} + \frac{f^{3}}{3!} \ldots \; z \; 1 + g + \frac{g^{2}}{2} + \frac{g^{3}}{3!} + \ldots$,
 recognising the expression as the exponential functions,
 $e^{f} \; z \; e^{g}$.
\end{proof}

 For comparing of functions, we do need well behaved functions,
 hence the monotonic requirements.
 The classes of functions may appear to be restricted,
 but this can be expanded in many ways. 
 Non-reversible arithmetic can be used to remove transient terms.
 For an additive comparison, we may only need
 $f-g$ to be ultimately monotonic, and not the ratio.
 
 So important is the determination
 of these 
 classes of functions 
 that they lead to the following definition
 and conjecture. 

 In accordance with redefining other infinitary
 calculus relations, we redefine
 the L-function in $*G$.
\bigskip
\begin{defy}\label{DEF034}
Define an L-function in $*G$ without implicit
 complex
 numbers
 as a 
 finite combination of $\{ +, -, \div, \times, \mathrm{ln}, \mathrm{e} \}$ operations.
\end{defy}

 The following is given by Hardy
 as a theorem \cite[p.24 Appendix I]{ordersofinfinity},
 but here stated as a conjecture
 because we re-defined the L-functions in
 $*G$ instead of $\mathbb{R}$.
 For example, the limit in $*G$ Part 6 can be a function,
 as it can contain
 infinitesimals and infinities.
 A transfer of the theorem in $*G \mapsto \overline{\mathbb{R}}$
 Part 4 would result in
 the theorem stated by Hardy. 
\bigskip
\begin{conjecture}\label{P028}
Any L-functions at infinity is ultimately continuous and monotonic.
 
 If $f$ and $g$ are L-functions then 
 at $f \; z \; g|_{n=\infty}$,
 $z$ is unique and 
 $z \in \{ \prec, \, \succ, \, \propto \}$
\end{conjecture}

 Additionally the L-functions have the property
 that if $f$ and $g$ are L-functions, 
 so is their ration $f/g$.
  Truncated infinite series can also be L-functions.
\bigskip
\begin{mex}
$x^{2} = e^{2\,\mathrm{ln}\,x}|_{x=\infty}$
 is an L-function.
\end{mex}

\subsection{M-functions an extension of L-functions} \label{S0303}
 Surprisingly, considering infinity as a point in the comparison theory at infinity is not enough.
 On occasion, it is beneficial to compare between two infinities.

 This idea of comparison between two infinities
 was indirectly taken and adapted from NSA 
 where
 convergence was determined by integrating at
 infinity:  integrated between two infinities to determine convergence or divergence.
 Similarly for comparison, we can compare at infinity
 by comparing at infinity, over an infinite interval.

\bigskip
\begin{conjecture}
 L-functions can be compared over any infinite interval.
 $f,\, g \in \text{L-functions}$; 
 $a,\,b, \in \Phi^{-1}$; $b-a = \pm\infty$; $\{ \prec, \succ, \propto \} \in z$:
 $f \; z  \;g|_{x=[a,b]} \equiv f \; z \; g|_{x=\infty}$ 
\end{conjecture}

 We would like to extend Conjecture \ref{P028} to include other functions
 such as $n!|_{n=\infty}$ which do not contain a finite number of 
 multiplication operations, but is monotonic, and either ever increasing
 or ever decreasing.
 The following is an attempt to capture this.

 What is the M-word? Marriage. 
 We define a marriage of properties from
 Conjecture \ref{P028} the monotonic L-functions
 and `infinite term functions'.
  
 By the following definition, assuming the conjecture is true, all 
 L-functions are M-functions.
\bigskip 
\begin{defy}
 M-functions satisfy one of the following comparisons: $\propto$, $\prec$ or $\succ$
 for any infinite interval.
\end{defy}
 If the functions have the same behaviour at infinity, that is $f, g \in \Phi^{-1}$;
 then only consider one infinite interval.
 \[f(n),\, g(n) \in \text{M-functions};\, a,\, b,\, b-a \in \Phi^{-1};\]
 \[ \text{If } f \prec g|_{[a,b]} \text{ then by definition } f \prec g|_{n=\infty} \] 

 Consider the following problem, first
 easily solved with Stirling's formula, then without.  

 With Stirling's formula $e^{n}n! = c n^{n+\frac{1}{2}}|_{n=\infty}$,
 the comparison is obvious. 
\begin{align*}
 n^{n} \; z \; e^{n}n!|_{n=\infty} \\  
 n^{n} \; z \; c n^{n+\frac{1}{2}}|_{n=\infty} \\ 
 1 \; z \; c n^{\frac{1}{2}}|_{n=\infty} \\ 
 z = \; \prec 
\end{align*}
\begin{mex}\label{MEX062}
 Without Stirling's formula, 
 the problem appears more difficult.
 Firstly, re-organize the comparison.
\begin{align*}
 n^{n} \; z \; e^{n}n!|_{n=\infty} \\
 n^{n} \; z \; e^{n}\prod_{k=1}^{n}k|_{n=\infty} \\
 n^{n} \; z \; \prod_{k=1}^{n} ek|_{n=\infty} 
\end{align*}
 Consider the comparison
  between two infinities. For example $(n, 2n]|_{n=\infty}$.
\begin{align*}
 \prod_{k=n+1}^{2n} n \;\; z \; \prod_{k=n+1}^{2n} ek \\
 \prod_{k=n+1-n}^{2n-n} n \;\; z \; \prod_{k=n+1-n}^{2n-n} e(k+n) \\
 n^{n} \;\; z \; \prod_{k=1}^{n} (e (n+k))|_{n=\infty} \\
 n^{n} \;\; z \; \prod_{k=1}^{n} (ne (1+\frac{k}{n}))|_{n=\infty} \\
 n^{n} \;\; z \; n^{n} e^{n} \prod_{k=1}^{n}(1+\frac{k}{n})|_{n=\infty} \\
 1 \; z \; e^{n} \prod_{k=1}^{n}(1+\frac{k}{n})|_{n=\infty} \\
 z = \; \prec 
\end{align*}
\end{mex}
\bigskip
 In Example \ref{MEX062}, it is assumed that $n^{n}, e^{n}n! \in \text{M-functions}$.

 The theory of functions is very important, as if we can guarantee certain properties,
 the analysis can be developed in powerful ways.
 In whatever form that the theory finally takes, identifying and restricting the functions
 will allow the application of comparison algebra to be more consistent and exacting. 
 The development of the applications theory which we believe is a new field of mathematics
 rests on the comparison function theory. It is no longer a question of `will it work', but
 `how does it work'.
 Getting this right potentially means being independent from NSA for solving large classes of problems,
 a goal worth striving for. 
 A language of functions rather than a language of sets for
 solving theory with functions is required.

\section{The transfer principle} \label{S04}
 Between gossamer numbers and the reals,
 an extended transfer principle founded on approximation is 
 described, 
 with transference between different number systems
 in both directions, and within the number systems themselves.
 Therefore an extended transfer principle with non-reversibility is established.
 As a great variety of transfers are possible, hence a mapping
 notation is given.
 In $*G$ we find equivalence with a limit with division and comparison
 to a transfer $*G \mapsto \mathbb{R}$ with comparison.
\subsection{Introduction} \label{S0401}
 The transfer principle in Non-Standard Analysis (NSA)
 generally translates between
 the hyperreals $*\mathbb{R}$ and
 the reals
 $\mathbb{R}$.
 We are similarly interested in
 a transfer principle between  
 the gossamer number system $*G$ Part 1 and real numbers $\mathbb{R}$; including their variants.

 For many reasons, we need to work in the more detailed number
 system. Any such work requires us to interpret or bring back
 the results into $\mathbb{R}$ or otherwise. 
 We note that $*G$ has $\mathbb{R}$ embedded within, making the
 transfer from $\mathbb{R}$ to $*G$ possible.
 However, the nature of a statement in $*G$ may not be able 
 to be expressed in $\mathbb{R}$.
 While the transfer is possible, the meaning may change. (See Example \ref{MEX203})

 We provide another view of the transfer principle
 that is based on
 approximation, a process with indeterminacy.
 This is motivated by the fact that if given
 a number in $*G$ that is not an infinity,
 we can truncate
 successive orders of infinitesimals, and when all the
 infinitesimals are truncated we have only the real component
 remaining.
 Truncating all the infinitesimals 
 is defined as 
  the standard part $\mathrm{st}()$ function,
 which results in a transfer from $*G \mapsto \mathbb{R}$.

 However, taking just one truncation can change an
 inequality. Hence, a more general view
 of a transfer from one state to another is warranted.
 We can also see the algebra of comparing functions Part 3
 as transfers. 

 Then truncating a Taylor polynomial is a transfer;
 $*G \mapsto *G$,
 as truncating a Taylor polynomial may involve infinitesimals,
 which are not in $\mathbb{R}$. 

 More general questions can be asked.
Consider
 the two number systems, one with infinitesimals and
 infinites, the other the reals.
 If $a \gt b$ in $*G$, is this
 true in $\mathbb{R}$? Under what conditions
 is this true?

 That is, can we in one number system,
 transfer to the other number system?
 So, rather than
 working in reals, and extending
 the reals which is implicitly done (for example the
 evaluation of a limit),
 you can deliberately work in one or
 the other number systems, and
 transfer between them.
 
 Surprisingly we
 are applying the transfer principle 
 all the time,
 for example in evaluating limits.
 The limits themselves, 
 having 
 infinitesimals or infinities do
 not belong in $\mathbb{R}$.
 By taking the limit, and truncating the infinitesimals that
 remain, you are effectively taking the standard part of
 the expression.
 That this is not discussed but
 assumed true,
 is part of our culture.
\subsection{Transference} \label{S0402}
 The transfer principle itself is a realization 
 of the `Law of continuity': a heuristic principle 
 developed by Leibniz described in \cite[p.2]{leibnizecont}
\begin{quote} The rules of the finite succeed
 in the infinite and vice versa...
\end{quote} 
 Leibniz rules can be explained
 by $\mathbb{R}$ and $*G$ being fields.
 This law (before Leibniz's characterisation)
 from the beginnings of calculus was used for considering
 infinitary numbers, and their transition.
 It was a consequence of 
 mixing
 infinitesimal quantities with algebra, and then
 getting back tangible results. For example, Fermat 
 integrated a power $\int a^{n}\,dn = \frac{a^{n+1}}{n+1}$.
 To demonstrate the necessity of infinitesimals and infinities
 in the mathematics, from \cite[p.64--65]{fermat} we reconstruct the calculation in $*G$.
\bigskip
\begin{mex}
 Fermat integrating $y=x^{n}$ over $x=[0,a]$, partitioned by $ar^{i}$ where
 $i=\infty \ldots 0$. $t_{i} = a r^{i}$; 
 $\int_{0}^{a} x^{n}\,dx$
 $= \sum_{i=\infty}^{0} t_{i} (t_{i+1}-t_{i})$
 (a Riemann sum of unequal partitions);
 $t_{i+1}-t_{i} = a r^{i} (r-1) \in \Phi$;
 $r \to 1$: $r \in 1+\Phi$; for the Riemann sum to exist, else a divergent sum.
 
 $\sum_{i=\infty}^{0} t_{i} (t_{i+1}-t_{i})$
 $= - \sum_{i=0}^{\infty} a^{n} r^{ni} a r^{i} (r-1)|_{r=1}$
 $= (1-r) a^{n+1} \sum_{i=0}^{\infty} (r^{n+1})^{i}|_{r=1}$
 $= (1-r) a^{n+1} \frac{1}{1-r^{n+1}}|_{r=1}$
 $= (1-r) a^{n+1} \frac{1}{(1-r)(1+r + r^{2} + \ldots + r^{n})}|_{r=1}$
 $=  a^{n+1} \frac{1}{1+r + r^{2} + \ldots + r^{n}}|_{r=1}$
 $= \frac{a^{n+1}}{n+1}$
\end{mex}

 To handle the apparent paradox, the new state is in $*G$ a higher dimensional calculus.
 The above example is done in $*G$ and implicitly brings the result back to real numbers.
 This is exactly what you do when taking a limit. The infinitesimals are
 approximated to $0$ by $\Phi \mapsto 0$ and we are again in the standard real number system.
 Of course there can be no infinities here, so if you intend to increase $n$ to infinity,
 the expression is still in $*G$.

 Hence, the mechanics of calculus require a higher dimensional view.

 That we \text{approximate} in $*G$ produces a transition. 
  But to what purpose? While the `rules' require
 the law, the transfer itself is usually between states.

 For geometric surfaces,
 we can easily visualise the law of continuity applying
 to computer generated meshes.
 As the mesh is refined, a smoother surface appears.
 For a 2D example, at infinity, a polygon of equal sides inside a circle
 becomes the circle.

 That space with infinitesimals existing was predominant in their
 minds.
 Leibniz gives the example of two parallel lines infinitesimally
 close that
 never meet \cite[p.1552]{leibnizecont}.
 Du Bois-Reymond
 constructs an infinity of curves infinitely close, therefore parallel to a straight line 
 Part 3 Example \ref{MEX019}.
 A transfer could be made from these curves to the straight line.

 The following example is perhaps a more complicated transfer, as
 a radical state change occurs, but only at infinity.
 We describe
 a fixed ellipse with one focal point at
 the origin and send the other focal point to infinity.
 The ellipse becomes a parabola, but only after the variable
 of the focal point is sent to infinity before the other
 variables. See \ref{S0607}. A variable reaching infinity before another
\bigskip
\begin{mex}\label{MEX210}
\cite[p.8--9]{leibnizecont} 
With a closed curve the ellipse becomes an open curve, the parabola,
 but only with the focus at infinity.
Send the focus $h$ to infinity.
\begin{align*}
(x^{2} + y^{2})^{\frac{1}{2}} + (x^{2} + (y-h)^{2})^{\frac{1}{2}} = h+2|_{h=\infty} \\
 x^{2} + y^{2} + x^{2} + (y-h)^{2} + 2( (x^{2} + y^{2})(x^{2} + (y-h)^{2}) )^{\frac{1}{2}} = (h+2)^{2}|_{h=\infty} \\
 2x^{2} + 2y^{2} - 2yh + 2( (x^{2} + y^{2})(x^{2} + (y-h)^{2}) )^{\frac{1}{2}} = 4h + 4|_{h=\infty} \tag{Apply non-reversible arithmetic Part 5} \\
 \tag{$2x^{2} + 2y^{2} - 2yh = -2yh|_{h=\infty}$ as $-2yh \succ 2x^{2} + 2y^{2}|_{h=\infty}$, $4h+4 = 4h|_{h=\infty}$ } \\
 ( (x^{2} + y^{2})(x^{2} + (y-h)^{2}) )^{\frac{1}{2}} = yh + 2h|_{h=\infty} \\
  (x^{2} + y^{2})(x^{2} + (y-h)^{2}) = (h(y+ 2))^{2}|_{h=\infty} \\
  (x^{2} + y^{2})h^{2} = h^{2} (y+ 2)^{2}|_{h=\infty} \\
  x^{2} = 4y+4
\end{align*}

 That is, a closed curve is broken open. The ellipse
 is broken to form a parabola at infinity.
 For any finite values
 the curve is always closed, and is an ellipse.
\end{mex}

 The example highlights the directional nature of change.
 After applying non-reversible arithmetic to the equation, a transfer process
 takes place to the new state.
\bigskip
\begin{defy}\label{DEF203}
$\overline{\mathbb{R}} = \mathbb{R} \cup \pm \infty$ the extended real numbers.
\end{defy}
\bigskip
\begin{mex}\label{MEX203}
 From \cite{tprinciple} reformed
 in $*G$. 
 Let $n \in \mathbb{J}_{\infty}$,
 $w \in \mathbb{J}$ be finite then
 $\sum_{k=1}^{w}1 \lt n|_{n =\infty}$
 cannot be transferred to $\mathbb{R}$ because it lacks infinity
 elements in $\mathbb{R}$ then $*G \not \mapsto \mathbb{R}$.
 However, since the extended reals $\overline{\mathbb{R}}$ have 
 infinity the transfer is possible;
 $*G \mapsto \overline{\mathbb{R}}$: 
 $\sum_{k=1}^{w}1 \lt \infty$, which is slightly different
 as the extended reals $\overline{\mathbb{R}}$
 only have two infinity elements, $\pm \infty$.
\end{mex}
\bigskip
\begin{mex}\label{MEX204}
 In $*G$,
 $2 + \frac{1}{n} \gt 2|_{n=\infty}$, but
 $*G \not \mapsto \mathbb{R}$.
 However, if we replace the strict inequality
 to include equality,
 the transfer is possible.
 $2 \geq 2$ in $\mathbb{R}$. (See Theorem \ref{P003})
\end{mex}

\begin{quote}
 A transfer principle states that all statements of some language that are true for some
 structure are true for another structure\cite{tprinciple}.
\end{quote}

\begin{quote}
A sentence in $\varphi$ in $L(V(S))$ is true
 in $V(S)$ if and only if its *-transform $\mathrel{^*\!}\! \varphi$ 
 is true in $V( \mathrel{^*\!}\!S)$ \cite[p.82]{nigelc}
\end{quote}

 From Example \ref{MEX204} we see the transfer
 definitions given above are not adequate.
 While it is
 very important and most useful to take a proposition
 in one number system, and have the proposition true in
 another.
  For example, theorem proving
 where if true in one system implies the truth
 in the other.
 However, the principle as stated is not complete
 because a transfer can change the relation's meaning.

 We put forward a definition of the extended transfer principle,
 which in part, is based on approximation. 
 Where, by realizing infinitesimals, we can truncate
 expressions.
  By seeing the continued truncation of infinitesimals
 as a sequence of smaller operations, we can
 transfer within the same number system.
 
We find such truncation
 can describe non-reversible processes, 
 which lead to non-reversible arithmetic Part 5.

 The second part of the transfer principle generalization
 is its directional nature. Transfers exist in both directions. 

 Possibilities arise from non-uniqueness, for example, 
 a single point of discontinuity in $\mathbb{R}$ can be
 continuous in $*G$;
 transferring from $*G$ to the point discontinuity
 in $\mathbb{R}$ can be done in several ways.
 Perhaps a deeper
 transfer is the promotion of an infinitesimal
 to a small value. 
\bigskip
\begin{defy}\label{DEF040}Transfer principle:  
 Assume an implementation of the ``Law of continuity" between
 $\mathbb{R}$ or $\overline{\mathbb{R}}$ and $*G$ or $*\overline{G}$.
 For each number $x$ in the target
 space, $x \mapsto x'$ in the image space.
 If true over the domain in the target space, then
 it is true in the image space.
\end{defy}
\bigskip
\begin{defy}
 Extended transfer principle:
 Depending
 on context we can transfer in either direction,
 and in any combinations of number systems
 and operations. Further, dependent on the transfer,
 the relations may change.
\end{defy}

 Example \ref{MEX203} is an extended transfer. For further examples
 see Table \ref{FIG20} Mapping examples.

 We differentiate between infinitesimals
 and zero. Similarly we differentiate between
 and an infinity such as $n^{2}|_{n=\infty}$ and the 
 number $\infty$.

We will define an operation to convert from 
 ``an infinitesimal" to zero, and an operation
 to convert from ``an infinity" to infinity. 
 In other words, zero is a generalization of
 infinitesimals and its own unique number.
 Similarly, infinity
 is a generalization of infinities, and its own 
 unique number.

\bigskip
\begin{defy}\label{DEF046}
We say ``realizing an infinitesimal"
 is to set the infinitesimal to $0$, and
 ``realizing an infinity" is
 to set an infinity to $\infty$.
\end{defy}
\bigskip
With these definitions an infinitesimal is not
 $0$, but a realization of it.
 ``An infinity" is not infinity,
 but an instance of it.
 By the `realization' operation we convert
 infinitesimals and infinities respectively to $0$
 or $\infty$.
 The numbers $0$ or $\infty$,
 while mutual inverses, have no specific inverses. After a realization,
 you cannot go back. 
\bigskip
\begin{mex}\label{MEX037}
 $\infty \notin \Phi^{-1}$, but
 $\Phi^{-1}=\infty$ as a left-to-right generalization
 is true. Similarly $0 \notin \Phi$,
 but $\Phi = 0$.
\end{mex}
\bigskip
\begin{mex}\label{MEX039}
 We may have $n^{2}|_{n=\infty}=\infty$. The left side
 is a specific instance of the right side generalization.
 Similarly for zero, $\frac{1}{n}|_{n=\infty}=0$.
\end{mex}
\bigskip
\begin{mex}\label{MEX036} 
$(\frac{1}{n}, \frac{1}{n+1}, \frac{1}{n+2}, \ldots)|_{n=\infty} = (0, 0, 0, \ldots)$ is a null
 sequence, 
 $\frac{1}{n}|_{n=\infty} \in \Phi$
\end{mex}
\bigskip
\begin{mex}
If we consider realizing an infinitesimal
 before dividing,
 $1/\delta = 1/0$, $\delta \in \Phi$,
 then we can interpret
 $1/0=\infty$ 
 as a generalization of a specific infinity
 which we do not know about, nor necessarily care.
\end{mex}
\bigskip
\begin{defy}
 Let 
 $\mathrm{rz}(x)$ realize
 infinitesimals and infinities
 in an expression $x$.
 Lower order of magnitude terms are deleted by
 non-reversible addition:
 if $a \succ b$ then $a+b=a$ Part 5.
\end{defy}
\bigskip
\begin{mex}
 $\mathrm{rz}(n^{2}+n+1)|_{n=\infty} = n^{2}|_{n=\infty}$
 as $n^{2} \succ n+1|_{n=\infty}$.
\end{mex}
\bigskip
\begin{defy}
Let $\mathrm{rz}(z)$ realize
 infinitesimals and infinities in an additive sense
 about the relation $z$.
\end{defy}
\bigskip
\begin{mex}
 Realize relation $n \; z \; n^{2}|_{n=\infty}$, 
 $n \; \mathrm{rz}(z) \; n^{2}|_{n=\infty}$,
 $0 \; \mathrm{z} \; n^{2}|_{n=\infty}$ because $n \; z \; n^{2}$,
 $0 \; z \; n^{2}-n$, $0 \; z \; n^{2}|_{n=\infty}$ as $n^{2}-n = n^{2}|_{n=\infty}$.
\end{mex}

Fisher \cite[p.115]{fisher} comments,
 while following du Bois-Reymond's infinitary calculus: 
\begin{quote}
For the objects of his infinitary calculus, ..., are functions
 which do not form a field under
 the operations he considers(addition and composition),
 whereas his reference to
 ordinary mathematical quantities obeying the same rules that
 hold for finite quantities might be taken to refer to a field.
\end{quote}

 While $*G$ we believe forms a field Part 1,
 this is the first step before applying arithmetic
 that has no inverse.
 Realizing can simplify the expression,
 however such an operation is non-reversible.
 Continued application could lead to $0$ or $\infty$,
 which could then be captured as a theorem.
 However in this case $0$ and $\infty$ have no inverses,
 with respect to their state before realization.
 While we define $\infty$ and $0$ as mutual inverses,
 the realization was likely done before this,
 perhaps through a limit.

 The properties which make infinitary
 calculus not a field are the valuable properties.
 We approximate and use the field.
 The transfer (when realizing) is not independent of the number system,
 but part of it.

 If after realizing a number,
 has the number changed type?
 Has the meaning of the relation changed?
\bigskip
\begin{mex}\label{MEX202}
 $x^{2} + \frac{1}{x}|_{x=\infty}$ realized to $x^{2}$ by discarding the
 infinitesimal, apply a transfer principle to
 bring back to $\mathbb{R}$, $y=x^{2}$. 
\end{mex}

When transferring from
 a higher dimension number system
 to a lower dimension number system,
 in general, we need to consider the 
 law of continuity after
 truncating.

The purpose of working in one number system
 is often to transfer the results to
 another.  For example, we could solve a problem 
 in integers with real numbers,
 and transfer the result back to working in the integer domain. 
\bigskip
\begin{defy}
Let $A$ and $B$ be number systems, we say $A \mapsto B$ to mean
 the number system $A$ is projected onto or maps to the number system $B$. 
\end{defy}

Hence when realizing infinitesimals in
 $*G$ to $\mathbb{R}$ we approximate and
 simplify expressions.
 Let $x = a + \delta$, $a \in \mathbb{R}$, $\delta \in \Phi$.
 Truncating all the infinitesimals (assuming $\Phi^{-1} \notin x$)
 is the same as Robinson's NSA,
 ``taking the standard part" \cite[p.57]{abraham}, $\mathrm{st}(a+\delta)=a$. 
 If we interpret this as converting a number with
 infinitesimals to a real number, 
 we can see $*G$ as a more detailed space,
 where the numbers can ultimately be realized
 as reals. So $0$ in reals may be expressed
 as an infinitesimal in $*\overline{G}$,
 as $0$ is a projection
 of an infinitesimal to the reals.
\bigskip
\begin{prop}\label{P029}
 $\delta \in \Phi$;
  $*G \mapsto \mathbb{R}: \delta \mapsto 0$
\end{prop}
\begin{proof}
By approximation, repeated
 truncation of the
 infinitesimals leaves $0$.
\end{proof}

In the realization of infinitesimals and infinities,
 information can be lost. $\frac{1}{n} \prec 1|_{n=\infty}$
 becomes $0 \lt 1$. If the relation was reordered differently,
 the much greater than
 relation could remain, but the realization
 would not be to the real number system,
 but the extended reals.
 $\frac{1}{n} \prec 1|_{n=\infty}$,
 $1 \prec n|_{n=\infty}$, $1 \prec \infty$.
\bigskip
\begin{mex}\label{MEX053}  
 Let $\delta \in \Phi$,
 definition $f/g= \delta$ in $*G$ becomes
 $f/g=0$ in $\mathbb{R}$.
\end{mex}

A similar process exists for infinities, where the
 ``infinities" are realized and converted to the "infinity".
\bigskip
\begin{mex}\label{MEX054} 
 $\frac{1}{n+1} \lt \frac{1}{n}|_{n=\infty}$ in $*\overline{G}$ is true,
 but contradicts in $\mathbb{R}$ when we realize the infinitesimals: $0 \lt 0$.
 Similarly rearranging to compare infinities,
 $\frac{1}{n+1} \lt \frac{1}{n}|_{n=\infty}$,
 $\frac{n}{n+1} \lt 1|_{n=\infty}$,
 $n \lt n+1|_{n=\infty}$,
 realizing the infinities contradicts; $\infty \lt \infty$.
\end{mex}
\bigskip
\begin{mex}\label{MEX040} 
$\frac{1}{n} \gt \frac{1}{n^{2}}|_{n=\infty}$,
 $n^{2} \frac{1}{n} \gt n^{2} \frac{1}{n^{2}}|_{n=\infty}$,
 $n \gt 1|_{n=\infty}$, $\infty \gt 1$ is true. 

$\frac{1}{n} \gt \frac{1}{n^{2}}|_{n=\infty}$,
 realizing $*G \mapsto \mathbb{R}$, 
 $0 \ngtr 0$ is false, but
 $\gt \, \Rightarrow \, \geq$ is true. 
\end{mex}

Infinitesimals being smaller than any
 number in $\mathbb{R}$. Within an inequality, they can
 change to equality when removed.
\bigskip
\begin{mex}
 Let $\delta \in \Phi$, 
 $(*G, e^{f} \lt e^{f+\delta}) \mapsto (\mathbb{R}, e^{f} == e^{f})$
\end{mex}
\bigskip
The following theorems may, if unfamiliar, seem trivial.
 If some proposition is true for a range, it is also
 true for its subrange: why would we make this
 into a theorem?
 
 The very reduction of range can
 greatly simplify the complexity of cases involved.
 Hence, why construct theorems for reals and 
 infinities if we only need to handle the infinite case?

Doing so, we believe leads to a
 radically different view of convergence and a new way to
 integrate: \cite[Convergence sums ...]{cebp2}
 and rearrangement theorems with order on the infinite interval
 \cite[Rearrangements of convergence sums at infinity]{cebp8}.

 By partitioning an interval between
 the infinireals and other numbers,
 we can separate arguments on finite numbers and infinireals
 to arguments on infinireals alone,
 and transfer when we need to go back to real or gossamer numbers.

 Particularly important is the implication that
 partitioning by a finite bound,
 and including the infinity cases
 implies the infinity case, be it infinitely
 small or infinitely large.

 We have kept the transfer notation $\mapsto$ as we are losing
 information in the process. 
\bigskip
\begin{defy}
 A `bounded number' is a number that is not an infinireal.
 All reals are bounded numbers, and so are all reals except
 $0$ with infinitesimals. If $x$ is a bounded number
 then $x \in *G - \mathbb{R}_{\infty}$
\end{defy}
\bigskip
\begin{defy}
 We say an `implicit infinite condition' has a domain
 that includes both finite numbers (which can include infinitesimals)
 and infinireals $\mathbb{R}_{\infty}$.
Let $x$ and $x_{0}$ be either real or
 gossamer numbers. \\ 
 1. $\forall x \gt x_{0}$ where $x_{0}$ is finite. \\
 2. $\forall x: |x| \lt x_{0}$ where $x_{0}$ is finite.
\end{defy}
Since the finite numbers  are partitioned by the
 infinite numbers (the infinireals), we remove the finite condition.
\bigskip
\begin{theo}\label{P075}
If an implicit infinite condition at infinity
 determines some proposition $P$
 then we can transfer to the infinitely large domain.
 \[(\forall x: x \gt x_{0}) \mapsto (x \in +\Phi^{-1}) \]
\end{theo}
\begin{proof}
\[ [x \gt x_{0}] = [x \gt x_{0} ][x \lt +\Phi^{-1}] + [x \gt x_{0}] [x \in +\Phi^{-1}] \]
 Since choosing $x \in +\Phi^{-1}$ in the
 domain always
 satisfies the condition,
 the transfer is always possible.
\end{proof}
\bigskip
\begin{theo}\label{P076}
If an implicit infinite condition at the infinitely small determines some proposition $P$
 then we can transfer to the infinitely small domain.
 \[(\forall x: -x_{0} \lt x \lt x_{0}) \mapsto (x \in \Phi) \]
\end{theo}
\begin{proof}
\[ [|x| \lt x_{0}] = [|x| \lt x_{0} ][|x| \in +\Phi] + [x \lt x_{0}] [x \not\in +\Phi] \]
 Since choosing $x \in \Phi$ always
 satisfies $|x| \lt x_{0}$ in the above,
 the infinitely small case is always true and the transfer is always possible.
\end{proof}
\bigskip
\begin{defy}\label{DEF204}
In context, a variable $x$ can be described at infinity
 $|_{x=\infty}$ 
 corresponds with Theorem \ref{P075}
\end{defy}
\bigskip
\begin{defy}\label{DEF205}
In context, a variable $x$ can be described at zero 
 $|_{x=0}$  
 corresponds with Theorem \ref{P076}
\end{defy}

 Theorems \ref{P075} and \ref{P076} are a common reduction 
 within the transfer.
 Because of the variations of mapping involving the transfer from one
 domain to another, with different relations, we have
 developed a loose and not exact notation to communicate the mapping,
 and its context.
\bigskip
\begin{defy}
 Let $(K, \; \lt \! f_{b} \! \gt, \; \lt \! x \! \gt )$ describe the number system and context, where the
 angle brackets indicate optional arguments. $K$ is the number type, $f_{b} \in \mathbb{B}$ a binary relation.
\[ (\text{number type}, \lt \text{relation} \gt ,
 \lt \text{number} \gt ) \]
A mapping between domains can be described by
\[ (K, \; \lt \! f_{b} \! \gt, \; \lt \! x \! \gt ) \mapsto (K', \; \lt \! f_{b}' \! \gt, \; \lt \! x' \! \gt ) \]
\end{defy}

\begin{remk}
A number may 
 be input that is not of the same type
 as its result.
 For example 
 $f/g$ may be in number system $K$ 
 but neither $f$ nor $g$ need necessarily be in $K$.
 Limit calculations  
 happily 
 accept input with infinities and infinitesimals,
 but the limit can be in $\mathbb{R}$.
\end{remk}

Mapping can occur in different contexts: realization, rearrangement of expression, transfer principle. 
 The mapping can be in many different combinations. We summarize
 with a flexible notation; it is not at all strict. 

\begin{table}[H]
  \centering
  \begin{tabular}{|c|c|} \hline
    \text{Mapping} & \text{Comment} \\ \hline
 $(*G,/) \mapsto \mathbb{R}/\overline{\mathbb{R}}$ & \text{Limit $\frac{a_{n}}{b_{n}}|_{n=\infty}$ evaluation} \\ \hline
 $(\mathbb{R},/) \mapsto (*G,/)$ & \text{Undoing an implicit limit} \\ \hline
$(\forall x \gt x_{0}) \mapsto (\Phi^{-1}, |_{x=\infty})$ & \text{Law of continuity from $\mathbb{R}$ to $*G$} \\
$(\forall |x| \lt x_{0}) \mapsto (\Phi, |_{x=0})$ & \text{Theorems \ref{P075}, \ref{P076}} \\
     \hline 
$(*G,\Phi^{-1}) \mapsto (\overline{\mathbb{R}}, \infty)$ & \text{Realize infinities} \\ \hline
$(*G, \Phi) \mapsto (\mathbb{R},0)$ & \text{Realize infinitesimals, apply st() the standard part} \\ \hline
 If $(*G,\not\sim)$ then $(*G, \mathrm{rz}(\lt)) \mapsto (\mathrm{R}/\overline{\mathbb{R}}, \lt)$ & Theorem \ref{P002} \\ \hline
If $(*G, \not \simeq)$ then $(*G\backslash\{\Phi^{-1}\},\lt) \mapsto (\mathbb{R}, \lt)$ & \text{Corollary \ref{P006}} \\ \hline
$(*G,\lt) \mapsto (\mathbb{R}/\overline{\mathbb{R}}, \leq)$ & \text{Loses information, Theorem \ref{P003} } \\ \hline
$(*G, (\Phi \lt \Phi)) \mapsto (\mathbb{R}, (0 \lt 0) )$ & \text{See Example }\ref{MEX054} \\ \hline
$(*\overline{G}, \infty) \not\mapsto (\mathbb{R},\infty)$ & Infinity is not in $\mathbb{R}$ \\ \hline
$(\mathbb{R}, f \not \in C^{0}) \mapsto (*G, f_{2} \in C^{0})$   & Adding information \cite{cebp10} \\ \hline
$(J_{\infty},n) \mapsto (*G,n)$ & \text{Discrete to continuous domain} \\ \hline
  $(\Phi,\delta_{n}) \mapsto (\mathbb{R},\delta_{n})$ & \text{Algorithm example} Part 5 Example \ref{MEX206} \\ \hline
 $\sum a_{n}|_{n=\infty} \mapsto \sum_{k=k_{0}}^{\infty} a_{k}$ & \text{Convergences sums to sums} \cite[Theorem 11.1]{cebp2} \\ \hline
  \end{tabular}
  \caption{Mapping examples} \label{FIG20}
\end{table} 
 Consider a limit. While the image space
 may by $\mathbb{R}$, the solution space is $*G$ as it holds
 infinitesimals and infinities. Hence, given $*G \mapsto \mathbb{R}$,
 we can consider $*G$ and postpone or avoid the transfer.
 The implicit nature of the limit can be undone.
\bigskip
\begin{mex}\label{MEX005}
  $\frac{n^{2}+1}{n^{2}}|_{n=\infty}=1 \in \mathbb{R}$,
 but $n^{2}+1|_{n=\infty} \in *G$ and $n^{2}|_{n=\infty} \in *G$.
 Hence $\frac{n^{2}+1}{n^{2}}|_{n=\infty} \in *G$.
 However, the limit calculation 
 can be described by $*G \mapsto \mathbb{R}$.
\end{mex}

 If we consider the more general question of function
 evaluation, we have numbers that may be transfered
 between the number systems $\mathbb{R}$ and $*G$.
 While a function 
 returns a value,
 by a transfer process 
 it may not actually be calculated in that type.
 A transfer can occur between
 the calculation and the function's returned value.

 Consider now, the function return value location as 
 holding a local variable. 
 If the function type does not match
 the location type, a transfer is made.

 In the evaluation, we can show the implicit transfer.
$(*G, =) \mapsto (\mathbb{R},=)$.
 Then $\frac{a_{n}}{b_{n}}|_{n=\infty}=1$ can be
 multiplied through to $a_{n} = b_{n}|_{n=\infty}$ in $*G$. 

 We introduce a notation to explicitly describe a relation,
 to help describe the transference rather than of practical use.
 With
 transference, the relation argument types are likely to be in the less detailed
 number system, but the evaluation in the more detailed number system $*G$.
\bigskip
\begin{defy}\label{DEF206}
Let two arguments of a binary relation be
 described by their type $T1$, $T2$ where
 $z$ is the binary relation.
\[ ( T1 \; z \; T2 ) \]
\end{defy}
\bigskip
\begin{mex} \label{MEX201}
To undo an infinite operation, we need to 
 promote the numbers to $*G$. The comments
 indicate the left and right types on
 either side of the equality relation.
  \begin{align*}
    \frac{n^{2}+1}{n^{2}}|_{n=\infty}=1 & \tag{ $*G = \mathbb{R}\text{ or}\,*\!G$ } \\
    *G \mapsto \mathbb{R} \\ 
    \frac{n^{2}+1}{n^{2}}|_{n=\infty}=1  & \tag{ $*G = *G$ } \\
    n^{2}+1 = n^{2}|_{n=\infty} & \tag{ $*G = *G$ } \\
  \end{align*}
\end{mex}
\bigskip
\begin{mex} \label{MEX004}
 Promote a limit to a limit in $*G$.
$(\mathbb{R},0) \mapsto (*G,\Phi)$ 
  \begin{align*}
    \frac{\mathrm{sin}\,\frac{1}{n}}{n}|_{n=\infty}=0 & \tag{ $*G = \mathbb{R}$ } \\
\frac{\mathrm{sin}\,\frac{1}{n}}{n}|_{n=\infty}=\delta; \;\; \delta \in \Phi & \tag{ $*G = *G$ } \\
  \end{align*}
\end{mex}

$\mathbb{R} \mapsto *G$ is one-one as $\mathbb{R}$ is embedded in $*G$.
 However $*G \mapsto \mathbb{R}$ is different, as information about
 the infinitesimals is lost.
 Because $*G$ is more dense
 than $\mathbb{R}$,
 the transfer principle
 applied to 
 the strict inequalities $\{ \lt, \gt \}$
 for variables/functions
 which are infinitesimally
 close, fail. 
 Examples \ref{MEX054} and \ref{MEX040}
 implicitly worked in $*G$
 and the relations changed when projected onto $\mathbb{R}$.
\bigskip
\begin{prop}\label{P055}
 $\delta \in \Phi$; 
 $*G \mapsto \mathbb{R}/\overline{\mathbb{R}}$;
 If $h \not \simeq 0$
 in $*G$ then $h \neq 0$ in $\mathbb{R}/\overline{\mathbb{R}}$.
\end{prop}
\begin{proof}
Either $\mathbb{R} \in h$ or
 $\Phi^{-1} \in h$, both components map to non-zero elements in $\mathbb{R}/\overline{\mathbb{R}}$.  
\end{proof}
\bigskip
\begin{prop}\label{P056}
$*G \mapsto \mathbb{R}/\overline{\mathbb{R}}$: If $h \gt \Phi$ in $*G$ then
 $h \gt 0$ in $\mathbb{R}/\overline{\mathbb{R}}$.
\end{prop}
\begin{proof} 
If $h \gt \Phi$ in
 $*G$
 then either $h \in +\Phi^{-1} \mapsto \infty$
 or $h$ has $\mathbb{R}^{+} \mapsto \mathbb{R}^{+}$.
 Neither result is $0$ in $\mathbb{R}$.
\end{proof}
\bigskip
\begin{cor}
If $f \gt 0$ in $*G$ and $\mathbb{R}$ then
 $f \not\simeq 0$.
\end{cor}
\begin{proof}
 Assume true and show a contradiction. 
 Let $f = \delta$ in $*G$, $\delta \in \Phi^{+}$
 then $f \gt 0$ in $*G$. 
 $(*G, \;\; \delta \gt 0) \mapsto (\mathbb{R}, \;\; 0 \gt 0)$
 is contradictory.
\end{proof}
\bigskip
\begin{mex}
$*G \mapsto \mathbb{R}$,
 $\frac{1}{n+1} \lt \frac{1}{n} \mapsto \frac{1}{n+1} == \frac{1}{n}|_{n=\infty}$
\end{mex}
\bigskip
\begin{theo} \label{P001}
 If $(*G,\sim)$ then $(*G,\lt) \mapsto (\mathrm{R}/\overline{\mathbb{R}}, =)$ 
\end{theo}
\begin{proof}
 $f \lt g$,
 $\frac{f}{g} \lt 1$,
 since $f \sim g$ then let $\frac{f}{g} = 1- \delta$ to preserve the less than relation, $\delta \in \Phi^{+}$.
 Apply transfer, $\delta \to 0$,
 $\frac{f}{g}|_{n=\infty}=1$,
 $f=g$.
\end{proof}
\bigskip
\begin{mex}
$(*G,n \; \mathrm{rz}(\lt) \; n^{2})|_{n=\infty} \mapsto (\overline{\mathbb{R}}, 0 \lt \infty)$.
 $n \; z \; n^{2}|_{n=\infty}$,
 $0 \; z \; n^{2}|_{n=\infty}$,
 $0 \lt \infty$
\end{mex}
\bigskip
\begin{theo} \label{P002}
 If $(*G,\not\sim)$ then $(*G, \mathrm{rz}(\lt)) \mapsto (\mathrm{R}/\overline{\mathbb{R}}, \lt)$ 
\end{theo}
\begin{proof}
Let $f \; z \; g$.
 If either $\{ f, g \} \in \Phi^{-1}$,
 after realization, for the negative infinity case:
 $-\Phi^{-1} \lt 0 \mapsto -\infty \lt 0$,
 positive infinity case: $0 \lt \Phi^{-1} \mapsto 0 \lt \infty$.
 Without infinities, let $g = \alpha+f$,
 $\alpha \gt 0$ to maintain the inequality, 
 $\mathbb{R} \in \alpha$ as $f \not \simeq g$
 then $\alpha \in \mathbb{R}^{+}$. 
 In $*G$,
 $f \lt \alpha+g \mapsto \mathrm{st}(f) \; z_{2} \; \alpha' + \mathrm{st}(f)$ where $\alpha' \in \mathbb{R}^{+}$,
 $0 \; z_{2} \; \alpha'$,
 $z_{2} = \; \lt$.
\end{proof}
\bigskip
\begin{cor}\label{P006}
If $(*G,\not \simeq)$
 then
 $(*G\backslash\{\Phi^{-1}\},\lt) \mapsto (\mathbb{R}, \lt)$
\end{cor}
\begin{proof}
 $(*G,\not \sim)$ without infinity
 becomes $(*G, \not \simeq)$ at infinitely close test only.
\end{proof}
\begin{theo}
 $(\mathbb{R}/\overline{\mathbb{R}}, \lt) \mapsto (*G,\not \sim)$ and $(*G, \lt)$
\end{theo}
\begin{proof}
Since $\mathbb{R}$ embedded in $*G$,
 the $\lt$ relation follows.
 Since elements in $\mathbb{R}/\overline{\mathbb{R}}$ are separated by an infinity
 or real number the elements
 cannot be asymptotic in $*G$. 
\end{proof}
\bigskip
\begin{prop}
 $\mathrm{rz}(\sim) = \; \sim$
\end{prop}
\begin{proof}
 The limit itself considers magnitudes, usually
 by dividing the infinities and realizing the
 infinitesimals.
\end{proof}
\bigskip
\begin{theo} \label{P003}
 $(*G,\lt) \mapsto (\mathbb{R}/\overline{\mathbb{R}}, \leq)$
\end{theo}
\begin{proof}
 Since $\sim$ and $\not \sim$ cover all cases,
 combining Theorem \ref{P001} and Theorem \ref{P002} covers
 all cases, the union of the two images. 
\end{proof}

The limit calculation can be described as an evaluation
 in a more detailed number system with infinitesimals
 and infinities which is projected back to the real numbers. 

$\frac{a_{n}}{b_{n}}|_{n=\infty}=1$ can be expressed
 as $(*G = \mathbb{R})$ or $(\mathbb{R},=)$, even though
 the fraction may not be in $\mathbb{R}$. However
 the result of the ratio is in $\mathbb{R}$,
 hence we state it this way. 

 We found the
 limit to be a transfer as it realizes 
 infinitesimals and infinities.
 Hence, we provided a number system $*G$,
 which contains infinireals and better describes
 the limit calculation.

 In what follows, we are able to decouple a fraction
 about $1$, multiplying the numerator
 and denominator out, while still able to
 simplify as a fraction, through
 transfers.
\bigskip
\begin{theo} \label{P210}
$ z_{1}, z_{2} \in \mathbb{B}$;
 $a_{n} \gt 0$, $b_{n} \gt 0$.
\[ \frac{a_{n}}{b_{n}}|_{n=\infty} \; z_{1} \; 1 \;\; \Leftrightarrow \;\; a_{n} \; z_{2} \; b_{n}|_{n=\infty} \]
\begin{table}[H]
  \centering
  \begin{tabular}{|c|c|c|} \hline
  \text{Condition 1} & & \text{Condition 2} \\
  $z_{1}$ & $z_{2}$ & \\
  $\lt$ & $\lt$ & $(*G,\not \sim)$ \text{ then } $(*G,\lt) \mapsto (\mathbb{R}/\overline{\mathbb{R}}, \lt)$ \\ 
  $\gt$ & $\gt$ & $(*G,\not \sim)$ \text{ then } $(*G,\gt) \mapsto (\mathbb{R}/\overline{\mathbb{R}}, \gt)$ \\ 
  $=$ & $\sim$ & \text{not one of the other cases} \\ \hline
  \end{tabular}
\end{table}
\end{theo}
\begin{proof}
 While we can bring the fraction into $*G$ and multiply
 out the denominator,
 to show equivalence we need to show that the transfer
 back from $*G$ to $\mathbb{R}/\overline{\mathbb{R}}$
 is equivalent to the limit.

Condition 2 can map $*G \to \mathbb{R}/\overline{\mathbb{R}}$ as the cases
 are disjoint and cover $*G$.
 If we consider the fractions
 equality case, by definition
 is $\sim$.
 However, this case is considered by excluding the inequality
 cases, partitioning $*G$ into the three disjoint cases.
 For the inequality, condition 2 leads to condition 1 by
 Theorem \ref{P001}.
 All cases of $*G \to \mathbb{R}/\overline{\mathbb{R}}$ have been considered.
\end{proof}
\bigskip
\begin{theo} \label{P211}
 Extend Theorem \ref{P210}:
 $a_{n} \neq 0$, 
 $b_{n} \neq 0$,
\[ \frac{a_{n}}{b_{n}}|_{n=\infty} \;  z_{1} \; 1 \;\; \Leftrightarrow \;\; a_{n} \; (\mathrm{sgn}(b_{n}) z_{2}) \; b_{n}|_{n=\infty} \]
\end{theo}
\begin{proof}
 Since multiplying by a negative number inverts the inequality
 which is included within the theorem, the proof is
 identical to the proof given in Theorem \ref{P210}.
\end{proof}

 While Theorem \ref{P210} loses
 information as $*G$ is more dense,
 so is not reversible, 
 we can map back to $*G$.
 Consider the table relations
 $z_{1} \to z_{2}$,
 if we consider  $\mathbb{R}/\overline{\mathbb{R}} \to *G$,
 as $\mathbb{R}$ is embedded,
 the inequalities hold.
\subsection{Overview and prelude} \label{S0403}
 The significance of ratio comparison being expressed in $*G$ is that
 we can undo tests about a value.
 A ratio comparison is often given as a fixed truth set in stone, the
 pinnacle result. For example, the ratio test.
 We 
 need to investigate, undo or analyse tests, to modify or change them perhaps,
 to find out how they may work or build new ones.

 The major aspect to this paper is the idea of
 transference, where we explore the change in relations as a transition.
 In this way, we can understand why a statement with infinitesimals
 may have a strict inequality, but when transfered to reals has
 an equality.  

 Many of the definitions are introduced as a basis for
 later work.
 The partitioning between finite and infinite domains is particularly important,
 as seen in Definitions \ref{DEF204} and \ref{DEF205},
 because they are a problem separation and an extension to classical analysis.
 From Fermat and other examples, these elements always existed, but
 they needed form.

 The mappings Table \ref{FIG20} gives a range of applications and theory, and yet it is
 only a snapshot.  However, we felt a need to describe the extent of
 the theory, and hence in a simpler way start to reveal why $*G$ would
 be useful in numerical calculations as well as theoretical mathematics.

 While $\mathbb{R}$ is embedded in $*G$, this does not mean a translation
 between $\mathbb{R} \mapsto *G$ is one-to-one. For example, we could
 add information by inserting a curve at a singularity.

 That transference explains infinitesimal and infinitary calculations
 is interesting.  The missing mathematics of infinitesimals
 and infinities which have eluded analysis for centuries are the
 same operations that are performed on a day to day basis: primarily, approximation. 

 There is a paradigm shift. In particular, the exploration of
 possibilities is opened up.
 The user of this mathematics will need other tools for the larger workspace,
 but we believe will more likely find the mechanics of the mathematics.

 It is hard to understand how infinitesimals
 fell. For they were the most successful aspect of calculus and are a part of it.
 We feel that the inability to construct a number system that explained
 the infinitesimals and infinities led to their demise.
 Infinitesimals never truly disappeared,
 but instead went underground, incorporated in much of todays mathematics (see \cite{victors}).
 Their resurrection in 1960s through NSA
 was a sustained correction, that even today
 has not been fully achieved. Two hundred years of denial are not so easily fixed.
 
 Our view of
 the calculus is two tiered, with transference explaining the contradictions.
 It also opens worlds of possibilities. This is not the answer, only the beginning of
 the journey.

 We are not taught to deal with contradiction well.
 In freeing the mathematics by acknowledging different
 possibilities, the ability to analyse mathematics is greatly increased.

In relation to quantum mechanics described in the language of statistics theory,
 Einstein said ``God does not play dice".
 We believe this meant that everything
 has a structure, whether we can see it or not.
 It is quite conceivable that statistics in quantum mechanics
 works but may not reveal many hidden structures.
 Any theory, after all, is only one viewpoint.

 Similarly, G\"odel's incompleteness theorems \cite{godel} could be interpreted to
 mean that there is an infinity of mathematics to be discovered.
 How can we explore mathematics if we are unable to easily incorporate different points of view?
 How do we analyse different theories and algorithms, side by side?

 So, the major obstacle to the acceptance and use of infinitary mathematics is the mathematical language
 itself.
 The notation $|_{x=0}$ Definition \ref{DEF205} is defined in context.
 That is, until you use it in a calculation, it is undefined.
 This is different to most definitions, which are invariant and unchanging.
 However, in order to `free' the mathematics,
 we need to manage cases where we cannot foresee the future.
 This is the responsibility of the user of the mathematics.
 We believe this is necessary, because of the complexity of the subject itself.
 This decoupling of responsibility with notation has a purpose. It can produce very
 flexible mathematics, essentially creating dynamic definitions.

 I once asked a researcher if we could compare two algorithms with different
 structure that gave the same result. The answer was an affirmative no.
 However, if we could pull apart the algorithms, and perhaps investigate
 them in the infinite domain, it may just be possible to find answers
 and show linkages which without $*G$ or hyperreals would be impossible to find.

 The goal is not necessarily harmony or unity
 in the context of 
 one set of rules or only one consistently logical mathematical framework,
 but instead connected transitions and freedom
 which can give a more holistic view.

 So we find transference describes the world of mathematics which we see well,
 and hence its development could lead to tools that every numerical scientist
 or mathematician could use.
\section{Non-reversible arithmetic and limits} \label{S05}
 Investigate and define non-reversible arithmetic in $*G$ 
 and the
 real numbers. 
 That approximation of an argument of magnitude, is arithmetic.
 For non-reversible multiplication we 
 define a logarithmic magnitude relation $\succ\!\succ$.
 Apply the much-greater-than relation $\succ$ in the evaluation of limits.
 Consider L'Hopital's rule with infinitesimals and infinities, and in a comparison
 $f \; (z) \; g$ form.

\subsection{Introduction}\label{S0501}

 Two parallel lines may meet at infinity,
 or they may always be apart.
 Infinity is non-unique. We believe different
 number systems can co-exist by the 
 non-unique nature of infinity.

 We focus on the infinite case where the largest magnitude
 dominates, $a+b=a$ where $b \neq 0$.

 For example, arbitrarily truncating a Taylor series.
 $f(x+h)|_{h=0} = f(x) + h f'(x) + \frac{h^{2}}{2}f^{(2)}(x) + \dots|_{h=0}$
 $= f(x) + h f'(x)|_{h=0}$.

 That an infinite sum of the discarded
 lower order terms 
 have no effect on the outcome is explained
 by sum convergence theory. However, if such a sum was for example
 bounded above by an infinitesimal, then the realization
 of the sum (a transfer) to real numbers would discard the infinitesimal
 terms, $\Phi \mapsto 0$.
\begin{quote}
DU BOIS-REYMOND in his journal
articles on the infinitary calculus is not much interested in a theory of sets as
such, or even explicitly in sets on lines. He is interested in the nature of a linear
continuum, but chiefly because he wants to consider a more general "continuum"
of functions. He especially concerns himself with limit processes as they occur in
a linear continuum, since he wants to consider limit processes among functions. \cite[p.110]{fisher}
\end{quote}

 In our previous sections we have identified
 the investigation of functions
 as working in a higher dimensional number system than $\mathbb{R}$,
 and fitted $*G$ to du Bois-Reymond's work,
 as a representation of the continuum that he was investigating. 

 Since we believe $*G$ is a field,
 multiplication and addition by non-zero elements are
 reversible prior to realization. For example truncation
 $*G \mapsto *G$. After this process, information is lost
 and in general you cannot go back.

 What we demonstrate is that arithmetic `is' non-reversible;
 and that this is a major aspect of analysis. Ultimately
 we see this as a way for working with transference,
 with reasoning by magnitude. 
 
 For example, we prove L'Hopital's rule Proposition \ref{P217}
 not with
 an equality but with an argument of magnitude.
 This, after reading the original translation
 of the rule is closer to the discovery than
 the formation with the mean value theorem.
 And this is a problem, that mathematics 
 is being used not in an intuitive way, but
 as a means of proof. Newton did not publish
 the Principia with his calculus,
 but the traditional geometric arguments 
 which are ``extremely" difficult to use.
 Analysis is not complete without a more detailed investigation
 into arguments of magnitude.
\subsection{Non-reversible arithmetic} \label{S0502}
The following is a complicated argument
 that associates finite mathematics
 with reversible processes,
 and
 infinitary calculus with non-reversible processes.
 This results in non-uniqueness in 
 the additive operator, the consequences of which are profound.

To start, consider addition and multiplication.
 Given $x+2$, then $(x+2)-2$ gives back $x$, the $-2$ undoing
 the $+2$ operation. Similarly for multiplication,
 if we start at $5$, $5 \times 2=10$. Reversing,
 $10/2=5$. 
 Similarly with powers, where logarithms and powers are 
 each others inverse.

However, if the operator is not reversible you cannot
 undo the operator previously applied; it is as if
 the operation has disappeared. $n^{2}+\mathrm{ln}\,n =n^{2}|_{n=\infty}$ 
 leaves no evidence of adding $\mathrm{ln}\,n$, an infinity.
 We need a number system with infinities for
 this to occur.
 The much-greater-than relation with realization explains
 this.
\bigskip
\begin{theo}\label{P058}
 $a, b \in *G$; 
 $\,( a , b \neq 0)$
   Assume transfer between $\Phi$ and $0$.
\[ \text{If } a+b=a \Leftrightarrow a \succ b. \] 
\end{theo}
\begin{proof}
 $\delta \in \Phi$;

 If $a + b = a$,
 $\frac{a}{a}+\frac{b}{a} = \frac{a}{a}$,
 $\frac{b}{a} = 0$,
 $0 \mapsto \Phi$,
 $\frac{b}{a} = \delta$,
 $a \succ b$.

If $a \succ b$, $\frac{b}{a} = \delta$.
 Consider 
 $a+b = y$,
 $\frac{a}{a} + \frac{b}{a} = \frac{y}{a}$,
 $1 + \delta = \frac{y}{a}$,
 $\delta \mapsto 0$,
 $1 = \frac{y}{a}$,
 $a = y$,
 $a+b=a$
\end{proof}
\bigskip
\begin{cor}\label{P059}
 $a, b \in *G$; $\,b \neq 0$
\[ \text{If } a+b=a \text{ then after the addition and realization in general, it is impossible to determine }b. \] 
\end{cor}
\begin{proof}
 By Definition \ref{DEF015}.
\end{proof}

This is a fundamental statement about our numbers. We use Theorem \ref{P058} 
 whenever we approximate. We approximate whenever we use infinitesimal
 or infinite numbers. 
 By considering an alternative definition
 to the much-greater-than relation
 as an infinitesimal ratio ($a \succ b$ then $\frac{b}{a} \in \Phi)$ is explained.

 If $a \succ b$ is alternatively defined\cite[p.19]{ordersofinfinity}:
 if $a - b = \infty$ then $a \succ b$.
 We cannot apply the transfer $\Phi \to 0$
 then  
 $a+b=a$ is not true in general.
   Consider $a=e^{x+1}$, $b=-e^{x}$, 
 $e^{x+1}-e^{x}=(e-1)e^{x}|_{x=\infty}=\infty$ but $e^{x+1} \neq (e-1)e^{x}|_{x=\infty}$. 
 This lacks the `$a+b=a$' property given in Theorem \ref{P058}.
 Since this paper is concerned about
 non-reversibility we do not want to lose this
 property.
\bigskip
\begin{prop}
If  $a_{n} \sim b_{n}|_{n=\infty}$
 then there exists $c_{n}$:
 $a_{n} +c_{n} = b_{n}|_{n=\infty}$
 where $c_{n} \prec a_{n}$.
\end{prop}
\begin{proof}
 $a_{n} + c_{n} \; z \; b_{n}$,
 $a_{n} \; z \; b_{n}$ as $a_{n}+c_{n}=a_{n}$,
 $z = \; \sim$
\end{proof}
\bigskip
\begin{defy}\label{DEF049}
An Archimedean number system has no infinitesimals or infinities.
 A non-Archimedean number system is not an Archimedean number system.
\end{defy}
 That is, a non-Archimedean number system has non-reversible arithmetic, see Definition \ref{DEF015}.
 A consequence of the existence of ratios being an infinitesimal or an infinity. 

A consequence of the Archimedean property are unique inverses,
 both additive and multiplicative,
 which allow unique solutions to equations involving
 addition and multiplication.

 With the infinireals, as long as we do not approximate
 via the ``realization" operation,
 and assuming non-zero numbers, we have unique inverses when
 adding and subtracting and unique inverses when
 multiplying or dividing.

However the non-Archimedean property is necessary
 for Theorem \ref{P058}. We need an infinity of possibilities
 with addition
 before we can say an operation is non-reversible.
 However Archimedean systems can have non-reversible arithmetic.
\bigskip
\begin{theo}\label{P213}
Multiplication by $0$ in $\mathbb{R}$ 
 number systems is non-reversible.
\end{theo}
\begin{proof}
 Let $x \in \mathbb{R}$. 
 Consider the equation $x \times 0=0$.
 If you then ask what was the value of $x$,
 there is an infinity of solutions in $\mathbb{R}$.
 By Definition \ref{DEF015} the multiplication is  
 non-reversible.
\end{proof}

The number 
 zero collapses other numbers through
 multiplication. What was the original number
 being collapsed? 
 For this reason 
 $0$ is considered
 separately 
 from other numbers.
 By defining $0^{0}=0$
 we can extend the number system further, see Example \ref{MEX207}.
 
 The reals, with which we are so familiar,
 form an Archimedean number system 
 which excludes infinity.  

 This exclusion of infinity in many respects is illusionary,
 as an arbitrary large number acts as an infinity.
 This is exploited in proofs by retaining the property
 of being a finite number. 
 For the statement
 ``$\forall n \gt n_{0}$ the following is true ..." implicitly defines infinity.

Excluding that infinity
 implicitly exists even when ignored,
 the Archimedean property essentially
 means that there are no infinitesimals or infinities as elements.

There is fascination that both the finite perspective
 and infinite perspective can co-exist.
 Both can
 describe, even when they are contradictory viewpoints,
 that is, they are completely different views of the
 same event.
 1. $a+b$ for finite numbers excluding zero always changes
 the sum. 2. $a+b=a$

 If we add a cent to a million dollars we
 have a million dollars plus a cent.
 Alternatively, if we treat a million dollars as the infinity,
 then adding a cent does `nothing' to the sum,
 and the sum remains unchanged, as its magnitude
 is not changed.
 Hence, the accountant and the businessman
 may have different views on the same transaction.

Looking at infinitary calculus with
 infinity as a point, 
 consider $\mathrm{ln}\,n + n^{2}=n^{2}|_{n=\infty}$.
 The term $\mathrm{ln}\,n|_{n=\infty}=\infty$ is 
 infinite, but acts like a zero when added to
 a much larger infinity, 
 a consequence of $\mathrm{ln}\,n \prec n^{2}|_{n=\infty}$. Replacing $\mathrm{ln}\,n$ with $n$
 gives a similar
 result where $n|_{n=\infty}$ acts as a zero element, 
 when $n + n^{2}=n^{2}|_{n=\infty}$. 
 The additive identity is not unique.
 It should be apparent that there is an infinity
 of additive elements.

The same can also be true with regard
 to multiplication, where
 a multiplicative identity is not
 unique. That is, $1$ is not the only
 element which, when multiplied, does not
 change the number.
\bigskip
\begin{mex}\label{MEX205} 
 The following  
 demonstrates non-unique multiplicative
 identities. Let $f = \infty, g = \infty, h = \infty$.  
 Consider 
$f \cdot g \;\; z \;\; h$,  
$\mathrm{ln}(f \cdot g) \;\; (\mathrm{ln}\,z) \;\; \mathrm{ln}\,h$,
$\mathrm{ln}\,f + \mathrm{ln}\,g \;\; (\mathrm{ln}\,z) \;\; \mathrm{ln}\,h$. 
 Let  $\mathrm{ln}\,f \succ \mathrm{ln}\,g$, then  
$\mathrm{ln}\,f + \mathrm{ln}\,g = \mathrm{ln}\,f$ and 
$\mathrm{ln}\,f \;\; (\mathrm{ln}\,z) \;\; \mathrm{ln}\,h$.  
 When reversing the process is possible,
 $f \;\; z \;\; h$ and   
 $g$ is a multiplicative identity.
\end{mex} 
\bigskip
\begin{defy}\label{DEF015}
We say that an arithmetic is non-reversible if there
 is an infinite number of
 additive identities
 or there is an infinite number
 of multiplicative identities.

Let $g(n)$ be one of the infinitely
 many additive identities.
 Let $h(n)$ be one of the infinitely 
 many multiplicative identities.

If one of the following is true, then
 we say that the arithmetic is non-reversible arithmetic.
\[ f(n) + g(n) = f(n) \text{ or } f(n) h(n) = f(n) \]
\end{defy}
 Consider when lower order terms realized? 

Theorem \ref{P058}
 proves non-reversible addition. 
 From this, 
 Theorem \ref{P215} proves non-reversible multiplication.
 These operations, particularly addition have wide applicability.
 However, we stress that this is only one possibility at
 infinity when via algebra and a transfer, $\Phi \mapsto 0$.
 That at infinity we can have non-uniqueness 
 is further exploration of what constitutes the continuum. 

 While Theorem \ref{P213} multiplication by $0$ in $\mathbb{R}$
 is non-reversible, it is isolated. By excluding the
 element $0$ in $\mathbb{R}$, multiplication and division
 are reversible. $\infty \times c = \infty$ when
 $c \neq 0$ has a similar property, but in $\mathbb{R}$,
 $\infty$ is excluded. 
 $0 \times \infty$ is better considered in $*G$ before
 the infinitesimals or infinities are realized.
 The order in which numbers are realized matters.
 
 Generally addition is easier for reasoning with non-reversible
 arithmetic than multiplication;
 probably a consequence of multiplication being defined as
 repeated addition, a more complex operation.
 Hence, for example, when applicable, $\prod_{k=1}^{n}a_{k} = e^{\sum_{k=1}^{n} \mathrm{ln}\,a_{k}}$ transforms a product to a power with a sum for reasoning. 

 Calculus has
 been extended in many ways to use
 infinity in calculations and theory
 because infinity is so useful.
 Calculus often states the numbers in $\mathbb{R}$
 but then reasons in $*G$. Limits are a typical example.
 We generally agree with this as the utility of calculation
 is paramount.
 
 However, just as we discuss the atom in teaching physics,
 because 
 it describes fundamental properties,
 a teaching of
 the infinitesimal is warranted.  
 Not the exclusion of its existence, which is fair
 to say is the current practice.  

 Infinitary calculus facilitates arguments 
 with magnitude, and has the potential
 to reduce the use of inequalities
 in analysis, 
 beyond what Non-Standard-Analysis or present analysis does. 
 This could make problem-solving more accessible 
 by reducing the technical difficulties associated with
 the use of inequalities.

 For example, one may experience 
 the problem of not knowing or not using the right
 inequality, and become stuck; without
 specific knowledge, no progress is likely to
 be made. This applies to specialized domains
 where networking may be required.
 If, however, the problem could be done without
 ``networking", time would be saved.

 As well as providing alternative arguments
 to problems, infinitary calculus can
 be combined with standard mathematical
 arguments and inequalities,
 leading to them being used in new ways.
 
Well how can this be achieved?  

Simply put, by employing non-reversible arithmetic;
 that is, in number systems with non-Archimedean properties,
 non-reversible arguments can be made.
 This is indeed possible
 with the
 much-less-than $( \prec )$ and
 much-greater-than $(\succ)$
 relations defined by du Bois-Reymond. 

 We have an equation
 that we would like to solve
 hence the need
 for a new number system.
 Considering the scales of infinity,
 there does not exist $c \in \mathbb{R}$
 with such a property to move
 between the infinities because $c$ is
 finite. 
\bigskip
\begin{mex}\label{MEX058}
Let $c \in \mathbb{R}$ then $c \prec \infty$,
$\not\exists c:$ 
$cn^{2} = n^{3}|_{n=\infty}$. 
 Solving in $*G$, $cn^{2} = n^{3}|_{n=\infty}$,
 $c = n|_{n=\infty}$, $c \in \Phi^{-1}$.
 This implies $c \not\in \mathbb{R}$.
 If we restricted $c$ to $\mathbb{R}$
 then the equation would not have been solved,
 as $c$ was infinite.
\end{mex}

 An Archimedean number system cannot solve this.
 Just as we needed $i$ to solve $x^{2}=-1$,
 we need a non-Archimedean number system
 to solve the equation with
 infinities.

The definition of $\prec$ and little-o 
 distinguish between infinities.
 $a \prec b$ if $\frac{a}{b} \in \Phi$.
 If two functions differ infinitely through
 division, then they must also differ infinitely
 through addition to a greater degree
 (See Theorem \ref{P058}).

 Depending on the context, 
 lower-order-magnitude terms may be disregarded.
 $f(x) + g(x) = g(x)|_{x=\infty}$. $f(x)=\infty$, $g(x)=\infty$.
 Here $f(x)|_{x=\infty}$ acts as a zero 
 identity element,
 even though $f(x)$ is not 
 zero.
 However $f(x)$ has its magnitude
 dwarfed by
 the much larger $g(x)$, so 
 $f(x)$ is  
 negligible.   

 To avoid summing infinite
 collections of terms, 
 the general restriction when applying the simplification
 $a_{n}+b_{n}=a_{n}|_{n=\infty}$
 for infinitary or infinitesimal $a_{n}$ and $b_{n}$ is that the rule is only good for a finite sum
 of infinitesimals or infinities. 

 A sum of infinities or infinitesimals to infinity, can
 itself step up in orders. 
 In other words, generally apply the simplification to a finite
 number of times.
 If you were to sum infinities
 or infinitesimals to infinity,
 you would need to integrate instead.
 Or apply truncation when sum convergence is known.

 This is discussed in detail \cite[Convergence sums ...]{cebp2}. The
 assumption of independence of sums
 may by invalid at infinity \cite{cebp8}.
 Briefely, how we view finite mathematics may be completely different
 to how we view mathematics ``at infinity" because it is a 
 much larger space.

The advantages of such simplification 
 can allow classes of functions to be reduced.
\bigskip
\begin{mex}\label{MEX059}
 $\{ \frac{1}{x^{2} + \pi}, \frac{1}{x^{2} -3x}, \ldots \}|_{x=\infty}$
 simplify to considering $\frac{1}{x^{2}}|_{x=\infty}$.
\end{mex}

Applications include taking the limit which applies
 in summing infinitesimals to zero,
 hence truncating a series. 
The 
 arguments
 can apply to diverging
 sums as well. 

Use of extended calculus in $*G$ as a heuristic, 
 for developing algorithms.
\bigskip
\begin{mex}\label{MEX206} 
  Developing an algorithm for
  approximating $\sqrt{2}$. 
  $x, x_{n}, \delta_{n} \in *G$;
  $\,\delta \in \Phi$. 
  \begin{align*}
    (x+\delta)^{2} =2 \\ 
    x^{2} + 2x \delta + \delta^{2} =2|_{\delta=0} \\ 
    x^{2} + 2x \delta  =2|_{\delta=0} 
    & \tag{$2x \delta \succ \delta^{2}|_{\delta=0}$}  \\ 
    x_{n}^{2} + 2x_{n} \delta_{n}  =2|_{\delta=0} & \tag{Developing an iterative scheme.} \\ 
    \delta_{n} = \frac{1}{x_{n}} - \frac{x_{n}}{2} & \tag{Solving for $\delta_{n}$} \\ 
 x_{n+1} = x_{n} + \delta_{n}  &  \tag{Progressive sequence of $x_{n}$ } \\
  \end{align*}
If $\delta_{n} \to 0$ then at infinity $\delta_{n}$ becomes an infinitesimal(see \cite[Example 2.12]{cebp4},
 that is $\delta_{n}|_{n=\infty} \in \Phi$,
 then  $(x+\delta)^{2} \simeq 2$ is solved, and $x_{n+1} \simeq x_{n} + \delta_{n}$ has $x_{n}$ converging.  
 Starting the approximation with $x_{0}=1.5$, $x_{5}$ 
 is correct to $47$ places
 (for a numerical calculation with Maxima see \cite{maxima47}), where the algorithm
 was transferred from $*G$ to $\mathbb{R}$.  
 $(*G, \delta_{n}, x_{n}) \mapsto (\mathbb{R}, \delta_{n}, x_{n})$,
 an infinitesimal was promoted to a real number.
\end{mex}

Turning towards
 the number system, 
 what is common is
 the operation
 of numbers at zero or infinity.
 That is where the numbers
 display non-Archimedean properties.

In fact, all function evaluation
 is at zero or infinity.
 Simply shift the origin.
 Zero and
 infinity form a number system, $\mathbb{R}_{\infty}$.
 The cardinality of $\mathbb{R}_{\infty}$ 
 is infinitely larger than the cardinality of $\mathbb{R}$.
 Then the gossamer number system $*G$ is much larger
 than the real number system(the reals which are embedded within it). 

In this paper, addition simplification
 is applied to solving  
 relations
 by converting a series of relations 
 to a sum, where lower order
 terms are discarded and the relations solved.
\bigskip
\begin{defy} \label{DEF207}
 Using the Iverson bracket
 notation (see \cite[p.24]{graham})
\[
[ f \; z \; g ]
 = \left\{ 
  \begin{array}{rl}
    1 & \; \text{when the relation } f \; z \; g \text{ is true,} \\
    0 & \; \text{when then relation } f \; z \; g \text{ is false.}
  \end{array} \right. 
\]
\end{defy}

In a more radical approach
 to demonstrate addition as a basis for
 building relations,
 we can define $0^{0}=0$ as another extension
 to the real number system,
 which then allows the building of the
 comparison greater than function (see Example \ref{MEX207}).

 Canceling the $0$'s is not allowed, as
 by definition this is now a non-reversible process,
 as we view either multiplication
 by $0$ or multiplication by $\frac{0}{0}$
 as collapsing the number to $0$.
 We also get to test if a number is zero or not.
\bigskip
\begin{mex}\label{MEX207} Non reversible mathematics to build the relations. See Definition \ref{DEF207}
\[ [x \gt 0 ] = \frac{x + |x|}{2x}\]
 When $x=0$, $\frac{0+|0|}{2 \times 0} = \frac{0}{0}=0$. 
 When $x \gt 0$, $\frac{x+|x|}{2x} = \frac{2x}{2x}=1$. 
 When $x \lt 0$, $\frac{x+|x|}{2x} = \frac{0}{2x}=0$.
\[ [x \neq 0 ] = x^{0} \]
 When $x = 0$, $0^{0}=0$. When $x \neq 0$, $x^{0}=1$. 
\end{mex}

\subsection{Logarithmic change} \label{S0503}
We introduce a relation, for better explaining
 magnitudes, and their comparison.
 While $f \succ g$ describes an infinity in
 the ratio of $f$ and $g$, there could be a much larger change
 in the functions themselves.

If we consider the operations of addition,
 multiplication as repeated addition,
 a power as repeated multiplication,
 all these operations accelerate change.
 Conversely, subtraction, division, and logarithms
 of positive numbers to greater degrees
 decelerate change. 

 Consider the log function as
 undoing change, then applying to both sides of a 
 relation and comparing,
 we can determine a much-greater than relationship,
 and `if' one exists, we infer an
 infinity between the functions.

 For example, in solving
 $f \; z \; g$
 if we find a
 $\mathrm{ln}\,f \succ \mathrm{ln}\,g$ relationship.
 The undoing log operation revealed a much-greater-than relationship.
\bigskip
\begin{defy}\label{DEF079}
 We describe a logarithmic magnitude.
 \[ \text{We say } f \succ\!\succ g \text{ when } \mathrm{ln}\,f \succ \mathrm{ln}\,g \]
 \[ \text{We say } f \prec\!\prec g \text{ when } \mathrm{ln}\,f \prec \mathrm{ln}\,g \]
 \[ \text{We say } f \text{ `log dominates' } g \text{ when } f \succ\!\succ g. \]
 \[ \text{We say } f \text{ is `log dominated' by } g \text{ when } f \prec\!\prec g. \]
\end{defy}

 The much-greater-than relation
 $f \succ g$ may have a
 log dominating relation $f \succ\!\succ g$ or
 be log dominated by $g$: $f \prec\!\prec g$ or
 no such relationship.
 That is the relations between
 the magnitude and the logarithmic magnitude 
 are not necessarily in the same direction.  
 An exception is when both positively diverge,
 see Proposition \ref{P214}.
 
 A logarithmic magnitude  
 is like a derivative.
 A derivative's sign
 is not necessarily the same
 as the function's sign.
 The much greater than
 relation is independent 
 in direction to the log dominating relation.

 Logarithmic magnitude can describe
 non-reversible product arithmetic 
 (Definition \ref{DEF015}).
 $0 \cdot \infty$
 indeterminate case arises in the calculation
 of the limit.
 We prove the non-reversible product as a consequence of
 non-reversible addition (see Theorem \ref{P215}).
\bigskip
\begin{mex} Consider $x^{n} \; z \; n|_{n=\infty}$
 where $|x| \lt 1$.
\begin{align*}
 x^{n} \; z \; n|_{n=\infty} \\
 \mathrm{ln}(x^{n}) \; (\mathrm{ln}\,z) \; \mathrm{ln}\,n|_{n=\infty} \\
 n \, \mathrm{ln}\,x \; (\mathrm{ln}\,z) \; \mathrm{ln}\,n|_{n=\infty} \\
 n \, \mathrm{ln}\,x \succ \mathrm{ln}\,n|_{n=\infty} \tag{By Definition \ref{DEF079}} \\
 x^{n} \succ\!\succ n|_{n=\infty} 
\end{align*}
 While $x^{n} \prec n|_{n=\infty}$ as $0 \prec \infty$,
 the logarithmic magnitude of $x^{n}$ is much greater than $n$
 with $x^{n} \succ\!\succ n|_{n=\infty}$.
\end{mex}

 When simplifying products by non-reversible arithmetic,
 for example in the calculation of limits, 
 rather than solve with products, reason
 by exponential and logarithmic
 functions which are each other inverses,
  converting the problem of multiplication to one with addition.
 $f \cdot g = e^{\mathrm{ln}(f \cdot g)}$
 $= e^{\mathrm{ln}\,f + \mathrm{ln}\,g}$.
 If possible, apply non-reversible arithmetic:
 $\mathrm{ln}\,f + \mathrm{ln}\,g = \mathrm{ln}\,f$
 or
 $\mathrm{ln}\,f + \mathrm{ln}\,g = \mathrm{ln}\,g$.
\bigskip
\begin{mex}\label{MEX208}
When $|x| \lt 1$, evaluate $x^{n}\cdot n|_{n=\infty}$.

This is an indeterminate form $0 \cdot \infty$.
$x^{n} \cdot n = x^{n}|_{n=\infty}$ is harder
 to understand than when the problem is reformed
 and when simplifying, non-reversible arithmetic applied 
 on a sum and not a product.
\begin{align*}
 x^{n}\cdot n|_{n=\infty} \\
 = e^{\mathrm{ln}(x^{n}\cdot n)}|_{n=\infty} \\
 = e^{n \, \mathrm{ln}\,x + \mathrm{ln}\,n}|_{n=\infty} \tag{$n \succ \mathrm{ln}\,n$ then apply non-reversible arithmetic  }  \\
 = e^{n \, \mathrm{ln}\,x }|_{n=\infty} \tag{$n \, \mathrm{ln}\,x + \mathrm{ln}\,n = n \, \mathrm{ln}\,x|_{n=\infty}$} \\
 = x^{n}|_{n=\infty}
\end{align*}
\end{mex}
\bigskip
\begin{theo}\label{P215} 
 Non-reversibility in a product. 
 Let $a$ and $b$ be positive. 
\[ \text{If } a \succ\!\succ b \text{ then } a \cdot b =a \]
\end{theo}
\begin{proof}
$ab$
 $= e^{\mathrm{ln}(ab)}$
 $= e^{\mathrm{ln}\,a + \mathrm{ln}\,b}$
 $= e^{\mathrm{ln}\,a}$
 $= a$, as $a \succ\!\succ b$ then $\mathrm{ln}\,a \succ \mathrm{ln}\,b$.
\end{proof}
\bigskip
\begin{mex}\label{MEX209} If we know the log magnitude relationship,
 we may directly calculate. \\
$x^{n} \cdot n|_{n=\infty}$
 $= x^{n}|_{n=\infty}$
 as $x^{n} \succ\!\succ n|_{n=\infty}$
\end{mex}
\bigskip
\begin{prop}\label{P214}
 Let $f=\infty$, $g=\infty$.
If $f \succ\!\succ g$ then $f \succ g$.
\end{prop}
\begin{proof}
$f \succ\!\succ g$ then $\mathrm{ln}\,f \succ \mathrm{ln}\,g$.
 Since there is no smallest infinity,
 $f_{2} = \mathrm{ln}\,f = \infty$,
 $g_{2} = \mathrm{ln}\,g = \infty$.
 $f_{2} \succ g_{2}$.
 By the following theorem: $a = \infty$, $b=\infty$,
 if $a \succ b$ then $e^{a} \succ e^{b}$
 (see Table \ref{FIG05}) then
 $f_{2} \succ g_{2}$, 
 $e^{f_{2}} \succ e^{g_{2}}$,
 $f \succ g$.
\end{proof}
\subsection{Limits at infinity}\label{S0504}
\begin{defy}\label{DEF088}
In context, we say $f(x)|_{x=\infty}$ then $\mathrm{sup} \lim\limits_{x \to \infty} f(x)$,
 similarly  $\mathrm{inf} \lim\limits_{x \to \infty} f(x)$
\end{defy}
 When the definition is put into a context such as a relation, 
 since the condition is assumed to be true for all $n$ at infinity
 (else the condition is false and a contradictory statement),
 the exact lower and upper bound 
 language can optionally be excluded. 
\bigskip
\begin{mex}
 Condition $\mathrm{inf} \lim\limits_{n \to \infty} \rho_{n} \gt 1$ becomes $\rho_{n}|_{n=\infty} \gt 1$.

 Similarly condition $\mathrm{sup} \lim\limits_{n \to \infty} \rho_{n} \lt 1$ becomes $\rho_{n}|_{n=\infty} \lt 1$.
\end{mex}

The following demonstrates the application
 of the notation and ideas discussed
 in this paper about limits
 and the more general 
 at-a-point
 evaluation.

In computation of limits, 
 infinity can be as useful in simplifying
 expressions as
 infinitesimals.
 So rather than 
 dividing and forming the infinitesimals,
 instead apply arguments of magnitude.
 Let the user choose.
 The non-reversible arithmetic works
 either way.
\bigskip
\begin{mex}\label{MEX024} 
A simple example will show this.
 $\frac{ 3n + 5}{5n}|_{n=\infty}$
 $=\frac{3n}{5n}|_{n=\infty}=\frac{3}{5}$
 The justification being $3n+5=3n|_{n=\infty}$
\end{mex}
\bigskip
\begin{mex}\label{MEX025} 
 Apostol 
 \cite[3.6.7]{apostol},
$\lim\limits_{x \to 0} \frac{ x^2-a^2}{ x^2 + 2ax + a^2}$, 
 $a \neq 0$,  
 $\frac{ x^2-a^2}{ x^2 + 2ax + a^2}|_{x=0}$,
 $= \frac{ -a^{2}}{a^{2}} = -1$
 as $-a^{2} \succ x^{2}|_{x=0}$ and similarly
 $a^{2} \succ 2ax \succ x^{2}|_{x=0}$
\end{mex}
\bigskip
\begin{mex}\label{MEX026} 
 Apostol 
 \cite[7.17.28]{apostol},
$\lim\limits_{x \to \infty} \{ (x^5+7x^4+2)^c - x \}$,  
 Find $c$ for non-zero limit (
 $c=0$ may collapse to $1-x|_{x=\infty}$).
 Using $x^{5}+7x^{4}+2=x^{5}|_{x=\infty}$
 as $x^{5} \succ x^{4} \succ x^{0}|_{x=\infty}$, 
 $(x^5+7x^4+2)^c - x |_{x=\infty}$ 
 $= x^{5c} - x |_{x=\infty} = b$
 then $c=\frac{1}{5}$
 as the difference reduces the power by one to a finite value.
 I.e. a limit.
\end{mex}
\bigskip
\begin{mex}\label{MEX027} 
$( x^5 (1 + \frac{7}{x} + \frac{2}{x^5}))^{\frac{1}{5}} -x |_{x=\infty}$ 
$= ( x^5 (1 + \frac{7}{x} ))^{\frac{1}{5}} -x |_{x=\infty}$ \\ 
expand with the binomial theorem. 
$(1+x)^w = 1 + w x + w(w-1) \frac{x^2}{2!} + \ldots$,  
$(1 + \frac{7}{x})^{\frac{1}{5}} = 1 + \frac{1}{5} \frac{7}{x} + \frac{1}{5} \frac{-4}{5} \frac{7^2}{x^2} \frac{1}{2} + \ldots$  
 then 
$x(1 + \frac{7}{x})^{\frac{1}{5}} - x|_{x=\infty}$
 $= x + \frac{7}{5}  + \frac{1}{5} \frac{-4}{5} \frac{7^2}{x} \frac{1}{2} + \ldots - x|_{x=\infty}$ 
$= \frac{7}{5}  + \frac{1}{5} \frac{-4}{5} \frac{7^2}{x} \frac{1}{2} + \ldots|_{x=\infty}$ 
$= \frac{7}{5}$ 
\end{mex}

 Arguments of magnitude are commonly used in calculations.
 Apostol 
 \cite[pp 289--290]{apostol}
 discusses
 polynomial approximations
 used in the calculation
 of limits, where
 the relation is within
 the little-o variable.  

Computing the terms
 separately lead to
  the indeterminate form $0/0$;
 by computing the numerator
 and denominator as a coupled problem, 
 leading magnitude terms may be
 subtracted (for example through factorization). 

 When we remove little-o the
 calculation is not cluttered. 
 If you need
 to be exact include it, but if
 not then it may as well be omitted.

Using the identity $\frac{1}{1-x}= 1 + x + x^{2} + \ldots$, 
$\frac{1}{1 - (\frac{x^{2}}{2} - o(x^3))}$ 
 $= 1 + \frac{1}{2}x^{2} - o(x^{3})$ as $x \to 0$
 becomes 
 $\frac{1}{1-\frac{x^{2}}{2}} = 1 + \frac{x^{2}}{2}|_{x=0}$
 Truncation is part of calculations context
 and assumed to be the case.

Applying the at-a-point notation
 to some limits.
Since the series expressions
 have terms forming a scale of infinities,
 often only a fixed number of terms
 with the expansions need be used.
 Taylor series, the binomial expansion,
 trigonometric series and others
 can be viewed as not unique since
 they have an infinity of terms.
\bigskip
\begin{mex}\label{MEX028} 
$\frac{a^{x}-b^{x}}{x}|_{x=0}$
$=(e^{x\,\mathrm{ln}\,a}-e^{x\,\mathrm{ln}\,b} ) \frac{1}{x}|_{x=0}$, expanding
 the exponential series for the first three terms, 
$\frac{a^{x}-b^{x}}{x}|_{x=0}$
$= (1+ x\,\mathrm{ln}\,a + (x\,\mathrm{ln}\,a)^{2}\frac{1}{2} - (1 + x\,\mathrm{ln}\,b + (x\,\mathrm{ln}\,b)^{2} \frac{1}{2} ) \frac{1}{x}|_{x=0}$ 
$= (x\,\mathrm{ln}\,a  - x\,\mathrm{ln}\,b ) \frac{1}{x}|_{x=0}$ 
$= \mathrm{ln}\,a  - \mathrm{ln}\,b$ 
$= \mathrm{ln}\frac{a}{b}$ 
\end{mex}

 With known algebraic identities,
 such $n^{\frac{1}{n}}|_{n=\infty}=1$
 or  
 $e = \frac{n}{(n!)^{\frac{1}{n}}}|_{n=\infty}$  
 can easily be used to
 solve limits.
\bigskip
\begin{mex}\label{MEX029} 
 \cite[p.39, 2.3.23.b]{kaczor}
 , 
$(\frac{(n!)^{3}}{n^{3n}e^{-n}})^{\frac{1}{n}}|_{n=\infty}$
$= e(\frac{(n!)^{3}}{n^{3n}})^{\frac{1}{n}}|_{n=\infty}$
$= e((\frac{n!}{n^{n}})^{\frac{1}{n}})^{3}|_{n=\infty}$
$= e(\frac{(n!)^{\frac{1}{n}}}{n})^{3}|_{n=\infty}$
$= e(e^{-1})^{3}$
$=e^{-2}$
\end{mex}

The typical interchange between zero and infinity is useful.
\[ x|_{x=0^{+}} = \frac{1}{n}|_{n=\infty} \text{ then } \\
 f(x)|_{x=0^{+}} = f(\frac{1}{n})|_{n=\infty} \]
\smallskip
\begin{mex}\label{MEX030} 
 \cite[7.17.18]{apostol} 
$\lim\limits_{x \to 0^{-}} (1-2^x)^{\mathrm{sin} \, x}$ 
$= \lim\limits_{x \to 0^{-}} (1-2^x)^{x}$
 as $\mathrm{sin}(x)$
$= x$ for small $x$. 
 Show $y=1$. 
 Let $y \in *G: y = (1-2^x)^{x}|_{x=0^{-}}$,
$\mathrm{ln}\,y =  x \,\mathrm{ln}(1-2^x)|_{x=0^{-}}$,
 $0 \cdot \infty$ form.

 With a log expansion, the problem can can be solved.
 $\mathrm{ln}(1-2^{x})|_{x=0^{-}}$
 $= 2^{x} + \frac{(2^{x})^{2}}{2} +  \frac{(2^{x})^{3}}{3} + \ldots|_{x=0^{-}}$
 $= 2^{x}|_{x=0^{-}}$.
$\mathrm{ln}\,y =  x \,\mathrm{ln}(1-2^x)|_{x=0^{-}}$
 $= x 2^{x}|_{x=0^{-}}$,
 $y = ( e^{ 2^{x}})^{x}|_{x=0^{-}}$
 $= ( e^{1})^{x}|_{x=0^{-}}$
 $=1$
\end{mex}
\bigskip
\begin{mex}
Solve $\delta ^{\delta} = y$, $\delta \in \Phi$.
 $\mathrm{ln}(\delta ^{\delta}) = \mathrm{ln}\,y$,
 $\delta \, \mathrm{ln}\,\delta = \mathrm{ln}\,y$.
 Noticing $\delta \cdot \mathrm{ln}\,\delta = 0 \cdot -\infty = \frac{1}{\infty} \cdot - \infty$, which 
  can be expressed as $-\infty/\infty$ and differentiated using L'Hopital's rule,
 $\delta \, \mathrm{ln}\,\delta$ 
 $= \frac{ \mathrm{ln}\,\delta}{1/\delta}$
 $= \frac{ 1/\delta}{-1/\delta^{2}}$
 $= -\delta = 0$ by $(*G,\Phi) \mapsto (\mathbb{R},0)$, $0=\mathrm{ln}\,y$ then $y=1$.
\end{mex}

Since the magnitude $\{ \prec, \preceq \}$ and other
 relations 
 are defined in terms of ratios, 
 when comparing two functions
 in a ``multiplicative sense",
 we convert between the fraction and
 the comparison. 
\bigskip
\begin{prop} \label{P014}
 When $\frac{f}{g} \in *G$ and $z$ is defined in a ``multiplicative sense", $g \neq 0$  
 \[ \frac{f}{g}\; z \; 1 \Leftrightarrow f \; z \; g \] 
\end{prop}
\begin{proof}
 Since no information is lost, and the operation is reversible,
 $\frac{f}{g} \cdot g \; z \; 1 \cdot g$,
 $f \; z \; g$.
\end{proof}

 Consider the problem process $\frac{f}{g} \Rightarrow f \; z \; g$.
 $\frac{f}{g}=1$,
 $(*G = 1) \mapsto$
 $(*G = *G)$.
 Though non-uniqueness also has advantages.
 $(*G \; z \; 1) \mapsto$
 $(*G \; z \; *G)$.

 It is common for a problem to be phrased as if in
 $\mathbb{R}$ but in actuality is in $*G$.
 Then after algebraically manipulating in
 $*G$, the information has to transfer back
 to $\mathbb{R}$, or be phrased as such.

In fractional form $\frac{f}{g}$
 then becomes particularly convenient to apply
 what we know about fractions to the comparison. For
 example we may transform the comparison to a point
 where L'Hopital's rule can be applied.
 With the extended number system $*G$,
 the indeterminate forms 
 $0/0$ and $\infty/\infty$
 are expressed as $\Phi/\Phi$ and $\Phi^{-1}/\Phi^{-1}$
 respectively.

 We proceed with another proof of L'Hopital's rule where
 we use infinitary calculus theory and work in $*G$.
 L'Hopital's argument is interpreted in $*G$
 (see Proposition \ref{P217})
 and the algebra is explained with non-reversible arithmetic,
 directly calculating the ratio (see \cite{hopital2}).
\bigskip
\begin{lem}\label{P216}
 A ratio of infinitesimals
 is equivalent to a ratio of infinities.
 $\frac{ \Phi}{\Phi} \equiv \frac{ \Phi^{-1} }{ \Phi^{-1}}$
\end{lem}
\begin{proof}
 $f, g \in \Phi^{-1}$;
 $\frac{f}{g} = \frac{ \frac{1}{g} }{ \frac{1}{f} }$, but
 ; $\frac{1}{g}, \frac{1}{f} \in \Phi$;
 then $\frac{f}{g}$ of the form 
 $\frac{ \Phi}{\Phi}$.
 The implication in the other direction
 has a similar argument.
 $a, b \in \Phi$;
 $\,\frac{a}{b} = \frac{ \frac{1}{b} }{ \frac{1}{a} }$,
 but ; $\frac{1}{b}, \frac{1}{a} \in \Phi^{-1}$;
 then $\frac{a}{b}$ of the form 
 $\frac{ \Phi^{-1}}{\Phi^{-1}}$.
\end{proof}
\bigskip
\begin{prop}\label{P217}
 $f, g \in \Phi$;
 If 
$\lim\limits_{x \to a} \frac{ f'(x)}{g'(x)}$ exists then
 $\frac{f(a)}{g(a)} = \frac{f'(a)}{g'(a)}$
\end{prop}
\begin{proof}
 Since $\lim\limits_{x \to a} \frac{ f'(x)}{g'(x)}$ exists, we can vary about
 $x=a$ in $f$ and $g$.
 $\frac{f(a)}{g(a)}$
 $= \frac{f(a+h)}{g(a+h)}|_{h=0}$
 $= \frac{ f(a) + f'(a)h }{g(a) + g'(a)h}|_{h=0}$.
 Choose $h \in \Phi:$ $f'(a)h \succ f(a)$
 and $g'(a)h \succ g(a)$.
 Then
 $f(a) + f'(a)h = f'(a)h|_{h=0}$,
 $g(a) + g'(a)h = g'(a)h|_{h=0}$.
 $\frac{ f(a) + f'(a)h }{g(a) + g'(a)h}|_{h=0}$
 $= \frac{f'(a)h}{g'(a)h}|_{h=0}$
 $= \frac{f'(a)}{g'(a)}$
\end{proof}
\bigskip
\begin{theo}\label{P218}
 L'Hopital's rule (weak) in $*G$.
\end{theo}
\begin{proof}
 The indeterminate form $0/0$
 is represented by $\Phi/\Phi$ in 
 $*G$.
 A transfer $\Phi \mapsto 0$
 confirms this.
 The indeterminate form $\infty/\infty$
 is represented as $\Phi^{-1}/\Phi^{-1}$.
 Similarly a transfer $\Phi^{-1} \mapsto \infty$
 confirms this.

 The indeterminate form
 $\Phi^{-1}/\Phi^{-1}$
 by Lemma \ref{P216} can
 be transformed to the indeterminate
 form $\Phi/\Phi$.

 Apply Proposition \ref{P217} to
 the indeterminate form $\Phi/\Phi$.
\end{proof}
\bigskip
\begin{theo} \label{P015}
 Comparison form of 
 L'Hopital's rule. If
 $f/g$ is in indeterminate form
 $\frac{\Phi}{\Phi}$ or $\frac{\Phi^{-1}}{\Phi^{-1}}$,
 when $f'/g'$ exists 
 then $f \; z \; g \Rightarrow f' \; z \; g'$ 
 where $z \in \{ \prec, \propto, \succ \}$.
\end{theo}
\begin{proof}
 Equivalent to L'Hopital's rule.
 See Theorem \ref{P218}.
\end{proof}
\bigskip
\begin{mex}\label{MEX056} 
$f' \succ g' \Rightarrow f \succ g$. $1 \succ \frac{1}{n}|_{n=\infty}$,
 $n \succ \mathrm{ln}\,n|_{n=\infty}$.
\end{mex}
\bigskip
\begin{mex}\label{MEX031} 
Applying
 L'Hopital's rule reaches a $0 \;\; z \;\; \infty$ form.
 Hence a much greater than or much less than relationship. 
 Solve $\mathrm{ln}\,x \;\; z \;\; x^{2}|_{x=\infty}$.
 This is in indeterminate form $\infty \;\; z \;\; \infty$,
 differentiate.  $\frac{1}{x} \;\; (D z) \;\; 2x|_{x=\infty}$,
 $\frac{1}{x} \prec 2x|_{x=\infty}$, $D z = \; \prec$,
  $z = \int\!\!\prec \; = \; \prec$ then $\mathrm{ln}\,x \prec x^{2}|_{x=\infty}$
\end{mex}

If the limit exists then the comparison and limit are solved
 for $f(x) \propto g(x)|_{x=a}$ as a standard application
 of L'Hopitals rule with the comparison notation.
\bigskip
\begin{mex}\label{MEX032} 
Computing an indeterminate form $0/0$. 
 $\lim\limits_{x \to 2} \frac{ 3x^{2}+2x-16 }{ x^{2}-x-2 }$, 
 $3x^{2}+2x-16 \;\; z \;\; x^{2}-x-2|_{x=2}$, 
 $0 \;\; z \;\;0$ form then differentiate. 
 $6x +2 \;\; (Dz) \;\; 2x-1|_{x=2}$,  
 $14 \;\; (Dz) \;\; 3$,  
$\frac{ 3x^{2}+2x-16 }{ x^{2}-x-2 }|_{x=2} = \frac{14}{3}$ 
\end{mex}
\bigskip
\begin{mex}\label{MEX033} 
 Indeterminate form $\infty/\!-\infty$.
 $\frac{v}{\mathrm{ln}\,v}|_{v=0}$
 $= \frac{1}{1/v}|_{v=0} = v|_{v=0} = 0$
 then $v \prec \mathrm{ln}\,v|_{v=0}$.
 Since the relation could occur with
 the relation notation,
 $v \; z \; \mathrm{ln}\,v|_{v=0}$,
 $1 \; Dz \; \frac{1}{v}|_{v=0}$,
 $1 \; Dz \; \infty$,
 $Dz = \;\prec$, $z=\;\prec$.
\end{mex}
\bigskip 
\begin{mex}\label{MEX034} 
 \cite[p.8]{ordersofinfinity}
 .  Show $P_{m} \succ Q_{n}$ when $m \gt n$,
 given
 $P_{m}(x) = \sum_{k=0}^{m} p_{k} x^{k}$
 and 
 $Q_{n}(x) = \sum_{k=0}^{n} q_{k} x^{k}$ for positive coefficients. 
 Let $m = n + a$, $a \gt 0$. 
 $P_{m} \;\; z  \;\; Q_{n}|_{x=\infty}$, $\infty \;\; z \;\; \infty$,  
 $D^{n} P_{m} \;\; (D^{n} z)  \;\; D^{n}  Q_{n}|_{x=\infty}$, 
 $D^{n} P_{m} \;\; (D^{n} z)  \;\; \alpha |_{x=\infty}$, 
 $\beta x^{a} \;\; (D^{n}z) \;\; \alpha|_{x=\infty}$, 
 $\beta x^{a} \succ \alpha|_{x=\infty}$, 
 integrating $n$ times preserves this relation
 and solves for $z$.
\end{mex}

Comparison can be in a ``multiplicative sense",
 or an ``additive sense".  In the additive sense,
 we treat the expression more as a relation with
 addition and we may add and subtract, but
 drawing conclusions with much-larger-than relations
 may be problematic. 
 $2x \; z \; 3x$, 
 $0 \; z \; x$, 
 $0 \prec \infty$,
 may then mistakingly draw the conclusion $2x \prec 3x|_{x=\infty}$.
 In the multiplicative sense, divide by $x$,
 $2x \prec 3x|_{x=\infty}$,
 $2 \prec 3$ is false.
 Both comparisons are beneficial.

\bigskip 
\begin{mex}\label{MEX041} 
Consider $x^{2} \succ x|_{x=\infty}$.
 By L'Hopital, 
 $x^{2} \; z \;x$, $2x \; z \; 1$, $z = \;\succ$. 
\\
 By multiplicative $z$, 
 $x^{2} \; z \; x$, 
 $\frac{x^{2}}{x} \; z \; 1$, $x \; z \; 1$, $z=\;\succ$. 
\\
 However a variation, divide by $x$,
 $x \; z \; 1$, $1 \; z \; \frac{1}{x}$,
 realize the infinitesimal, 
 $1 \; z \; 0$, $z = \; \gt$ does not solve for $\succ$.
 In `realizing' the infinitesimal information
 is lost as in $\mathbb{R}$.
\end{mex}
\bigskip 
\begin{mex}\label{MEX042}  
 \cite[WolframMathworld]{wolfLHopital}
An occasional example where L'Hopital's rule fails.
 Applying the rule swaps the arguments to opposite sides.
 Since the relation is equality, this is indeed true.
 $\frac{u}{(u^{2}+1)^{\frac{1}{2}}}|_{u=\infty}$,
 $u \; z \; (u^{2}+1)^{\frac{1}{2}}|_{u=\infty}$,
 $\frac{d}{du}u \; (Dz) \; \frac{d}{du} (u^{2}+1)^{\frac{1}{2}}|_{u=\infty}$,
 $1 \; z \; \frac{1}{2} (u^{2}+1)^{-\frac{1}{2}} 2u|_{u=\infty}$, 
 $(u^{2}+1)^{\frac{1}{2}} \; z \; u|_{u=\infty}$.
 Applying arguments of magnitude, 
 $\frac{u}{(u^{2}+1)^{\frac{1}{2}}}|_{u=\infty}$
 $=\frac{u}{(u^{2})^{\frac{1}{2}}}|_{u=\infty}$
 $=\frac{u}{u}|_{u=\infty}=1$
\end{mex}

\section{Part 6 Sequences and calculus in $*G$} \label{S06}
 With the partition of positive integers and positive 
 infinite integers,
 it follows naturally that sequences
 are also similarly partitioned,
 as sequences are indexed on integers.
 General convergence of a sequence
 at infinity is investigated.
 Monotonic sequence testing by comparison.
 Promotion of a ratio of infinite integers
 to non-rational numbers is conjectured.
 Primitive calculus definitions with 
 infinitary calculus, epsilon-delta proof 
 involving 
 arguments of magnitude are considered.

\subsection{Introduction}\label{S0601}
The discovery of infinite integers leads to the obvious 
 existence of the infinitesimals Part 1, 
 as dividing $1$ by an infinite integer is not a real
 number.
 However, it also does so
 much more.  For theorems, the separation of
 the finite and infinite is possible, rather
 than having a single theorem which addresses both cases.

As sequences are indexed by integers, we similarly
 find that sequences can be partitioned.
 The following investigates some of the 
 mechanics of sequences, particularly at infinity.
 For example it could be that over time
 sequences supersede sets. 
 Sequences can be viewed as more primitive structures.

\subsection{Sequences and functions}\label{S0602}
 We have extended the sequence notation to include intervals,
 as we deem that `the order' is the most important property.
\bigskip
\begin{defy}\label{DEF208} Join or concatenate two sequences. 
$(a) + (b) = (a,b)$ A sequence can be deconstructed. 
 $(a, b) = (a) + (b)$ The operator $+$ is not commutative,
 $(b) + (a) = (b,a)$,  and in general $(b,a) \neq (a,b)$.
\end{defy}
\bigskip
\begin{defy}\label{DEF209}
Compare sequences component wise on relation $z$.
 $(a_{1}, a_{2}, \ldots ) \; z \; (b_{1}, b_{2}, \ldots)$ then $(a_{1} \; z \; b_{1}, a_{2} \; z \; b_{2}, \ldots )$
\end{defy}

\bigskip
\begin{mex}
While we often use functions to state a comparison of 
 sequences, we may use sequence notation.
 Here, infinite positive integers are implied.
 $(n) \prec (n^{2})|_{n=\infty}$
\end{mex}
\bigskip
\begin{defy}
 If a set has a `less than' relation, the sequence of the set is
 ordered. If $X$ is the set, let $( X )$ be the sequence of $X$ with the
 order relation.
\end{defy}
\bigskip
\begin{mex}
 $( +\Phi)$ is the ordered sequence of positive infinitesimals,
 $(\mathbb{N}_{\infty})$ is the ordered sequence
 of infinite positive integers,
 $(+\Phi^{-1})$ is the ordered sequence of positive infinite numbers,
 $(\mathbb{N})$ is the ordered sequence of natural numbers, $(\mathbb{N}_{\lt})$ ordered
 sequence of finite natural numbers.
\end{mex}

We can iterate over infinity in the following way.
 Consider the infinite sequence $1, 2, 3, \ldots$
 We can express this with a variable $n$.
 i.e. $1, 2, 3, \ldots, n, n+1, n+2, n+3, \ldots|_{n=\infty}$.
\bigskip  
\begin{defy} Any integer sequence can be composed of both
 finite and infinite integers.
\[ (1, 2, 3, \ldots, k)|_{k \lt \infty} + (\ldots, n-1, n, n+1, n+2, n+3, \ldots )|_{n=\infty} \] 
\[ (\mathbb{N}_{\lt}) + (\mathbb{N}_{\infty}) \]
\end{defy}

With the establishment of the existence of 
 the infinite 
 integers, since a sequence is indexed
 by integers we can partition the sequence into
 finite and infinite parts.
 Further all sequences
 with integer indices,
 implicitly 
 or explicitly are of this form.
 A finite sequence is deconstructed
 with no infinite part.
\bigskip
\begin{defy}\label{DEF066}
Define a sequence at
 infinity $(a_{n})|_{n=\infty}$ to iterate over the
 whole infinite interval, 
 \[(\ldots, a_{n-2}, a_{n-1}, a_{n}, a_{n+1}, a_{n+2}, \ldots)|_{n=\infty} \]
 or to count from a point onwards, generally in a positive direction.
 \[ (a_{n}, a_{n+1}, a_{n+2}, \ldots)|_{n=\infty} \]
\end{defy}

 The concept of a sequence at infinity
 is particularly important,
 as we now can separate and partition 
 finite and infinite numbers.
\bigskip
\begin{defy}\label{DEF065}
A sequence can be deconstructed
 into both finite and infinite parts.  
\[ (a_{1}, a_{2}, \ldots ) = (a_{1}, a_{2}, \ldots, a_{k})|_{k \lt \infty} + (\ldots, a_{n}, a_{n+1}, \ldots)|_{n=\infty} \]
\[ (a_{1}, a_{2}, \ldots ) = (a_{k})|_{1 \leq k \lt \infty} + (a_{\mathbb{N}_{\infty}}) \]
\[ (a_{1}, a_{2}, \ldots ) = (a_{\mathbb{N}_{\lt}}) + (a_{\mathbb{N}_{\infty}}) \]
\end{defy}

What is striking is that at infinity there is
 no minimum or maximum elements.
 If $n=\infty$ is an infinity, so is $n-1$, $n-2$, $\ldots$.
 Similarly for the continuous variable. If $x=\infty$,
 so is $x-1$, $x-2$, $\ldots$.
 While infinity has 
 no lower or upper bound,
 we may find it useful to define
 the `ideal' min and max elements, as these can 
 describe an interval.
\bigskip
\begin{defy} Ideal minimum and maximum numbers \\
 Let $\mathrm{min}(\mathbb{N}_{\infty})$ be an ideal 
 minimum of the lowest positive infinite integer. \\
 Let $\mathrm{max}(\mathbb{N}_{\infty})$ be an ideal 
 maximum of the highest positive infinite integer. \\
 Let $\mathrm{min}(+\Phi^{-1})$ be an ideal minimum of the lowest positive infinite number. \\
 Let $\mathrm{max}(+\Phi^{-1})$ be an ideal maximum of the highest positive infinite number.
\end{defy}
\smallskipneg
\[ (\ldots, n-1, n, n+1, \ldots )|_{n=\infty} = (\mathrm{min}(\mathbb{N}_{\infty}), \ldots, \mathrm{max}( \mathbb{N}_{\infty} ) ) = (\mathbb{N}_{\infty}) \]
\[ (\ldots + [x-1, x) + [x, x+1) + [x+1, x+2) + \ldots )|_{x=\infty} = ( \mathrm{min}(+\Phi^{-1}), \mathrm{max}(+\Phi^{-1}) ) = (+\Phi^{-1}) \]

 Considering $1, 2, 3, 4, \dots$, we believe that the infinity was thought of as an open set,
 $(1, 2, \ldots )$, but with infinite integers,
 this may be better expressed with an interval notation $[1, 2, \ldots, \infty]$. 
\bigskip
\begin{defy}
 An interval can be deconstructed into real and infinite real parts
\[  (\alpha, \infty] = (x)|_{\alpha \lt x \lt \infty} + (+\Phi^{-1}) \]
\end{defy}
That there is no infinite integer lower bound
 often does not matter.
 Once we arrive at infinity, we may iterate from a
 chosen point onwards.
\bigskip
\begin{mex}
 In describing the
 function $\frac{1}{n}$
 as a sequence, we often
 say $(\frac{1}{2}, \frac{1}{3}, \ldots)$
 which includes both finite and infinitesimal numbers.
 By considering the infinite sequence,
 $(\frac{1}{n})|_{n=\infty}$ we
 now are describing the infinitesimals
 only.
\end{mex}

We would like to iterate
 over infinity for various reasons.
 On occasion it is necessary
 to iterate not over all the infinities,
 but between two infinities.
 For example,
 like with NSA(Non-Standard Analysis), 
 iterate
 between two infinities
 $\omega$ and $2 \omega$.
 In our notation,
 we can start counting
 at infinity, till 
 we reach the next infinity.
 Construct an auxiliary
 sequence for this purpose.
\[ (a_{n+1}, a_{n+2}, \ldots a_{2n}) = (b_{1}, b_{2}, \ldots b_{n}) \text{ where } b_{k}=a_{n+k}|_{n=\infty} \]

From another perspective, we can use
 the same notation to iterate over all infinity.
 Iterating over infinity at infinity,
 so that the finite part is removed, $(a_{n+k})|_{k=n=\infty}$
 iterates over all the infinite elements.

While there is no lower infinite bound for
 the infinite integers, this is not really a 
 problem as
 we need not consider all infinite elements,
 but elements from a certain point onwards,
 hence Definition \ref{DEF066}. 

We require sequences at infinity when building
 other structures at infinity. 
 The ordering property of sequences is separate
 to sets, which by their definition are unordered.
 
 Sequences can be transformed and  
 or rearranged, from one sequence to another,
 with an infinity of elements, in such a way
 to guarantee a property based on
 the order.  
 Our subsequent papers \cite[Convergence sums ...]{cebp2},
 \cite{cebp4} both
 require sequences in the ideas and proofs.

 Sequences are not restricted to
 discrete variables.
 As we can consider a function as a continuous sequence of points,
 we
 extend the sequence notation to the continuous variable.
 We would then consider the index which is also a continuous variable, the domain.
\bigskip
\begin{mex}
Partition the interval, $[\alpha, \infty]$ 
 $=(x)|_{x= [\alpha, \infty]}$
 $=(x)|_{[\alpha, x \lt \infty)} + (x)|_{+\Phi^{-1}}$
 or $[\alpha, +\mathbb{R}) + (+\Phi^{-1})$.
\end{mex}
\bigskip
\begin{defy} 
 We say a function is ``monotonically increasing" if $f(x+\delta) \geq f(x)$,
 ``monotonically decreasing" if $f(x+\delta) \leq f(x)$, $\delta \in \Phi$.
\end{defy}
\bigskip
\begin{defy}
 We say a sequence is ``monotonically increasing" if $a_{n+1} \geq a_{n}$,
 ``monotonically decreasing" if $a_{n+1} \leq a_{n}$. 
\end{defy}
\bigskip
\begin{defy}\label{DEF063}
We say a sequence or function has ``monotonicity"
 if the sequence or function is monotonic: monotonically increasing
 or monotonically decreasing.
\end{defy}

 Determine if 
 a function is monotonic
 by
 comparing successive terms and solving for the
 relation. 
 For a continuous function we can often
 take the derivative. However, for sequences
 this may not be possible.
\bigskip
\begin{conjecture}
We can determine the monoticity of sequence $a_{n}|_{n=\infty}$ by  
 solving for relation $z$ in $*G$, $a_{n+1} \; z \; a_{n}|_{n=\infty}$,
 or if it exists its continuous version
 $a(n+1) \; z \; a(n)$.
\end{conjecture}
\bigskip
\begin{mex}\label{MEX021} 
Determine if the sequence $(a_{n})|_{n=\infty}$ is monotonic. $a_{n}=\frac{1}{n^{2}}$,
 compare sequential terms, 
 $a_{n+1} \; z \; a_{n}|_{n=\infty}$, 
 $\frac{1}{(n+1)^{2}} \; z \; \frac{1}{n^{2}}|_{n=\infty}$, 
 $n^{2} \; z \; (n+1)^{2}|_{n=\infty}$, 
 $n^{2} \; z \; n^{2} + 2n +1|_{n=\infty}$, 
 $0 \; z \; 2n +1|_{n=\infty}$, $z = \; \lt$,
 $a_{n+1} \lt a_{n}$ and the sequence is monotonically decreasing.
\end{mex} 
\bigskip
\begin{mex}\label{MEX022} 
Test if the sequence $(a_{n})|_{n=\infty}$ is monotonic, $a_{n} = \frac{1}{n^{\frac{1}{2}}+(-1)^{n}}|_{n=\infty}$.
 Let $j=(-1)^{n}$, 
 $a_{n} \;\; z \;\; a_{n+1}|_{n=\infty}$, 
 $\frac{1}{n^{\frac{1}{2}} + j} \;\; z \;\; \frac{1}{(n+1)^{\frac{1}{2}}-j}|_{n=\infty}$,
 $(n+1)^{\frac{1}{2}}-j \;\; z \;\; n^{\frac{1}{2}} + j|_{n=\infty}$, 
 $(n+1)^{\frac{1}{2}}- n^{\frac{1}{2}} \;\; z \;\; 2j|_{n=\infty}$, 
 using the binomial theorem
 $(n+1)^{\frac{1}{2}}- n^{\frac{1}{2}} = n^{-\frac{1}{2}}|_{n=\infty}=0$,
 $0 \;\; z \;\; 2j|_{n=\infty}$,
 $0 \;\; z \;\; (-1)^{n}|_{n=\infty}$,
 $z = \;\lt, \; \gt, \ldots$  then $(a_{n})|_{n=\infty}$ is not
 monotonic.
\end{mex}
\bigskip
\begin{mex}\label{MEX060}
An example of when not to apply infinitary argument simplification
 $a + b=a$. 

Test if the sequence $a_{n} = \frac{1}{n^{\frac{1}{2}}+(-1)^{n}}|_{n=\infty}$
 is monotonic.
 $n^{\frac{1}{2}} \succ (-1)^{n}|_{n=\infty}$,
 if we say  
 $a_{n} = \frac{1}{n^{\frac{1}{2}}+(-1)^{n}}|_{n=\infty}$
 $= \frac{1}{n^{\frac{1}{2}}}|_{n=\infty}$,
 the sequence $(\frac{1}{n^{\frac{1}{2}}})|_{n=\infty}$ is easily shown
 to be monotonic.
 However example \ref{MEX022} shows the sequence is not monotonic. 
 Even though the magnitude is infinitely small compared with the other function,
 the property of monoticity by adding $(-1)^{n}$ was changed.
\end{mex}
\subsection{Convergence}\label{S0603}
We now move on to a more theoretical 
 use of at-a-point definition. The Cauchy convergence,
 Cauchy sequence and limit,
 can be defined at infinity instead of
 both a finite and infinite perspective definition.
 Given that a number system exists
 at infinity and zero, this is more than justified,
 and since the definitions may be more primitive,
 may subsume the standard definitions. 

The problem with both the limit existence 
 and the Cauchy sequence convergence
 is that they both
 define convergence 
 to the point to include the point of convergence
 in the same space.
 While this is incredibly useful it is a subset
 of a more general convergence.

For example Hille
 \cite[p.17, Theorem 1.3.1]{hille}
 already assumes complex numbers, $z_{k} \in \mathbb{C}$ and 
 defines Cauchy convergence 
\[ | z_{m} - z_{n}| \lt \epsilon \text{ for } m, n \gt M(\epsilon). \]
Then provides the following corollary from
 the definition, expressed as a limit.
 \cite[p.71 (4.1.12)]{hille}
\[ \lim\limits_{m,n \to \infty} || z_{m} - z_{n} ||=0  \] 
However turning this around, 
 the corollary is the more primitive operation,
 that being their difference is zero.
 Make this the definition,
 defining a sequence as converging at infinity (Definition \ref{DEF211})
 and derive
 the Cauchy sequence (Definition \ref{DEF024}). 
\bigskip
\begin{defy}\label{DEF210}
Convergence is the negation of divergence.
\end{defy}
\bigskip
\begin{defy}
 A sequence with singularities diverges.
\end{defy}
\bigskip
\begin{theo}\label{P080}
A sequence without singularities before infinity can only diverge at infinity.
\end{theo}
\begin{proof}
Every finite sequence converges because the number of terms is finite
 and the terms are not singularities.
\end{proof}
\bigskip
\begin{cor}\label{P219}
 For a sequence without finite singularities,
 convergence or divergence is determined at infinity.
\end{cor}
\begin{proof}
 Since convergence is defined as the negation of divergence (Definition \ref{DEF210}),
 and divergence can only happen at
 infinity (Theorem \ref{P080}),
 then both convergence and divergence
 can only be completely determined
 at infinity.
 Note: this does not contradict
 finite sums
 converging, as at infinity
 their sum is $0$.
\end{proof}
\bigskip
\begin{defy}\label{DEF211}
A sequence $(a_{n})$ converges at infinity if
 given $\{ m, n \} \in \mathbb{N}_{\infty}$:
\[ a_{m} - a_{n} |_{\forall m,n = \infty} \simeq 0 \] 
\end{defy}

\begin{defy}\label{DEF024}
A Cauchy sequence converges if 
 the sequence $(x_{n})|_{n=\infty}$ converges (Definition \ref{DEF211}) and 
 the finite and infinite numbers are
 the same type of
 number.
 $n \lt \infty \text{ then }x_{n} \in W$ and $x_{m}|_{m=\infty} \in W$
\end{defy}

In Definition \ref{DEF211} of
 sequence convergence,
 the number types can be different
 as
 $\Phi$ is composed of the infinireals.
 Cauchy convergence Definition $\ref{DEF024}$
 derives from defining convergence Definition \ref{DEF211}. 

 Similarly the limit definition changes.
 Define evaluation at a point, then
 define the limit as the evaluation at the 
 point and in the same space.
 Definition \ref{DEF068} the limit derives from Definition
 \ref{DEF215} evaluation at a point.

 The idea of a sequence not being
 convergent because it is not
 `complete' is a narrow view.

Consider the computation of two integer
 sequences $a_{n}$ an $b_{n}$ where their
 ratio for finite values is always rational,
 but what they are approximating is not.
 The Cauchy sequence convergence does
 not explain the differing number types,
 only convergence.

The limit fails to be defined when a ratio between 
 these two sequences is considered.
 This is a simple operation. The best answer
 that can explain the calculation is that
 at infinity the ratio is promoted,
 a rational approximation at infinity can be
 promoted to a transcendental number.
\bigskip
\begin{conjecture}
 There exists ratios of infinite integers 
 of the form
 $\frac{\mathbb{N}_{\infty}}{\mathbb{N}_{\infty}}$
 which can be transfered to real numbers.
\end{conjecture}

Since all irrationals including transcendental numbers
 are calculated by integer sequences, such a restriction
 on the ratio of two integer sequences not converging is
 absurd. Such sequences do converge at infinity.
 
If $\frac{ a_{n}}{b_{n}}|_{n=\infty}$ converges
 at a point
 not in the limit. 

 If $\{ a_{n},\, b_{n} \} \in \mathbb{N}$ are integers,
 $\frac{a_{n}}{b_{n}} \in \mathbb{Q}$,
 but $\frac{a_{n}}{b_{n}} \to \mathbb{Q}'$ then
 $\lim\limits_{n \to \infty} \frac{a_{n}}{b_{{n}}}$ does not exist.
 However $\frac{a_{n}}{b_{n}}|_{n=\infty}$ does not have this restriction.
 $\frac{a_{n}}{b_{n}}|_{n=\infty} \in \frac{\mathbb{J}_{\infty}}{\mathbb{J}_{\infty}} \in \mathbb{Q_{\infty}}$,
 but $\mathbb{Q}_{\infty}$ for any non-rational number approximation
 is promoted to 
 $\mathbb{Q}'$, if the approximation exists in $\mathbb{R}$.

These rational approximations are common.
 All calculations of  
 numbers in $\mathbb{R}$ are reduced to integer
 calculations. However such calculations need to be explained
 in a higher dimension number, at least with $\mathbb{R}_{\infty}$ 
 because infinite integers are involved.
\bigskip
\begin{mex}
Construct an integer sequence to approximate $\sqrt{3}$.
 Hence, we consider $\sqrt{3}$ as the 
 ratio of two infinite integers at infinity.
 $(x-1)^{2} = 3$ has a solution $x=1+\sqrt{3}$.
 $x^{2}-2x+1 = 3$, $x = 2 + \frac{2}{x}$.
 Develop an iterative scheme.
 $x_{n+1} = 2 + \frac{2}{x_{n}}$.
 Assume an integer solution, $x_{n} = \frac{a_{n}}{b_{n}}$.
 $\frac{a_{n+1}}{b_{n+1}} = 2 + \frac{2 b_{n}}{a_{n}}$,
 $\frac{ a_{n+1}}{b_{n+1}} = \frac{2a_{n} + 2 b_{n}}{a_{n}}$.
 Let $b_{n+1} = a_{n}$ then $a_{n+1} = 2 a_{n} + 2 a_{n-1}$.
 For two initial values, $a_{0}=1$, $a_{1}=1$,
 $a_{2} = 2 a_{1} + 2 a_{0} = 2 \cdot 1 + 2 \cdot 1=4$,
 the sequence
 generated is $(1,1,4,10,28,76,208,568,1552, \ldots)$. 
 $\sqrt{3} = \frac{a_{n}}{b_{n}}-1$
 $=\frac{a_{n}+2 a_{n-1}}{a_{n}}|_{n=\infty}$.
\end{mex}

The way around this difficulty    
 by saying it is not important through
 definition is problematic.
 Simply promoting the two numbers
 being divided to the same number system
 (as demonstrated by Hille \cite[p.17, Theorem 1.3.1]{hille}),
 so that by definition and only by definition
 they are the same number type;
 and therefore the ratio converges in the same space,
 is incomplete.

The alternative Definition \ref{DEF211} 
 define the same concept 
 more generally.
 Admittedly the problem of `promotion' is not
 explained, but is 
 acknowledged.
 
While this may be a controversial finding,
 it suggests either that notions
 of convergence have not been entirely settled,
 or that they are incomplete.
 Especially for the most basic operations.

 In summary, there are two problems
 with the Cauchy sequence. It is derived from
 a more general convergence, and 
 is better explained in a space at infinity,
 with infinite integers.
\subsection{Limits and continuity} \label{S0604}
The standard epsilon definition of a limit Definition \ref{DEF212}
 can be improved by explicitly
 defining $\{ \epsilon, \delta \} \in \Phi^{+}$ as, by the
 conditions, these numbers become 
 infinitesimals.
 Therefore this is an implicit infinitesimal definition.
\bigskip
\begin{defy}\label{DEF212}
 The symbol $\lim\limits_{x \to p}f(x) =A$
 means that for every $\epsilon \gt 0$,
 there is a $\delta \gt 0$ such that
\[ |f(x)-A| \lt \epsilon \text{ whenever } 0 \lt |x-p| \lt \delta \text{.} \]
\end{defy}

 If we consider a limit definition \cite[p.129]{apostol} given
 by Apostol,
 we can generalize the definition in $*G$ to include infinitesimals,
 thereby making the definition explicit.
 A statement with infinitesimals, let $\epsilon \in +\Phi$: 
 $|f(x)-A| \lt \epsilon$
 can be equivalently expressed:
 $f(x) - A \in \Phi \cup \{ 0 \}$.
\bigskip
\begin{defy}\label{DEF213}
The symbol $\lim\limits_{x \to p}f(x) =A$ in $*G$
 means that when $x-p \in \Phi$ 
 then
 $f(x)-A \in \Phi \cup \{0\}$.
\end{defy}
\bigskip
\begin{mex}
$\lim\limits_{n \to \infty} \frac{ n^{3}+\frac{1}{n}}{4n^{3}}$
 $= \lim\limits_{n \to \infty} \frac{ 3n^{2} - n^{-2}}{12n^{2}}$
 $= \lim\limits_{n \to \infty} \frac{ 6n + 2 n^{-3}}{24n}$
 $= \lim\limits_{n \to \infty} \frac{ 6 - 6 n^{-4}}{24}$
 $= \frac{1}{4} - \frac{1}{4n}|_{n=\infty}$
\end{mex}

 The limit $\lim\limits_{n \to \infty} \frac{a_{n}}{b_{n}}$
 in $\mathbb{R}$ implicitly applies a transfer $*G \mapsto \mathbb{R}$
 Part 4.
 A limit in $*G$ (Definition \ref{DEF213})
 by default does not do a transfer, but this is easily done.
\bigskip
\begin{prop}\label{P220}
 If a limit exists in $\mathbb{R}$ then
 $*G \mapsto \mathbb{R}$:
 in $*G$, 
  $\mathrm{st} (\lim\limits_{x \to p} f(x)) = \lim\limits_{x \to p} f(x))$ in $\mathbb{R}$.
\end{prop}
\begin{proof}
 $\mathbb{R}$ is a subset of $*G$. 
 Then a transfer must exist,
 since limits are actually calculated
 in $*G$.
 During the transfer, infinitesimals
 are mapped to zero, $\Phi \mapsto 0$.
\end{proof}
 
 Limits and continuity are tied
 together in $\mathbb{R}$,
 however we will see that this is
 often not the case in $*G$.
 For example, 
 we can have
 a discontinuous staircase function
 in $\mathbb{R}$ which is continuous
 in $*G$. 

 However, we can similarly define
 continuity in $*G$ with limits.
 Since this is before the transfer
 principle $*G \mapsto \mathbb{R}$
 is applied then there is no paradox.
 We believe that the definition
 of continuity via limits
 has the advantage of an `at-a-point'
 perspective.

 Consider the `principle of variation', 
 which for a continuous variable
 is the `law of adequality' \cite[p.5]{victors}:
 $d(f(x)) = f(x+\delta)-f(x)$ 
 leads to the derivative, as a ratio
 of infinitesimals.
 $df(x) = f(x + dx)-f(x)$,
 $\frac{df(x)}{dx} = \frac{f(x+dx)-f(x)}{dx}|_{dx=0}$

 However, the principle of variation is also
 applicable to discrete change,
 where $dn = (n+1)-n=1$
 is a change in integers,
 which we interpret to derive
 a derivative of a sequence \cite{cebp4}.

 Continuity can be defined either
 by the principle of variation
 or the limit.
 Continuity can
 be expressed as a variation;  taking two points infinitesimally close, 
 and their difference is an infinitesimal.
\bigskip
\begin{defy}\label{DEF214}
 A function $f: *G \to *G$
 is continuous at $x$; 
 $f(x), x, y  \in *G$;
 $\delta x, \delta y  \in \Phi$;
\[ y = f(x) \]
\[ y + \delta y = f(x + \delta x) \]
\end{defy}
\bigskip
\begin{defy} \label{DEF215}
 A function $f: *G \to *G$
 is continuous
 at $x=a$ and
 has been evaluated to $L$: $f(x)|_{x=a}=L$,  
\[  \text{If }  \forall x: x \simeq a \text{ then }
 f(x) \simeq L. \]
\end{defy}

 $x \simeq a$ is equivalent to
 $x+\delta = a$ when $\delta \in \Phi$
 or $x=a$ Definition \ref{DEF201}.
\bigskip
\begin{lem}
 Definition \ref{DEF214} implies Definition \ref{DEF215}.
\end{lem}
\begin{proof}
 Consider Definition \ref{DEF214}.
 $y + \delta y = f(x + \delta x)|_{x=a}$, 
 $y \simeq f(x+\delta x)|_{x=a}$, 
 let $L = f(x)|_{x=a}$, 
 $L \simeq f(x+\delta x)|_{x=a}$, 
 $L \simeq f(x)|_{x \simeq a}$
\end{proof}

 A definition of a limit, less general than a definition
 of evaluation at a point is given.
\bigskip
\begin{defy}\label{DEF068} 
 A function $f: *G \to *G$
 is a limit at $x=a$ if $f(x)$ is continuous
 at $x=a$ and is the same number type $W$.
\[ \text{If } f(x)|_{x=a}=L \text{ and }  \{f(x), \, L \} \in W \text{ then }  f(x)|_{x=a} \text{ is a limit.} \] 
The symbol $\lim\limits_{x \to a}f(x) =L$ in $*G$
 means that when $x \simeq a$ 
 then
 $f(x) \simeq L$.
\end{defy}

 Considering the larger picture between the two-tiered number systems.
 What is being claimed is that providing a finite and infinite perspective
 definition does not describe well what is happening,
 particularly when ``at infinity" simplifies the explanation.
 These ideas of defining at infinity
 extend into many other definitions.

As there appear to be different kinds of arithmetics
 and convergence as governed by what happens
 at infinity, defining convergence in general at infinity
 makes more sense.
\bigskip
\begin{mex}
 Define $x \gt 1: \; \sum_{k=0}^{\infty} x^{k} = \frac{1}{1-x}$.
 Then we can have an infinite sum with positive
 terms that has a negative solution.  Let $\frac{1}{1-x}=-w$,
 solve for $x$, $x = \frac{1+w}{w}$.
 Case $w=2$, $x = \frac{3}{2}$,
 $\sum_{k=0}^{\infty} (\frac{3}{2})^{k} = -2$.

 Not only has a positive sum of terms become negative,
 we needed an infinity of terms to interpret the sum,
 for the sum to have meaning.
\end{mex}

 This example is relevant because 
 our applications following  
 this series \cite[Convergence sums ...]{cebp2}
 define convergence at infinity,
 as does Robinson's non-standard analysis
 (here infinity defined as not finite is
 interpreted as successive orders of numerical infinities $\omega$).
\subsection{Epsilon-delta proof}\label{S0606}
 We again look at the limit in the
 guise of the epsilon-delta proof.
 \cite{wolfepdel} comments on
 the generalization in multidimensional space
 with the norm and open balls.

For complex proofs, NSA has
 been successful as an
 alternative to  
 Epsilon-Delta management,
 and other proofs solving
 in the higher number system
 and transferring the results
 back into the reals.
 We would expect our calculus
 to also be useful in proving
 propositions and theorems,
 with similar purpose to NSA but
 in another way.

 An epsilon-delta definition proof, 
 if $|x-x_{0}| \lt \delta$
 then $|f(x)-f(x_{0})| \lt \epsilon$, 
 in a minimalistic sense 
 is not
 a finite inequality,
 but an inequality at infinity with infinitesimals,
 as we can derive the finite inequality.
 By the
 transfer principle,
 project the statement from $*G$
 into $\mathbb{R}$.
 E.g.
 $\delta_{1} \in \Phi$;
 $\,\delta_{2} \in \mathbb{R}^{+}$;
 $\,(*G, |x-x_{0}|\delta_{1}) \mapsto (\mathbb{R},|x-x_{0}| \delta_{2})$

 An abstraction, by removing
 the inequality relation, expressing
 the relation as a variable which
 is infinitesimal,
 hence
 a more direct reasoning.
\bigskip
\begin{defy}\label{DEF069} 

The Epsilon-Delta Proof with $\Phi$
\[ \text{If } x-x_{0} = \Phi \text{ then } f(x)-f(x_{0}) = \Phi, \, n=\infty\text{.} \]
\end{defy}
\bigskip
\begin{mex} 
\cite[Epsilon-Delta Proof]{wolfLHopital}
 $f(x) = ax+b$; 
 $a, b \in \mathbb{R}$;
 $a \neq 0$.
 Show $f(x)$ is continuous.
\begin{align*}
 x-x_{0} & = \Phi \\
 f(x)-f(x_{0})
 & = (ax+b) - (a x_{0}+b) \\
 & = a (x-x_{0})  = \Phi \tag{ as $a \Phi = \Phi$ }
\end{align*}
\end{mex}

An Epsilon-Delta definition and proof with a
 similar 
 structure 
  \cite{wolfepdel} could be given where
 the real numbers are replaced by $*G$. For example it 
 may not be enough that the numbers are infinitesimals (Definition \ref{DEF069}),
 but we may require
 the infinitesimals to be continually
 approaching $0$. (See Proposition \ref{P222}) 
\[ \delta_{n} \to 0 \text{ replaced with } \delta_{n} \succ \delta_{n+1} \]
\begin{defy}\label{DEF070}
The Epsilon-Delta Proof in $*G$
\[ \delta_{n} = x-x_{0} \in \Phi; \;\; \epsilon_{n} = f(x)-f(x_{0})\in \Phi ; \;\; n=\infty \]
\[ \text{ If } \delta_{n} \succ \delta_{n+1} \text{ then } \epsilon_{n} \succ \epsilon_{n+1} \] 
\end{defy}
\subsection{A two-tiered calculus}\label{S0605}
 We can work in $*G$
 and project back or transfer to
 $\mathbb{R}$ or $\overline{\mathbb{R}}$, or $*G$.
 The overall reason for doing this was a separation of the
 finite and infinite domains, thereby separating
 and isolating the problem.

 We introduce the following sequence
 definitions as a consequence of
 a two-tiered calculus.
 It is possible
 for a sequence to plateau in $*G$ and
 project back to a convergent sequence 
 in $\mathbb{R}$. 
 Similarly a divergent
 sequence could plateau in $*G$
 and diverge in $\overline{\mathbb{R}}$.
 We need to be able to describe
 arbitrarily converging
 and diverging 
 sequences to guarantee certain properties and avoid the plateau. 
 Hence, additional requirements
 are needed to manage the
 sequences. 
\bigskip
\begin{defy}\label{DEF216}
We say
 $x_{n} \to 0$
 then
 $x_{n} \in \Phi$ and is decreasing
 in magnitude: $|x_{n+1}| \leq |x_{n}|$.
\end{defy}
\bigskip
\begin{defy}\label{DEF217}
We say
 $x_{n} \to \infty$
 then
 $x_{n} \in \Phi^{-1}$ and is increasing 
 in magnitude: $|x_{n+1}| \geq |x_{n}|$.
\end{defy}

 Since a variable may be expressed as a point,
 the sequences described can be extended to 
 the continuous variable. 
 An adaptable notation, given
 that we may need
 different sequences for particular
 problems and theory.

 For example, $x \to \infty$,
 $(x)$
 indefinitely increases
 and is positive monotonic,
 $f(x)|_{x=\infty} = \ldots$

 The other type of sequences
 in general use are 
 a partition. For
 example, for all $x \gt x_{0}$.
\bigskip
\begin{defy}\label{DEF218}
In context, a variable $x \to 0$ can
 be described at zero by
 $x \in \Phi$ or
 Definition \ref{DEF216} or
 Definition \ref{DEF221} or
 other
 as $|_{x=0}$.
\end{defy}
\bigskip
\begin{defy}\label{DEF219}
In context, a variable $x \to \infty$ can
 described at infinity by
 $x \in \Phi^{-1}$ or
 Definition \ref{DEF217} or
 Definition \ref{DEF222} or
 other
 as $|_{x=\infty}$.
\end{defy}
\bigskip
\begin{defy}\label{DEF220}
A `subsequence' is a sequence formed from 
 a given sequence by deleting
 elements without changing the relative
 position of the elements.
\end{defy}
\bigskip
\begin{defy}\label{DEF221}
 We say a sequence is `indefinitely decreasing'
 in magnitude.
 $x_{n} \to 0$; 
 $n$, $n_{2} \in \mathbb{N}_{\infty}$;
 there exists
 $n_{2}: n_{2} \gt n$
 and $x_{n_{2}} \prec x_{n}$.
\end{defy}
\bigskip
\begin{defy}\label{DEF222}
 We say a sequence is `indefinitely increasing'
 in magnitude.
 $x_{n} \to \infty$;
 $n$, $n_{2} \in \mathbb{N_{\infty}}$;
 there exists
 $n_{2}: n_{2} \gt n$
 and $x_{n_{2}} \succ x_{n}$.
\end{defy}
\bigskip
\begin{prop}\label{P221}
If $x_{n} \to 0$ is indefinitely decreasing
 there exists a subsequence $(\nu_{n}): \nu_{n+1} \prec \nu_{n}|_{n=\infty}$ and $\nu_{n} \to 0$
\end{prop}
\begin{proof}
 Since $x_{n}$ is decreasing in magnitude,
 we can always choose a subsequent much-less-than
 term.
\end{proof}

 We use infinity
 arguments with order in
 the proof of
 Proposition \ref{P222}.
 Normally we would send
 $h \to 0$ before $\delta \to 0$.
 However, 
 if the solution is
 independent of the infinity,
 consider $\delta \to 0$ before
 $h \to 0$.
 Then we reason that the
 derivative must be an infinitesimal.
\bigskip
\begin{prop}\label{P222}
If $\delta_{n} \to 0$ is indefinitely decreasing and strictly positive monotonic  decreasing then
\[ D \delta_{n}|_{n=\infty} \in -\Phi \]
\end{prop}
\begin{proof}
 $h \in \!+\Phi;$
 Strictly monotonic decreasing $\delta_{n}$
 then
 $\delta_{n+1} \lt \delta_{n}$,
 $\delta_{n+1} - \delta_{n} \lt 0$,
 $\frac{ \delta_{n+1} - \delta_{n}}{h} \lt 0$,
 $D \delta_{n} = \frac{ \delta_{n+1} - \delta_{n}}{h}|_{h=0}$
 is negative.

 Consider the infinite state where 
 $\delta_{n} \to 0$ before $h \to 0$.
 Since $\delta_{n}$ can be made arbitrarily small,
 then 
 $\delta_{n+1}-\delta_{n} \prec h$,
 $\delta_{n+1}-\delta_{n} \in \Phi$,
 $\frac{ \delta_{n+1} - \delta_{n}}{h}|_{h=0} \in \Phi$,
 $D \delta_{n} \in \Phi$.
\end{proof}
\bigskip
\begin{lem}\label{P223}
If $f(x)$ and $g(x)$ are positive monotonic functions, with relation
 $z: z \in \{ \lt, \leq, \gt, \geq \}$, $f(x) \; z \; g(x)$ for all $x$ in a given
 domain, such a relation can be reformed to a positive inequality:
 $\phi \gt 0$ or $\phi \geq 0$
\end{lem}
\begin{proof}
 A less than relation can always be expressed
 as a greater than relation by swapping the arguments sides.
 If $f \lt g$ then $g \gt f$.
 If $f \leq g$ then $g \geq f$.
 If $f(x) \gt g(x)$ then $f(x)-g(x) \gt 0$.
 If $f(x) \geq g(x)$ then $f(x) - g(x) \geq 0$.
\end{proof}
\bigskip
\begin{prop}\label{P224}
 $z \in \{ \gt, \geq \}$;
 If $f(x) \; z \; 0$  and $f(x)$  is monotonically increasing then 
 $D f(x) \; z \; 0$ where $f(x)$ is not constant.
\end{prop}
\begin{proof}
 $h \in \Phi^{+}$, 
 $f(x+h) \; z \; f(x)$, 
 $f(x+h) - f(x) \; z \; 0$,
 $\frac{f(x+h) - f(x)}{h} \; z \; 0$,
 $\frac{f(x+h) - f(x)}{h}|_{h=0} \; z \; 0$,
 $D f(x)  \; z \; 0$
\end{proof}
\bigskip
\begin{prop}\label{P022}
$z \in \{ \gt, \geq \}$;
 If $f(x)$ and $g(x)$ are positive monotonic functions: $f(x) \; z \; g(x)$
 over a positive domain then
 $D f(x) \; z \; D g(x)$, where $f(x) -g(x)$ is monotonically increasing.
\end{prop}
\begin{proof}
 Reorganise the relation to be positive, Lemma \ref{P223}.
 Apply Proposition \ref{P224}.
\end{proof}
\bigskip
\begin{prop}\label{P024}
 If $f \succ 0$ and $f$ is a positive monotonic
 increasing function then ignoring integration
 constants and integrating in a positive interval,
 $\int f \succ 0$.
\end{prop}
\begin{proof}
 Since integrating an infinitesimal or
 infinity does not result in $0$,
 then integral must have a 
 much-greater-than relationship with 
 $0$.
\end{proof}
\bigskip
\begin{theo}\label{P023}
Let $f$ and $g$ be positive monotonic functions:
 $f \succ g$.
 If integrated over a positive interval ignoring
 integration constants then $\int f \succ \int g$.
\end{theo}
\begin{proof}
 $f \succ g$, 
 $f-g \succ 0$,
 apply Proposition \ref{P024}.
\end{proof}

This infinitesimal and infinitary analysis is more suited
 to a functional approach and does not explicitly use sets,
 compared with NSA.
 Hence the complexity of use would likely make this calculus
 more accessible.
 We have found, empirically, different solutions to problems
 and in many cases simpler reasoning than with standard
 calculus, such as found in \cite{kaczor}.
 That is, we have constructed a new calculus
 of sum convergence \cite[Convergence sums ...]{cebp2}. 

The following propositions define
 continuity and calculus in $*G$.
\bigskip
\begin{prop}\label{P225}
 A function $f: *G \mapsto *G$ is uniformly continuous
  if $f(x) \simeq f(y)$ when $x, y \in *G$
 and $x \simeq y$. 
\end{prop}
\bigskip
\begin{prop}\label{P226}
A function $f: *G \mapsto *G$ 
 is differentiable at $x\in *G$
 iff there exists $b\in *G$:
\[ \frac{f(x)-f(a)}{x-a} \simeq b \text{ when } x \simeq a \]
\end{prop}

 Many of the classical results can be proved in
 $*G$.
 Assuming the Taylor series in $*G$ with arbitrary
 truncation, that is a well-behaved function, prove
 Newton's method.
\bigskip
\begin{theo}\label{P227}
 When $f(x_{n+1}) \prec f(x_{n})|_{n=\infty}$
 and $x_{n+1} \simeq x_{n}|_{n=\infty}$ then
$x_{n+1} = x_{n} - \frac{ f(x_{n}) }{ f'(x_{n}) }|_{n=\infty}$ 
\end{theo}
\begin{proof}$h, f(x_{n}) \in \Phi$;
 $\,x_{n}$, $f^{(w)}(x_{n}) \in *G$
\begin{align*}
 f(x_{n} + h) = f(x_{n}) + h f'(x_{n}) +\frac{h^{2}}{2!} f'^{2}(x_{n})+\ldots|_{n=\infty}|_{h=0} & \tag{Assume continuity} \\
 f(x_{n} + h) = f(x_{n}) + h f'(x_{n})|_{n=\infty}|_{h=0}  & \tag{Choose $h=x_{n+1}-x_{n}$} \\
 f(x_{n+1}) = f(x_{n}) + (x_{n+1}-x_{n}) f'(x_{n})|_{n=\infty}  & \tag{Non-reversible arithmetic} \\
 & \tag{$f(x_{n})-f(x_{n+1}) = f(x_{n})|_{n=\infty}$ as $f(x_{n}) \succ f(x_{n+1})|_{n=\infty}$ } \\
 0 = f(x_{n}) + (x_{n+1}-x_{n}) f'(x_{n})|_{n=\infty}  & \tag{Solve for $x_{n+1}$} \\
 x_{n+1} = x_{n}- \frac{f(x_{n})}{f'(x_{n})}|_{n=\infty} & \tag{Transfer principle $*G \mapsto \mathbb{R}$} 
\end{align*}
\end{proof}
\subsection{A variable reaching infinity before another}\label{S0607}
 With the knowledge of partial derivatives, there should
 be no argument against a variable reaching infinity before
 another. That this is not taught or seen this way,
 once stated should be accepted as fact.
\begin{quote} The partial derivative is equivalent to
 one variable reaching infinity before the other variables.
\end{quote}

 With the finite and infinite separation,
 we will seek further mathematics. 
 The order of one variable reaching infinity before
 another is common as demonstrated by partial differential equations.
 We also have other language to capture the infinite state,
 such as `the characteristic differential equation'.
\begin{quote}
 This raises the possibility of combinations
 of variables reaching infinity in different
 orders. The algebra and space is amazingly complex, yet understandable.
\end{quote}

 Rather than purge mathematics of complexity for certainty, a consideration of 
 orderings we believe has led to new rearrangement theorems and analysis
 (see \cite{cebp2}, \cite{cebp8}). 
 This is a necessary correction to current mathematics 
 which rejects infinitary calculus and Euler in superficial ways.

Uniform convergence, absolute sum
 convergence and other concepts, which
 for example highlight when the problem
 is independent of the order,
 could be investigated with orderings.

 Turning the problem around, if the ordering
 does not matter, we only need to find the
 solution of one ordering to determine
 the whole solution. 
  (for example sum rearrangements at infinity \cite{cebp8})
\bigskip
\begin{mex} See Proposition \ref{P005} \end{mex}
 An example of a variable reaching infinity before
 another explains what others claim is `extraordinary reasoning'.
\bigskip 
\begin{mex}
 In following and explaining Euler's derivation
 of the exponential function
 expansion for $e^{x}$,
 Robert Goldblatt \cite[p.8]{gold} describes Euler's 
 reasoning 
\[ \frac{j(j-1)(j-2)\ldots (j-n+1)}{j^{n}}|_{j=\infty}=1 \]
 as `extraordinary', which, with polite wording,
 is an academics way of describing
 something as humbug. However, the calculation \it{is} explained by one variable
 reaching infinity before the other. 

 Firstly, consider the numerical evidence.
 $\frac{j}{j}|_{j=\infty}=1$;
 $\frac{j(j-1)}{j^{2}}|_{j=\infty}=1$;
 $\frac{j(j-1)(j-2)}{j^{3}}|_{j=\infty}=1$;
 $\ldots$

 If we consider the general $n$ term, we arrive at the expression
 $\frac{j(j-1)(j-2)\ldots (j-n+1)}{j^{n}}|_{j=\infty}$
 If $j=\infty$ before $n=\infty$, then from $j$'s perspective,
 $n$ is a constant. This is no different from the partial derivative
 case. Hence, simplify the constant by non-reversible arithmetic,
 $j(j-1)(j-2)\ldots (j-n+1)|_{j=\infty}$
 $= j^{n}|_{n=\infty}$ and the result follows.

 Consequently, $j \succ n|_{j,n=\infty}$. There is the possibility of developing
 algebra for these situations, to analyse the mathematics.
 
 We should point out that this is not the only possibility at infinity
 (as noted in
 subsequent papers there is algebra where lower order terms prevail at infinity),
 and that non-uniqueness exists at infinity. However,
 it `is' a valid possibility. 
 
 For the given problem, we can state Euler's reasoning in
 $*G$.

 $\omega \in \Phi$; $j = \frac{x}{\omega}$; $x \not \in \mathbb{R}_{\infty}$; $\omega = \frac{x}{j}$
 and $j \in \Phi^{-1}$; consider the Binomial expansion in $*G$, 
 which is most likely acceptable as we believe $*G$ is 
 a field.
\begin{align*}
 (1+k\omega)^{j} = 1 + j \frac{k \omega}{1!} + j(j-1) \frac{ (k \omega)^{2}}{2!} + \ldots \\
 = 1+\sum_{n=1}^{\infty} j(j-1)(j-2)\ldots(j-n+1) \frac{ k^{n} \omega^{n}}{n!}|_{j=\infty} \\
 = 1+\sum_{n=1}^{\infty} \frac{ j(j-1)(j-2)\ldots(j-n+1) }{j^{n}} \frac{ k^{n} x^{n}}{n!}|_{j=\infty} \\
 = 1+\sum_{n=1}^{\infty} \frac{ k^{n} x^{n}}{n!}|_{j=\infty} 
\end{align*}
 Hence, in this case Euler's logic is justified and
 explained in a more rigorous way, in the exact way Euler stated.
 Euler was not in error, but exactingly correct.

 Choosing $k=1$
 \[ (1+\omega)^{\frac{x}{\omega}} = \sum_{k=0}^{\infty} \frac{x^{n}}{n!} \]
 If $x=1$ then $(1+\omega)^{\frac{1}{\omega}}=e$.
 $(1+\omega)^{\frac{x}{\omega}}$
 $= ((1+\omega)^{\frac{1}{\omega}})^{x}$
 $=e^{x}$ then $e^{x} = \sum_{k=0}^{\infty} \frac{x^{n}}{n!}$.
\end{mex}

{\em RMIT University, GPO Box 2467V, Melbourne, Victoria 3001, Australia}\\
{\em chelton.evans@rmit.edu.au}

\begin{thebibliography}{99}\setlength{\itemsep}{-2mm}
\bibitem{history} \emph{The Continuum and the Infinitesimal in the Ancient Period}, \url{http://plato.stanford.edu/entries/continuity/#2}.
\bibitem{historyzeno} \emph{Zeno's Paradoxes}, \url{http://plato.stanford.edu/entries/paradox-zeno/}.
\bibitem{historyzenoachilles} \emph{Achilles and the Tortoise}, \url{http://plato.stanford.edu/entries/paradox-zeno/#AchTor}.
\bibitem{littleobigO} \emph{Little o's and Big O's}, \url{http://www.math.caltech.edu/~2010-11/1term/ma001a1/bigolittleo.pdf}.
\bibitem{apostol} T. M. Apostol, \emph{Calculus}, Vol. 1, Second Edition, Xerox College Publishing.
\bibitem{hille}E. Hille, \emph{Analytic Function Theory}, Chelsea publishing company, 1959.
\bibitem{ordersofinfinity}G. H. Hardy, \emph{Orders of infinity}, Cambridge Univ. Press, 1910. 
\bibitem{kaczor}W. J. Kaczor, M. T. Nowak, \emph{Real Numbers, Sequences and Series}, 2000
\bibitem{fisher} G. Fisher, \emph{The infinite and infinitesimal quantities of du Bois-Reymond and their reception} , 27. VIII. 1981, Volume 24, Issue 2, pp 101--163, Archive for History of Exact Sciences.
\bibitem{flatland} Edwin A. Abbott, \emph{Flatland}, \url{https://en.wikipedia.org/wiki/Flatland}
\bibitem{cebp10} C. D. Evans, W. K. Pattinson, \emph{The Fundamental Theorem of Calculus with Gossamer numbers}
\bibitem{cebp2} C. D. Evans, W. K. Pattinson, \emph{Convergence sums at infinity a new convergence criterion}
\bibitem{cebp4} C. D. Evans, W. K. Pattinson, \emph{Convergence sums and the derivative of a sequence at infinity}.
\bibitem{cebp8} C. D. Evans and W. K. Pattinson, \emph{Rearrangements of convergence sums at infinity}
\bibitem{alice} M. Gardner, \emph{The Annotated Alice. Lewis Carroll}, Penguin Books, 1970.
\bibitem{abraham} A. Robinson, \emph{Non-standard Analysis}, Princeton University Press, 1996
\bibitem{victors} J. Bair P. Blaszczyk, R. Ely, V. Henry, V. Kanoevei, K. Katz, M. Katz, S. Kutateladze, T. Mcgaffy, D. Schaps, D. Sherry and S. Shnider, \emph{IS MATHEMATICAL HISTORY WRITTEN BY THE VICTORS?}, see http://arxiv.org/abs/1306.5973 
\bibitem{orderedfield} \emph{Ordered field}, \url{http://nosmut.com/Ordered_field.html}.
\bibitem{pointsett} J. H. Manheim, \emph{The Genesis of Point Set Topology}, Pergamon Press, 1964.
\bibitem{wolfLHopital} \emph{L'Hopital's Rule}, \url{http://mathworld.wolfram.com/LHopitalsRule.html}.
\bibitem{wolfepdel} \emph{Epsilon-Delta proof}, \url{http://mathworld.wolfram.com/Epsilon-DeltaProof.html}.
\bibitem{nigelc}N. Cutland, \emph{Nonstandard Analysis and its Applications}, Cambridge University Press, 1988.
\bibitem{graham} R. L. Graham, D. E. Knuth, O. Patashnik, {\em Concrete Mathematics}, Second Edition, Addison-Wesly, Boston, 1989.
\bibitem{maxima47} \emph{root 2 calculation}, \url{http://www.fluxionsdividebyzero.com/Downloads/2014-11-15/doc000000622.wxm}
\bibitem{hopital2} \emph{How did Bernoulli prove L'Hôpital's rule?}, \url{http://mathoverflow.net/questions/51685/how-did-bernoulli-prove-lh\%C3\%B4pitals-rule}
\bibitem{tprinciple} \emph{Wikipedia: transfer principle}, \url{http://wikipedia.org/wiki/Transfer_principle}
\bibitem{gold} Robert Goldblatt, \emph{Lectures on the Hyperreals An Introduction to Nonstandard Analysis}, Springer, 1991.
\bibitem{hypercard} \emph{Cardinality of the hyperreal numbers}, \url{http://math.stackexchange.com/questions/54059/cardinality-of-the-set-of-hyperreal-numbers}.
\bibitem{leibnizecont} Katz, M.; Sherry, D., \emph{Leibniz’s laws of continuity and homogeneity}, Notices of the American Mathematical Society 59 (2012), no. 11, 1550--1558, see http://arxiv.org/abs/1211.7188
\bibitem{fermat} Eli Maor, \emph{e The Story of a Number}, Princeton University Press, 1994.
\bibitem{godel} \emph{Wikipedia: G\"odel's incompleteness theorems}, \url{http://wikipedia.org/wiki/G\%C3\%B6del's_incompleteness_theorems}
\end{thebibliography}
\end{document}